 \documentclass[10pt,twoside,reqno]{amsart}
 \usepackage[top=1in, bottom=1in, left=0.8in, right=0.8in]{geometry}
\usepackage{graphicx}
\usepackage{amsmath,latexsym,amssymb}
\usepackage{amsfonts}
\usepackage{mathtools,esint}
\usepackage{amsmath}
\usepackage{amscd,amsthm}
\usepackage{stmaryrd}
\usepackage{fancyhdr}
\usepackage{commath,enumerate}
\usepackage{indentfirst, hyperref}
\usepackage{caption, subcaption}
\usepackage{dsfont}
\usepackage{bbold}
\usepackage{pgf,tikz}
\usepackage{xcolor}
\usepackage{mathrsfs}
\usetikzlibrary{arrows,calc,positioning}
\usepackage{float}
\numberwithin{figure}{section}
\usepackage[bottom]{footmisc}
\usepackage{datetime}

\usepackage[numbers,sort&compress]{natbib}
\usepackage{cancel}

\usepackage{mathtools}

\numberwithin{equation}{section}
\newtheorem{theorem}{Theorem}[section]
\newtheorem{lemma}[theorem]{Lemma}
\newtheorem{proposition}[theorem]{Proposition}

\theoremstyle{definition}

\newcommand{\Rm}[1]{
  \textup{\uppercase\expandafter{\romannumeral#1}}
}

\newcommand{\A}{\mathbf{A}}
\newcommand{\Ac}{\mathcal{A}}
\newcommand{\B}{\mathcal{B}}
\newcommand{\F}{\mathcal{F}}
\newcommand{\C}{\mathcal{C}}

\newcommand{\ub}{\mathbf{u}}
\newcommand{\Lc}{\mathcal{L}}

\newcommand{\cutoffxi}{\mathfrak{d}}
\newcommand{\med}{\mathrm{med}}
\newcommand{\cutoff}{\mathfrak{b}}

\newcommand{\lb}{\mathbf{l}}
\newcommand{\n}{\mathbf{n}}
\newcommand{\Nc}{\mathcal{N}}

\renewcommand{\P}{\mathcal{P}}

\newcommand{\R}{\mathbb{R}}

\newcommand{\Tb}{\mathbf{T}}
\newcommand{\Tf}{\mathfrak{T}}
\renewcommand{\u}{\mathbf{u}}
\newcommand{\x}{\mathbf{x}}
\newcommand{\Z}{\mathbb{Z}}

\newcommand{\xib}{\boldsymbol{\xi}}

\def\vp{\varphi}
\def\ve{\varepsilon}
\def\px{\partial_x}
\def\ds{\displaystyle}

\newcommand{\diff}{\,\mathrm{d}}

\allowdisplaybreaks[4]

\usepackage{mathtools}
\newcommand{\Lbr}{\mathrlap{[}{\hspace{0.15em}[\hspace{0.15em}}}
\newcommand{\Rbr}{\hspace{0.15em}\mathrlap{]}{\hspace{0.15em}]}}

\let\oldtocsection=\tocsection

\let\oldtocsubsection=\tocsubsection

\let\oldtocsubsubsection=\tocsubsubsection

\renewcommand{\tocsection}[2]{\hspace{0em}\oldtocsection{#1}{#2}}
\renewcommand{\tocsubsection}[2]{\hspace{1em}\oldtocsubsection{#1}{#2}}
\renewcommand{\tocsubsubsection}[2]{\hspace{2em}\oldtocsubsubsection{#1}{#2}}

\author{Fangchi Yan}
\address{Department of Mathematics, West Virginia University}
\email{fangchi.yan@mail.wvu.edu}
\author{Qingtian Zhang}
\address{Department of Mathematics, West Virginia University}
\email{qingtian.zhang@mail.wvu.edu}
\title[Global Solutions of QGSW Front]{Global solutions of quasi-geostrophic shallow-water fronts}
\date{\today.}

\begin{document}

\begin{abstract}
In this paper, we consider a family of piecewise constant solutions of the quasi-geostrophic shallow-water (QGSW) equation. We derive the contour dynamics equation of the QGSW front and prove the global existence of the solutions when the initial data is sufficiently close to a flat front.
\end{abstract}

\maketitle

\tableofcontents

\section{Introduction}

The quasi-geostrophic shallow-water (QGSW) equation \cite{CaCr13} is a transport equation in two space dimensions for an active scalar $q(t, \x)$,  $\x = (x, y)$,
\begin{equation}
\left\{
\begin{aligned}
\label{QGSW}
\begin{split}
& q_t + \u \cdot \nabla q = 0,\\
&\u=\nabla^\perp \Psi,\\
&-\Delta\Psi+F\Psi= q,
\end{split}
\end{aligned}\right.
\end{equation}
where $q$ is the potential vorticity, $\u$ is the velocity, $\Psi$ is the stream function, and $F$ is a positive constant. 
This equation describes the dynamics of a single fluid layer of ocean or atmosphere in the mid-latitude region. 
On the earth, the horizontal scale of the ocean or the atmosphere is about the order of thousands kilometers, which is much larger than the vertical scale, which is around $10$km. The time scale for fluid motion is same or longer than the scale of the earth rotation. In this situation, shallow water approximation can provide a reasonable reduction. This equation is also known as the Charney equation \cite{Charney48} in some geophysical context \cite{Ped87}, and its mathematical justification can be found in Majda’s book \cite{Maj03}. Another interesting point is that the same equation is also derived from the plasma physics \cite{HM78} and it is called Hasegawa-Mima equation in the plasma context, for example \cite{CNQ15}. 

The mathematical structure of QGSW equation is close to incompressible Euler equation, which is the case when $F=0$. Its local and global wellposedness of weak and strong solutions in $\R^2$ has been proved in \cite{Pau04, GH04} in the suitable Sobolev spaces. In these results, the vanishing far-field limit of potential velocity $q$ is a requirement. This contrasts the result we will present below in this paper, where the far-field limit of $q$ is nonzero. 

The main content of the current paper is to study a family of piecewise constant solutions for QGSW equation. This kind of solutions has been studied for several models related to 2D incompressible fluid. These models share the similar form of governing equations:
\begin{equation}\label{vorEq}
\left\{~\begin{aligned}
&\theta_t+\mathbf{u}\cdot\nabla\theta=0,\\
& \mathbf{u}=\nabla^\perp \Psi,\\
& \mathbf{A}\Psi=\theta,
 \end{aligned}\right.
\end{equation} 
where $\theta$ is the (potential) vorticity, $\ub$ is the velocity, and $\Psi$ is the stream function. In the third equation, $\A$ is an elliptic operator connecting the (potential) vorticity and the stream function. 
\begin{itemize}
\item When $\A=-\Delta$, it is the incompressible Euler equation. \item When $\A=(-\Delta)^{1/2}$, it is the surface quasi-geostrophic (SQG) equation. 
\item When $\A=(-\Delta)^{\alpha/2}$ with $\alpha\in(0,2]$, it is the generalized surface quasi-geostrophic (GSQG) equation. 
\item When $\A=F-\Delta$, where $F$ is a positive constant, it is the quasi-geostrophic shallow water (QGSW) equation. 
\end{itemize}
In all these models, the (potential) vorticity satisfies a transport equation. Therefore when the initial data is a piecewise constant function, it will remain so for all time. This is the main reason that we can study the piecewise constant solutions for these equations. When we study these solutions, we usually treat them as a free boundary problem for the curves of the vorticity discontinuities transporting along the fluid flows. 
According to the different geometries of these curves, we classify these problems as the vortex patch problem and the vortex front problem separately. We use the definitions in \cite{HSZ20b} to define the patch solution and the front solution.

Consider the piecewise constant solutions of \eqref{vorEq}
\begin{align}
\label{piecewisetheta}
\theta(t, \x) = \sum_{k = 1}^N \theta_k \mathds{1}_{\Omega_k(t)}(\x),
\end{align}
where $N \geq 2$ is a positive integer, $\theta_1, \dotsc, \theta_N\in \R$ are constants, and $\Omega_1(t),\dotsc, \Omega_N(t) \subset \R^2$ are disjoint domains such that
\[
\bigcup_{k = 1}^N \overline{\Omega_k(t)} = \R^2,
\]
and their boundaries $\partial \Omega_1(t), \dots, \partial \Omega_N(t)$ are smooth curves, whose components either coincide or are a positive distance apart. In \eqref{piecewisetheta}, $\mathds{1}_{\Omega_k(t)}$ denotes the indicator function of $\Omega_k(t)$.\\
\eqref{piecewisetheta} is a {\it patch solution} if it satisfies the following assumptions:
\begin{enumerate}
\item $N \geq 2$;
\item $\theta_N = 0$, but $\theta_k \in \R \setminus \{0\}$ for each $1 \leq k \leq N - 1$;
\item for each $1 \leq k \leq N - 1$, the region $\Omega_k(t)$ is bounded, and its boundary $\partial \Omega_k(t)$
is a smooth, simple, closed curve that is diffeomorphic to the circle $\mathbb{T}$;
\item the region $\Omega_N(t)$ is unbounded.
\end{enumerate}
\eqref{piecewisetheta} is an {\it (N-1)-front solution} if it satisfies the following assumptions:
\begin{enumerate}
\item $N \geq 2$;
\item $\theta_1,\cdots, \theta_N \in \R$ are constants such that $\theta_i\neq \theta_{i+1}$ for $i=1,\cdots, N-1$;
\item Each $\Omega_i(t)$, $i=1,\cdots, N,$ is unbounded and each of their boundaries is a simple, smooth curve diffeomorphic to $\R$.
\end{enumerate}

{\noindent\bf Vortex patch problems.}
For incompressible Euler equation, the contour dynamics equation is derived by Zabusky et.~al.~\cite{Zabusky}.  The boundary of a vortex patch remains smooth globally in time \cite{BC93, Che93, Che98}. Some special types of nontrivial global-in-time smooth vortex patch solutions are constructed in \cite{Bur82, CCGS16b, dlHHMV16, HM16, HM17, HMV13}.
For SQG and generalized SQG equation, local well-posedness of the contour dynamics equations is proved in \cite{CCG18, Gan08, GP21, CCCGW12}. The question of whether finite-time singularities can form in smooth boundaries of SQG or GSQG patches remains open. 
Numerical simulation suggests that the boundary of the SQG patch may develop self-similar singularities and provides evidence that two separated SQG patches can touch in finite time \cite{SD14, SD19}.
It is proved in \cite{GS14} that splash singularities cannot form. Some  nontrivial global solutions for SQG and GSQG patches have been constructed in \cite{CCGS16a, dlHHH16, GS19, HH15, HM17, GPSY19p}. Local existence of the vortex patch problem for singular SQG equation is studied in \cite{KR20a, KR20b}.
When a rigid boundary is present, the local existence of smooth GSQG patches is shown in \cite{GP21, KRYZ16, KYZ17, JZ21p}, and the formation of finite-time singularities is proved for a range of $\alpha$ close to $2$. By contrast, vortex patches in this setting (with $\alpha=2$) have global regularity \cite{KRYZ16}.
For QGSW equation, the vortex patch problem has been studied in \cite{HR21, JD20, PD12, PD13, DHC19}.
 
\vspace{0.2cm}

{\noindent\bf Vortex front problems.} 
For incompressible Euler equation, the vortex front problem is studied in \cite{BH10, Ray1895}, and the approximate models of the Euler front is justified in \cite{HSZ22}.
For GSQG equation with $\alpha\in[1,2]$, the vortex front equations are derived by a regularization procedure in \cite{HS18}. The SQG front equation is derived by a perturbation method in \cite{HSZ20a}.
For SQG equation, local existence and uniqueness for spatially periodic fronts are proved for $C^\infty$ initial data in \cite{Rod05} and analytic initial data in \cite{FR11}. Local well-posedness in Sobolev spaces for spatially periodic solutions of a cubically nonlinear approximation of the SQG front equation is proved in \cite{HSZ18}. Almost sharp SQG fronts are studied in \cite{CFR04, FLR12, FR12, FR15}, and smooth $C^\infty$ solutions for spatially periodic GSQG fronts with $1 < \alpha < 2$ also exist locally in time \cite{CFMR05}.

In the non-periodic setting, smooth solutions to the GSQG front equations with $0 < \alpha < 1$ on $\R$ are shown to exist globally in time for small initial data in \cite{CGI19}.
The analogous result is proved for the SQG front equation with $\alpha = 1$ in \cite{HSZ21}, and for GSQG fronts with $1<\alpha<2$ in \cite{HSZ20p}. The two-front problem of GSQG equation for different values of $\alpha\in(0,2]$ is studied in \cite{HSZ20b} and local well-posedness is proved.




\vspace{0.2cm}

{\noindent\bf QGSW front problem.}
In the current paper, we will derive the equation of the QGSW front. We consider a family of piecewise constant solutions of the form
\begin{equation}
q(t, \x) = \begin{cases} q_+ & \text{if $y >\varphi(t,x)$},\\q_- &\text{if $y<\varphi(t,x)$},\end{cases}
\label{front}
\end{equation}
where $q_+$ and $q_-$ are two constants. The set of discontinuities of $q$ is $\{(x,y)\mid y=\vp(t,x)\}$, which is called the QGSW front. Without loss of generality, we set $F=1$. We will show that the QGSW front satisfies the equation
\begin{equation}\label{QGSWfront}
 \vp_t(t,x)=\frac1\pi\int_{\R}(\vp_x(t,x)-\vp_{x}(t, x+\zeta))K_0\left(\sqrt{\zeta^2+(\vp(t, x)-\vp(t, x+\zeta))^2}\right)\diff \zeta,
\end{equation}
where $K_0(\cdot)$ is the modified Bessel function of the second kind. After simplifying this equation, we will see that this is a nonlinear, nonlocal dispersive equation. We will mainly study the local and global existence of solutions.

Since the equation \eqref{QGSWfront} is a quasilinear dispersive equation, the local existence and uniqueness of solutions in Sobolev space follow from Kato's theory \cite{Kat75a, Kat75b}. The idea is to use the linearized equation to construct a sequence of approximate solutions. Then using the uniform bound of the sequence in the higher regularity space and the contraction in the lower regularity space to obtain a convergent limit. This procedure is similar as the SQG front in \cite{HSZ21}. 

When proving global existence, we need the dispersive estimates. 
This can be achieved by the method of space-time resonances introduced by Germain, Masmoudi and Shatah \cite{Germain,GMS09, GMS12}, together with estimates for weighted $L^\infty_\xi$-norms --- the so-called $Z$-norms --- developed by Ionescu and his collaborators \cite{CGI19, DIP17, DIPP17, IP12, IP13, IP14, IP15, IPu16}.
One may have two additional difficulties compared with the previous results of the GSQG front problem \cite{HSZ21, CGI19, HSZ20p}. 
First, the dispersive relation $p(\xi)=-\xi(1+\xi^2)^{-1/2}$ satisfies $p''(0)=0$ and $p'''(0)\neq 0$, which implies that the frequency $\xi=0$ is a stationary phase and it will leads to slower dispersive decay rate, approximately $O(t^{-1/3})$, when doing oscillatory integral estimate \cite{Ste93}. To obtain $O(t^{-1/2})$ decay, we require some extra derivative in the lower frequency.  Fortunately, the structure of the nonlinear term can provide the extra derivative in the low frequency. This helps to gain $O(t^{-1/2})$ decay.

Another difficulty is that the dispersive relation $p(\xi)=-\xi(1+\xi^2)^{-1/2}$ does not satisfy scaling property. So there is no scaling vector field to use. To solve this problem, we need to introduce the profile function $\hat h(\xi)=e^{-ip(\xi)t}\hat \vp(\xi)$, and estimate $\partial_\xi\hat h$ by using the profile equation. In this procedure, there may be loss of derivatives when estimating the nonlinear terms. It can be solved by symmetrizing the nonlinearities and using the cancellation property of the symbol.

The paper is organized as follows. In Section \ref{sec:prelim}, we list some facts and notations to be used in the later sections. In Section \ref{sec:Der}, we derive the front equation. In Section \ref{sec:Loc}, we prove the local well-posedness of the front equation. After that, we consider the global solutions in Sections \ref{sec:Glo} to \ref{sec:Znorm}. The global existence theorem is stated in Theorem \ref{Thm:glo}. Sections \ref{sec:Dis}, \ref{sec:WE}, \ref{sec:Znorm} are mainly for dispersive estimates, weighted energy estimates and the space-time resonance analysis, separately.

\section{Preliminaries}
\label{sec:prelim}
We denote the Fourier transform of $f \colon \R\to \C$ by $\hat f \colon \R\to \C$, where $\hat f= \F f$ is given by
\[
f(x)=\int_{\R} \hat f(\xi) e^{i\xi x} \diff\xi,  \qquad \hat f(\xi)=\frac1{2\pi} \int_{\R}f(x) e^{-i\xi x}\diff{x}.
{\color{red}}
\]
For $s\in \R$, we denote by $H^s(\R)$ the space of Schwartz distributions $f$ with $\|f\|_{H^s} < \infty$, where
\[
\|f\|_{H^s} = \left[\int_\R \left(1+|\xi|^2\right)^s |\hat{f}(\xi)|^2\, \diff{\xi}\right]^{1/2}.
\]
Throughout this paper, we use $A\lesssim B$ to mean there is a constant $C$ such that $A\leq C B$, and $A\gtrsim B$ to mean there is a constant $C$ such that $A\geq C B$. We use $A\approx B$ to mean that $A\lesssim B$ and $B\lesssim A$.

Let $\psi \colon\R\to [0,1]$ be a smooth function supported in $[-8/5, 8/5]$ and equal to $1$ in $[-5/4, 5/4]$.
For any $k\in \mathbb Z$, we define
\begin{align}
\label{defpsik}
\begin{split}
\psi_k(\xi)&=\psi(\xi/2^k)-\psi(\xi/2^{k-1}), \qquad \psi_{\leq k}(\xi)=\psi(\xi/2^k),\qquad \psi_{\geq k}(\xi)=1-\psi(\xi/2^{k-1}),\\
\tilde\psi_k(\xi)&=\psi_{k-1}(\xi)+\psi_k(\xi)+\psi_{k+1}(\xi).
\end{split}
\end{align}
So we have the homogeneous dyadic decomposition
\[
1=\sum\limits_{k\in\mathbb Z}\psi_k(\xi),
\]
and the non-homogeneous dyadic decomposition
\[
1=\sum\limits_{k\in\mathbb N}\psi_k(\xi)+\psi_{\leq 0}(\xi).
\]
We denote the sub-index of the non-homogeneous dyadic decomposition as $\{\leq 0\}\cup \mathbb N$.

We denote by $P_k$, $P_{\leq k}$, $P_{\geq k}$, and $\tilde{P}_k$  the Fourier multiplier operators with symbols $\psi_k, \psi_{\leq k}, \psi_{\geq k}$, and $\tilde{\psi}_k$, respectively. Notice that $\psi_k(\xi)=\psi_0(\xi/2^k)$, $\tilde\psi_k(\xi)=\tilde\psi_0(\xi/2^k)$.

It is easy to check that
\begin{equation}
\label{psi-L2}
 \|\psi_k\|_{L^2}\approx  2^{k/2}, \qquad  \|\psi_k'\|_{L^2}\approx 2^{-k/2}.
\end{equation}

We will need the following interpolation lemma, whose proof can be found in \cite{IPu16}.
\begin{lemma}\label{interpolation}
For any $k\in\mathbb Z$ and $f\in L^2(\R)$, we have
\[
\|\widehat{P_kf}\|_{L^\infty}^2\lesssim \|P_k f\|_{L^1}^2\lesssim 2^{-k}\|\hat f\|_{L^2_\xi}\left[2^k\|\partial_\xi\hat f\|_{L^2_\xi}+\|\hat f\|_{L^2_\xi}\right].
\]
\end{lemma}

Define the symbol class
\[
S^\infty:=\{m: \R^d\to\C: m \text{ is continuous and }  \|m\|_{S^\infty}:=\|\F^{-1}(m)\|_{L^1}<\infty\}.
\]

\begin{lemma}[Algebraic property of $S^\infty$]
If $m$, $m'\in S^\infty$, then $m\cdot m'\in S^\infty$.
\end{lemma}

Given $m \in S^\infty(\R^{n})$, we define a multilinear operator $M_m$ acting on Schwartz functions  $f_1,\dots, f_m \in \mathcal{S}(\R)$ by
\[
M_{m}(f_1,\dotsc,f_n)(x)=\int_{\R^{n}} e^{ix(\xi_1+\dotsb+\xi_n)}m(\xi_1, \dotsc, \xi_n)\hat f_1(\xi_1) \dotsm \hat f_n(\xi_n)\diff{\xi_1} \dotsm \diff{\xi_n}.
\]

\begin{lemma}[Estimates of multilinear Fourier integral operators]\label{multilinear}
(i)~ Suppose that $1<p_1, \dotsc, p_n\leq \infty$, $0<p<\infty$,  satisfy
\[
\frac1{p_1}+\frac1{p_2}+ \dotsb +\frac1{p_n}=\frac1p,
\]
for every nonempty subset $J\subset \{1,2, \dotsc, n\}$. If $m \in S^\infty(\R^{n})$, then
\[
\|M_{m}\|_{L^{p_1}\times \dotsb \times L^{p_n}\to L^p}\lesssim \|m\|_{S^\infty}.
\]
(ii)~ Assume $p,q,r\in[1,\infty]$ satisfy $\frac1p+\frac1q=\frac1r$, and $m\in S^\infty$. Then for any $f,g\in L^2(\R)$, 
\[
\left|\int_{\R^2}m(\xi,\eta)\hat f(\xi)\hat g(\eta)\hat h(-\xi-\eta)\diff \xi\diff\eta\right|\lesssim \|m\|_{S^\infty}\|f\|_{L^p}\|g\|_{L^q}\|h\|_{L^r}.
\]
\end{lemma}

We define the $B^{a,b}$ semi-norm as 
\begin{equation}\label{bnorm}
\|f\|_{B^{a,b}}=\sum\limits_{j\in\Z}(2^{aj}+2^{bj})\|P_jf\|_{L^\infty}.
\end{equation}
Here $a<b$ and $a,b$ are the indices of higher and lower frequencies. Compared with the $L^\infty$ norm, $B^{a,b}$ satisfies one property not shared with the $L^\infty$ norm, that is
\[
\|f\|_{B^{a,b}}\approx \sum\limits_{j\in\Z}\|P_jf\|_{B^{a,b}}.
\]

\subsection{Some facts about the function $K_0$}
We list some properties of the function $K_0$ here for convenience. They are from the handbook \cite{OLBC}. 

The modified Bessel function of the second kind $K_0(x)$ is one solution of the equation
\begin{equation}\label{BesselEq}
xy''+ y'-xy=0,
\end{equation}
which satisfies
\[
K_0(x)\sim \sqrt{\frac{\pi}{2x}}e^{-x}, \quad \text{ when } x\to \infty.
\]
It satisfies the identities
\[
K_0(x)
=
\frac12
\int_{\R}
\frac{e^{i\xi x}}{\sqrt{1+\xi^2}}
\diff\xi
=
\frac12
\F^{-1}\left(\frac{1}{\sqrt{1+\xi^2}}\right),\quad  \text{and} \quad
K_0(x)=\int_0^\infty e^{-x\cosh t} \diff t.
\]
The power series of $K_0$ around $0$ is
\begin{align}
\label{K0-near0}
K_0(x)
&=-\left[\ln\left(\frac{|x|}2\right)+\gamma\right]I_0(x)+\sum\limits_{k=1}^\infty b_kx^{2k},
\end{align}
where $\gamma=0.577\cdots$ is Euler's constant and
\[
b_k=(\sum\limits_{i=1}^k \frac1i)(k!)^{-2}4^{-k}, \quad I_0(x)=\sum\limits_{k=0}^\infty \frac{1}{(k!)^2 4^k} x^{2k}.
\]
Its derivative satisfies
\[
K_0'(x)=-\int_0^\infty (\cosh t) e^{-x\cosh t} \diff t=-e^{-x}\int_0^\infty (\cosh t) e^{x(1-\cosh t)} \diff t.
\]
When $x>1$, 
\[
0<\int_0^\infty (\cosh t) e^{x(1-\cosh t)} \diff t\leq  \int_0^\infty (\cosh t) e^{(1-\cosh t)} \diff t,
\]
where the integral $\int_0^\infty (\cosh t) e^{(1-\cosh t)} \diff t$ is convergent and it is a constant independent with $x$. Therefore
\[
|K_0'(x)|\lesssim e^{-x}, \quad \text{when } \quad x>1.
\]
Since $K_0$ and $K_0'$ both satisfy the exponential bound, by the equation \eqref{BesselEq}, $K_0''$ also satisfies $|K_0''(x)|\lesssim e^{-x}$ when $x>1$. Taking n-th order derivative to \eqref{BesselEq}, we have
\[
\frac{\diff^n}{\diff x^n}(xy''(x))=\frac{\diff^n}{\diff x^n}[xy(x)]-y^{(n+1)}(x),
\]
which leads to
\[
xy^{(n+2)}(x)=-(n+1)y^{(n+1)}(x)+xy^{(n)}(x)+ny^{(n-1)}(x).
\]
By induction, we conclude
\begin{equation}\label{decayK0}
|K_0^{(n)}(x)|\lesssim O(n!)e^{-x}, \quad \text{when } \quad x>1.
\end{equation}

\section{The QGSW front equation}\label{sec:Der}
In this section we derive the QGSW front equation. Assume the domain
\begin{equation}
\label{defomega}
\Omega_t^\pm = \{(x,y)\in \R^2 : y \gtrless \varphi(t,x)\}
\end{equation}
is an upper/lower half-space whose boundary is the graph of the function $y=\varphi(t,x)$, and
\begin{equation}
q(t, \x) = \begin{cases} q_+ & \text{if $y >\varphi(t,x)$},\\q_- &\text{if $y<\varphi(t,x)$}.\end{cases}
\label{sqg_front}
\end{equation}
Assume the initial data is denoted as
\begin{equation}
q_0(\x) = \begin{cases} q_+ & \text{if $y >\varphi_0(x)$},\\q_- &\text{if $y<\varphi_0(x)$}.\end{cases}
\end{equation}
We take $F=1$. We mainly care about the solution of QGSW equation\eqref{QGSW} when the front $\varphi(t,x)$ is close to the $x$-axis.

\subsection{Steady state solutions}
As a start, we consider a family of solutions which are translation-invariant in $x$-direction. The corresponding front $\varphi(t, x)\equiv 0$. By removing the $x$-derivative in the system, we can rewrite the system as
\begin{equation}
\begin{aligned}
-\partial_y^2\Psi+\Psi=q_\pm, \quad& \text{ in } y\gtrless0,\\
\u_s=(u_s, v_s)^T=(-\partial_y\Psi,0)^T, \quad& \text{ in } \R^2.
\end{aligned}
\end{equation}
By solving the above equation, we have the general solution
\[
\Psi(y)=c_1^\pm e^y+c_2^\pm e^{-y}+q_\pm, \quad \text{ in } y\gtrless0,
\]
where $c_1^\pm$ and $c_2^\pm$ are constants. These four constants are not independent. Across the front, the velocity is continuous. 
So we propose the boundary conditions $\Psi_+=\Psi_-$ and  $\partial_y\Psi_+=\partial_y\Psi_-$ on $y=0$, which lead to
\begin{equation}
\left\{\begin{array}{ll}
c_1^++c_2^++q_+=c_1^-+c_2^-+q_-,\\
c_1^+-c_2^+=c_1^--c_2^-.
\end{array}\right.
\end{equation}
If we assume the velocity is bounded at $\pm\infty$, we can furthermore set $c_1^+=c_2^-=0$. Then $c_1^-=-c_2^+=\frac12 \Lbr q\Rbr$, where $\Lbr q\Rbr:=q_+-q_-$. Thus 
\[
\Psi(y)=\left\{
\begin{array}{ll}
\ds-\frac12\Lbr q\Rbr e^{-y}+q_+,\quad y>0,\cr\cr
\ds\frac12\Lbr q\Rbr e^{y}+q_-,\quad y<0.
\end{array}
\right.
\]
The horizontal velocity is $u_s(y)=-\partial_y\Psi(y)=-\ds\frac12\Lbr q\Rbr e^{-|y|}$.

\subsection{Equation of the front}
We derive the equation of the front when the front is the graph of a function which is close to the steady state $y=0$. 

Denote the Green's function for the elliptic operator $-\Delta +1$ as $G(x,y)$, defined as
\[
G(x,y)=\iint_{\R^2} \frac1{1+\xi^2+\eta^2}e^{i(x\xi+y\eta)}\diff \xi\diff\eta.
\]
To write the function and the integral into the polar coordinate, we denote $\xi=\rho\cos \alpha, \eta=\rho\sin\alpha$, $x=r\cos\theta, y=r\sin\theta$. 
Then 

\begin{align*}
G(r\cos\theta, r\sin\theta)&=\int_{0}^\infty\frac\rho{1+\rho^2}\int_0^{2\pi}e^{i\rho r\cos(\theta-\alpha)}\diff \alpha \diff \rho\\
&=\int_{0}^\infty\frac\rho{1+\rho^2}\int_0^{2\pi}e^{i\rho r\cos \alpha}\diff \alpha \diff \rho\\
&=\int_{0}^\infty\frac\rho{1+\rho^2}\int_0^{2\pi}\cos(\rho r\cos \alpha)\diff \alpha \diff \rho\\
&=2\pi\int_{0}^\infty\frac\rho{1+\rho^2}\left(\frac1{\pi}\int_0^{\pi}\cos(\rho r\cos \alpha)\diff \alpha\right) \diff \rho\\
&=2\pi\int_{0}^\infty\frac\rho{1+\rho^2}J_0(\rho r) \diff \rho=2\pi K_0(r),
\end{align*} 
where $J_0$ is the Bessel function of the first kind.
Notice that $J_0(z)\to 1$ as $z\to 0$, and $J_0(z)=\sqrt{2/(\pi z)}\big(\cos(z-\frac\pi4)+o(1)\big)$ as $z\to\infty$. The above integral converges. In fact, this is the Hankel transform of $1/(1+\rho^2)$, which equals to $K_0(r)$, the modified Bessel function of the second kind.

Using the facts that $K_0(r)\sim \sqrt{\pi/(2r)}e^{-r}$ as $r\to \infty$, and $K_0(r)\sim -\ln r$ as $r\to 0+$, we can obtain
\[
\Psi(x,y,t)=q_-\iint_{y'<\vp(t, x')}G(x-x',y-y')\diff x'\diff y'+q_+\iint_{y'>\vp(t, x')}G(x-x',y-y')\diff x'\diff y',
\]
in which the integral converges.

\noindent Denote the normal vector $\n=\frac{(-\vp_x, 1)}{\sqrt{1+\vp_x^2}}$. The front $\vp(t,x)$ satisfies
\[
\vp_t=\sqrt{1+\vp_x^2}\n\cdot\u=(-\vp_x,1)\cdot\u.
\]
Here the velocity
\[
\u=\nabla^\perp\Psi=\left(\begin{aligned}
q_-\iint_{y'<\vp(t, x')}\partial_{y'}G(x-x',y-y')\diff x'\diff y'+q_+\iint_{y'>\vp(t, x')}\partial_{y'}G(x-x',y-y')\diff x'\diff y'\\
q_-\iint_{y'<\vp(t, x')}-\partial_{x'}G(x-x',y-y')\diff x'\diff y'-q_+\iint_{y'>\vp(t, x')}\partial_{x'}G(x-x',y-y')\diff x'\diff y'
\end{aligned}\right).
\]
By Green's theorem,
\begin{align*}
\u&=\left(\begin{aligned}
q_-\int_{\R}G(x-x',y-\vp(t, x'))\diff x'-q_+\int_{\R}G(x-x',y-\vp(t, x'))\diff x'\\
q_-\int_{\R}\vp_{x'}(t, x')G(x-x',y-\vp(t, x'))\diff x'-q_+\int_{\R}\vp_{x'}(t, x')G(x-x',y-\vp(t, x'))\diff x'
\end{aligned}\right)\\
&=-\Lbr q\Rbr \int_{\R}G(x-x',y-\vp(t, x'))(1, \vp_{x'}(t, x'))^T\diff x'.
\end{align*}
Therefore, the equation of $\vp$ is
\begin{align*}
\vp_t(t,x)&=\Lbr q\Rbr\int_{\R}(\vp_x(t,x)-\vp_{x'}(t, x'))G(x-x', \vp(t, x)-\vp(t, x'))\diff x'\\
&=2\pi\Lbr q\Rbr\int_{\R}(\vp_x(t,x)-\vp_{x'}(t, x'))K_0(\sqrt{(x-x')^2+(\vp(t, x)-\vp(t, x'))^2})\diff x'\\
&=2\pi\Lbr q\Rbr\int_{\R}(\vp_x(t,x)-\vp_{x}(t, x+\zeta))K_0(\sqrt{\zeta^2+(\vp(t, x)-\vp(t, x+\zeta))^2})\diff \zeta.
\end{align*}
The equation is a nonlinear nonlocal equation. Without loss of generality, we assume $\Lbr q\Rbr=1/\pi$ 
 in the following, and the equation is written as
 \[
 \vp_t(t,x)=2\int_{\R}(\vp_x(t,x)-\vp_{x}(t, x+\zeta))K_0(\sqrt{\zeta^2+(\vp(t, x)-\vp(t, x+\zeta))^2})\diff \zeta.
 \]

\subsection{Structure of the equation}
To investigate the structure of the right-hand-side of the above equation, we split it into the linear and nonlinear terms:
\[
\vp_t=\Lc(\vp)+\partial_x\Nc(\vp),
\]
with 
\begin{align*}
\Lc(\vp)&:=2\int_{\R}(\vp_x(t,x)-\vp_{x}(t, x+\zeta))K_0(\zeta)\diff \zeta=c\vp_x(t,x)-\frac1\pi\int_{\R}\vp_{x'}(t, x')K_0({x'-x})\diff x',\\
c&:=2\int_{\R}K_0(r)\diff r=2\pi,
\\
\partial_x\Nc(\vp)&:=
2\int_{\R}(\vp_x(t,x)-\vp_{x}(t, x+\zeta))\left[K_0\left(\sqrt{\zeta^2+(\vp(t, x)-\vp(t, x+\zeta))^2}\right)-K_0(\zeta)\right]\diff \zeta.
\end{align*}
Here we have used the fact that $K_0$ is an even function, so the absolute value in $K_0$ can be removed. Since
\[
K_0(x)
=
\frac{1}{2}\int_{\R}\frac{1}{\sqrt{1+\xi^2}}e^{ix\xi}\diff\xi
=
\frac12
\F^{-1}\left(\frac{1}{\sqrt{1+\xi^2}}\right),
\]
we can rewrite the linear term as
\[
\Lc(\vp)=2\pi\vp_x- \frac{\partial_x}{\sqrt{1-\partial_x^2}}\vp.
\]
$\Lc$ is a linear differential operator with symbol $2\pi i\xi-\frac{i\xi}{\sqrt{1+\xi^2}}$. 
By a coordinate transform from $(t,x)$ to $(t, x-2\pi t)$, one can remove the linear transport term $2\pi \vp_x$.
 In the new coordinate, the equation is written as
\begin{equation}\label{FrEq}
\vp_t=-\frac{\partial_x}{\sqrt{1-\partial_x^2}}\vp+\partial_x\Nc(\vp).
\end{equation}

\subsection{Expansion of the nonlinear term}
For the nonlinear term $\Nc(\vp)$, split it into two terms $\Nc_L(\vp)$ and $\Nc_S(\vp)$ with
\[
\partial_x\Nc_L(\vp)=2\int_{\R}\chi_{\{|\zeta|>1\}}(\vp_x(t,x)-\vp_{x}(t, x+\zeta))\left[K_0\left(\sqrt{\zeta^2+(\vp(t, x)-\vp(t, x+\zeta))^2}\right)-K_0(\zeta)\right]\diff \zeta,
\]
\[
\partial_x\Nc_S(\vp)=2\int_{\R}\chi_{\{|\zeta|<1\}}(\vp_x(t,x)-\vp_{x}(t, x+\zeta))\left[K_0\left(\sqrt{\zeta^2+(\vp(t, x)-\vp(t, x+\zeta))^2}\right)-K_0(\zeta)\right]\diff \zeta,
\]
which correspond to the long-range interaction and the short-range interaction separately. Here $\chi_{\{|\zeta|>1\}}$ and $\chi_{\{|\zeta|<1\}}$ are the smooth cut-off functions of the sets $\{|\zeta|>1\}$ and $\{|\zeta|<1\}$ separately, and $\chi_{|\{\zeta|>1\}}+\chi_{\{|\zeta|<1\}}=1$.

{\bf 1.} By Taylor expansion,
\begin{align*}
&K_0\left(\sqrt{\zeta^2+(\vp(t, x)-\vp(t, x+\zeta))^2}\right)-K_0(\zeta)\\
=~&\sum\limits_{k=1}^\infty \frac1{k!}K_0^{(k)}(\zeta)\left[\sqrt{\zeta^2+(\vp(t, x)-\vp(t, x+\zeta))^2}-|\zeta|\right]^k.
\end{align*}
Since
\begin{align*}
\sqrt{\zeta^2+(\vp(t, x)-\vp(t, x+\zeta))^2}-|\zeta|
&=-|\zeta|\left(1-\sqrt{1+\frac{(\varphi(t,x)-\varphi(t, x+\zeta))^2}{\zeta^2}}\right)\\
&=|\zeta|\sum\limits_{l=1}^\infty \left(\begin{array}{c}\frac12\cr\cr l
\end{array}\right)\left|\frac{\varphi(t,x)-\varphi(t, x+\zeta)}{\zeta}\right|^{2l},
\end{align*}
we can write
\begin{align*}
&\sum\limits_{k=1}^\infty\frac1{k!}K_0^{(k)}(\zeta)\left(\sqrt{\zeta^2+(\vp(t, x)-\vp(t, x+\zeta))^2}-|\zeta|\right)^k\\
=&\sum\limits_{k=1}^\infty\frac1{k!}K_0^{(k)}(\zeta)|\zeta|^k\left(\sum\limits_{l=1}^\infty \left(\begin{array}{c}\frac12\cr\cr l
\end{array}\right)\left|\frac{\varphi(t,x)-\varphi(t, x+\zeta)}{\zeta}\right|^{2l}\right)^k\\
=&\sum\limits_{k=1}^\infty\frac1{k!}K_0^{(k)}(\zeta)|\zeta|^k \left(\sum\limits_{|\vec{i}|=k}^\infty \left(\begin{array}{c} k\\ \vec{i}
\end{array}\right)\prod\limits_{l=1}^\infty \left(\begin{array}{c}\frac12\\ l
\end{array}\right)^{i_l}\left|\frac{\varphi(t,x)-\varphi(t, x+\zeta)}{\zeta}\right|^{2i_ll}\right)\\
=&\sum\limits_{k=1}^\infty\frac1{k!}K_0^{(k)}(\zeta)|\zeta|^k \left(\sum\limits_{|\vec{i}|=k}^\infty \left(\begin{array}{c} k\\ \vec{i}
\end{array}\right)\left[\prod\limits_{l=1}^\infty \left(\begin{array}{c}\frac12\\ l
\end{array}\right)^{i_l}\right]\left|\frac{\varphi(t,x)-\varphi(t, x+\zeta)}{\zeta}\right|^{2\sum\limits_{l=1}^{\infty}i_ll}\right)\\
=&\sum\limits_{\mu=1}^\infty A_{\mu}(\zeta)\left|\varphi(t,x)-\varphi(t, x+\zeta)\right|^{2\mu},
\end{align*}
where
\[
A_{\mu}(\zeta)=\sum\limits_{k=1}^\infty\frac1{k!}K_0^{(k)}(\zeta)|\zeta|^{k-2\mu} \left(\sum\limits_{|\vec{i}|=k, \sum i_ll=\mu}^\infty \left(\begin{array}{c} k\\ \vec{i}
\end{array}\right)\left[\prod\limits_{l=1}^\infty \left(\begin{array}{c}\frac12\\ l
\end{array}\right)^{i_l}\right]\right).
\]
We remark here that $k-2\mu\leq0$, and
\begin{align*}
 \left(\begin{array}{c} k\\ \vec{i}
\end{array}\right)=\frac{k!}{i_1!i_2!\cdots i_l!\cdots},
\end{align*}
\begin{align*}
 \left(\begin{array}{c} \frac12\\ l
\end{array}\right)=\frac{(-1)^{l-1}(2l-3)!!}{2^l l!}\in \left[-\frac18, \frac12\right].
\end{align*}
We observe that when $k>\mu$, there is no vector $\vec{i}$ such that $|\vec{i}|=k$, $\sum i_l l=\mu$. The summation in $A_\mu(\zeta)$ is a finite sum, thus it must be convergent.

Therefore, 
\begin{align*}
K_0\left(\sqrt{\zeta^2+(\vp(t, x)-\vp(t, x+\zeta))^2}\right)-K_0(\zeta)
=\sum\limits_{\mu=1}^\infty A_{\mu}(\zeta)\left|\varphi(t,x)-\varphi(t, x+\zeta)\right|^{2\mu}.
\end{align*}
We notice that by \eqref{decayK0}, $|A_{\mu}(\zeta)|\lesssim e^{-|\zeta|}$ when $|\zeta|>1$. 

Therefore, the long-range interaction term can be written as
\begin{align*}
\partial_x\Nc_L(\vp)=&2\int_{\R}\chi_{\{|\zeta|>1\}}(\vp_x(t,x)-\vp_{x}(t, x+\zeta))\left[K_0\left(\sqrt{\zeta^2+(\vp(t, x)-\vp(t, x+\zeta))^2}\right)-K_0(\zeta)\right]\diff \zeta\\
=&2\int_{\R}\chi_{\{|\zeta|>1\}}(\vp_x(t,x)-\vp_{x}(t, x+\zeta)) \sum\limits_{\mu=1}^\infty   A_\mu(\zeta) \left|\varphi(t,x)-\varphi(t, x+\zeta)\right|^{2\mu}\diff \zeta.
\end{align*}

{\bf 2.} By \eqref{K0-near0}, we can calculate the Taylor expansion 
\begin{align*}
&K_0\left(\sqrt{\zeta^2+(\vp(t, x)-\vp(t, x+\zeta))^2}\right)-K_0(\zeta)\\
=~&\left[\ln\left(\frac{|\zeta|}2\right)+\gamma\right]I_0(\zeta)-\left[\ln\left(\frac{\sqrt{\zeta^2+(\vp(t, x)-\vp(t, x+\zeta))^2}}2\right)+\gamma\right]I_0(
\sqrt{\zeta^2+(\vp(t, x)-\vp(t, x+\zeta))^2})\\
&+\sum\limits_{k=1}^\infty b_k \left\{\left[\zeta^2+(\vp(t, x)-\vp(t, x+\zeta))^2\right]^k-\zeta^{2k}\right\}\\
=~&\gamma \left(I_0(\zeta)-I_0(
\sqrt{\zeta^2+(\vp(t, x)-\vp(t, x+\zeta))^2})\right)+\left[\ln\left(\frac{|\zeta|}2\right)-\ln\left(\frac{\sqrt{\zeta^2+(\vp(t, x)-\vp(t, x+\zeta))^2}}2\right)\right]I_0(\zeta)\\
&-\ln\left(\frac{\sqrt{\zeta^2+(\vp(t, x)-\vp(t, x+\zeta))^2}}2\right)\left[I_0(
\sqrt{\zeta^2+(\vp(t, x)-\vp(t, x+\zeta))^2})-I_0(\zeta)\right]\\
&+\sum\limits_{k=1}^\infty b_k \left\{\left[\zeta^2+(\vp(t, x)-\vp(t, x+\zeta))^2\right]^k-\zeta^{2k}\right\}.
\end{align*}
Taking into the Taylor expansion of $I_0$ and $\log$, we have
\begin{align*}
&K_0\left(\sqrt{\zeta^2+(\vp(t, x)-\vp(t, x+\zeta))^2}\right)-K_0(\zeta)\\
=~&-\left[\gamma+\ln\left(\frac{\sqrt{\zeta^2+(\vp(t, x)-\vp(t, x+\zeta))^2}}2\right)\right] \sum\limits_{k=1}^\infty \frac{1}{(k!)^2 4^k} \left[(\zeta^2+(\vp(t, x)-\vp(t, x+\zeta))^2)^k-\zeta^{2k}\right]\\
&-\frac12I_0(\zeta)\sum\limits_{k=1}^\infty \frac{(-1)^{k-1}}{k}\left(\frac{\vp(t, x+\zeta)-\vp(t,x)}{\zeta}\right)^{2k}+\sum\limits_{k=1}^\infty b_k \left\{\left[\zeta^2+(\vp(t, x)-\vp(t, x+\zeta))^2\right]^k-\zeta^{2k}\right\}\\
=~&-\left[\gamma+\ln\left(\frac{|\zeta|}2\right)+\frac12\sum\limits_{k=1}^\infty \frac{(-1)^{k-1}}{k}\left(\frac{\vp(t, x+\zeta)-\vp(t,x)}{\zeta}\right)^{2k}\right] \sum\limits_{k=1}^\infty \frac{1}{(k!)^2 4^k} \left[(\zeta^2+(\vp(t, x)-\vp(t, x+\zeta))^2)^k-\zeta^{2k}\right]\\
&-\frac12I_0(\zeta)\sum\limits_{k=1}^\infty \frac{(-1)^{k-1}}{k}\left(\frac{\vp(t, x+\zeta)-\vp(t,x)}{\zeta}\right)^{2k}+\sum\limits_{k=1}^\infty b_k \left\{\left[\zeta^2+(\vp(t, x)-\vp(t, x+\zeta))^2\right]^k-\zeta^{2k}\right\}\\
=~&
\sum\limits_{\mu=1}^\infty
B_\mu(\zeta)
(\vp(t,x)-\vp(t, x+\zeta))^{2\mu},
\end{align*}
where
\begin{align*}
B_\mu(\zeta)=&-\left[\gamma+\ln\left(\frac{|\zeta|}{2}\right)\right]
\sum\limits_{k=\mu}^\infty
\frac{1}{(k!)^24^k}
{k\choose \mu}|\zeta|^{2k-2\mu}
-
\frac12\sum\limits_{k=1}^\mu\frac{(-1)^{k-1}}{k}
\cdot
\sum\limits_{j=\mu-k}^\infty \frac{1}{(j!)^2 4^j} 
{j\choose \mu-k}|\zeta|^{2j-2\mu}
\\
&~
-\frac12I_0(\zeta)
\frac{(-1)^{\mu-1}}{\mu}|\zeta|^{-2\mu}
+
\sum\limits_{k=\mu}^\infty b_k 
{k\choose \mu}|\zeta|^{2k-2\mu}.
\end{align*}
We remark that the most singular term in $B_{\mu}(\zeta)$ is $|\zeta|^{-2\mu}$.

Therefore,
\begin{align*}
\partial_x\Nc_S(\vp)=&2\int_{\R}\chi_{\{|\zeta|<1\}}(\vp_x(t,x)-\vp_{x}(t, x+\zeta))\left[K_0\left(\sqrt{\zeta^2+(\vp(t, x)-\vp(t, x+\zeta))^2}\right)-K_0(\zeta)\right]\diff \zeta
\\
=&
2\int_{\R}\chi_{\{|\zeta|<1\}}(\vp_x(t,x)-\vp_{x}(t, x+\zeta)) \sum\limits_{\mu=1}^\infty   B_\mu(\zeta) \left(\varphi(t,x)-\varphi(t, x+\zeta)\right)^{2\mu}\diff \zeta.
\end{align*}

{\bf 3.} Then we write the nonlinear terms into the multi-linear Fourier integral operators. By combining the cubic terms from $\Nc_L$ and $\Nc_S$, we define the cubic nonlinear term as 
\begin{align}\nonumber
\partial_x\Nc_3(\vp)=&2\int_{\R}\chi_{\{|\zeta|>1\}}(\vp_x(t,x)-\vp_{x}(t, x+\zeta))(\vp(t,x)-\vp(t, x+\zeta))^2A_{\mu}(\zeta)\diff \zeta\\\nonumber
&+2\int_{\R}\chi_{\{|\zeta|<1\}}(\vp_x(t,x)-\vp_{x}(t, x+\zeta))(\vp(t, x+\zeta)-\vp(t,x))^2B_\mu(\zeta)\diff\zeta\\\label{N3}
=&\frac13\partial_x\iiint_{\R^3} \Tb_1(\eta_1,\eta_2,\eta_3)e^{i(\eta_1+\eta_2+\eta_3)x}\hat\vp(\eta_1)\hat\vp(\eta_2)\hat\vp(\eta_3)\diff \eta_1\diff \eta_2\diff \eta_3,
\end{align}
where
\begin{equation}\label{T1}
\begin{aligned}
&\Tb_1(\eta_1,\eta_2,\eta_3)=2\int_{\R}\chi_{\{|\zeta|>1\}}\prod_{j=1}^3(1-e^{i\eta_j\zeta})A_{\mu}(\zeta)\diff \zeta
+2\int_{\R}\chi_{\{|\zeta|<1\}} \prod_{j=1}^3(1-e^{i\eta_j\zeta})B_{\mu}(\zeta)\diff \zeta.
\end{aligned}
\end{equation}
We  remark that here and in the following, when there is no ambiguity, we may suppress the variable $t$ in the expression of the multilinear Fourier integral to make the expression shorter.

For the $m$-th ($m=2\mu+1$) degree term, we have
\begin{align*}
\partial_x\Nc_m(\vp)=~&2\int_{\R}\chi_{\{|\zeta|>1\}}(\vp_x(t,x)-\vp_{x}(t, x+\zeta))  \left(\varphi(t,x)-\varphi(t, x+\zeta)\right)^{2\mu} e^{-|\zeta|}A_\mu(\zeta)\diff \zeta
\\
&+2\int_{\R}\chi_{\{|\zeta|<1\}}(\vp_x(t,x)-\vp_{x}(t, x+\zeta))(\vp(t,x)-\vp(t, x+\zeta))^{2\mu} B_{\mu}(\zeta)
d\zeta
\\
=~&\frac1m\partial_x\iiint_{\R^{m}} \Tb_{\mu}(\eta_1,\cdots,\eta_m)e^{i\left(\sum\limits_{j=1}^m\eta_j\right)x}\hat\vp(\eta_1)\cdots\hat\vp(\eta_m)\diff \eta_1\cdots\diff \eta_{m},
\end{align*}
where
\begin{align}\label{Tmu}
\Tb_{\mu}(\eta_1,\cdots,\eta_m)
=2\int_{\R}\chi_{\{|\zeta|>1\}}\prod_{j=1}^m(1-e^{i\eta_j\zeta}) A_\mu(\zeta)\diff \zeta
+2\int_{\R}\chi_{\{|\zeta|<1\}} \prod_{j=1}^m(1-e^{i\eta_j\zeta})B_\mu(\zeta)
\diff\zeta.
\end{align}
Therefore 
\begin{align}\label{nonlinearity}
\Nc(\vp)=\sum\limits_{\mu=1}^\infty \Nc_{2\mu+1}(\vp)=\sum\limits_{\mu=1}^\infty
\frac1{2\mu+1}\iiint_{\R^{{2\mu+1}}} \Tb_{\mu}(\eta_1,\cdots,\eta_{2\mu+1})e^{i\left(\sum\limits_{j=1}^{2\mu+1}\eta_j\right)x}\hat\vp(\eta_1)\cdots\hat\vp(\eta_{2\mu+1})\diff \eta_1\cdots\diff \eta_{{2\mu+1}},
\end{align}
or equivalently,
\[
\widehat{\Nc(\vp)}(\xi)=\sum\limits_{\mu=1}^\infty\frac{1}{2\pi(2\mu+1)}\iiint_{\R^{{2\mu}}} \Tb_{\mu}(\eta_1,\cdots,\eta_{2\mu}, \xi-\eta_1-\cdots-\eta_{2\mu})\hat\vp(\eta_1)\cdots\hat\vp(\eta_{2\mu})\hat\vp(\xi-\eta_1-\cdots-\eta_{2\mu})\diff \eta_1\cdots\diff \eta_{{2\mu}}.
\]
In the following, we also denote $\Nc_{\geq5}(\vp):=\sum\limits_{\mu=2}^\infty \Nc_{2\mu+1}(\vp)$. So we can split the nonlinear term as $\Nc(\vp)=\Nc_3(\vp)+\Nc_{\geq5}(\vp)$.


\section{Local solutions}\label{sec:Loc}
In this section, we prove the local existence and uniqueness of the Cauchy problem
\begin{equation}
\label{front}
\left\{
\begin{aligned}
&\vp_t(t,x)=2\int_{\R}(\vp_x(t,x)-\vp_{x}(t, x+\zeta))K_0(\sqrt{\zeta^2+(\vp(t, x)-\vp(t, x+\zeta))^2})\diff \zeta
{
-2\pi\vp_x(t,x)
}
\\
&\hspace{1.1cm}=-\frac{\partial_x}{\sqrt{1-\partial_x^2}}\vp+\partial_x\Nc(\vp),\\
&\vp(x,0)=\vp_0(x),
\end{aligned}
\right.
\end{equation}
where $\Nc(\vp)=\sum\limits_{\mu=1}^\infty\Nc_{2\mu+1}(\vp)$ is the nonlinear term defined in \eqref{nonlinearity}. 

The strategy of the proof mainly follows Kato's theory of quasilinear equations \cite{Kat75b}. Firstly, we linearize the equation and construct a solution map. Then we prove the map is bounded in a higher regularity space and contraction in a lower regularity space. Therefore, the approximate sequence constructed by this solution map is a Cauchy sequence, thus convergent. Finally, we prove the continuity in time.

By taking the dyadic decomposition,
\begin{align*}
\widehat{\Nc_{2\mu+1}(\vp)}(\xi)=&\frac1{2\pi(2\mu+1)}\sum\limits_{j_1,\cdots,j_{2\mu+1}\in\Z}
\iint_{\R^{{2\mu}}} 
\Tb_{\mu}(\eta_1,\cdots,\eta_{2\mu}, \xi-\eta_1-\cdots-\eta_{2\mu})
\\
&\cdot\hat\vp_{j_1}(\eta_1)\cdots\hat\vp_{j_{2\mu}}(\eta_{2\mu})\hat\vp_{j_{2\mu+1}}(\xi-\eta_1-\cdots-\eta_{2\mu})\diff \eta_1\cdots\diff \eta_{{2\mu}}.
\end{align*}
Since the symbol $\Tb_\mu$ is symmetric with its variables, by a potential change of variables, we can assume $j_1\ge j_2\ge\cdots\ge j_{2\mu+1}$. The summation of the ordered indices is denoted by $\sum\limits_Q$.    So 
\begin{align*}
\widehat{\Nc_{2\mu+1}(\vp)}(\xi)=&\frac1{2\pi(2\mu+1)}\sum_{Q}
\iint_{\R^{{2\mu}}} \Tb_{\mu}(\eta_1,\cdots,\eta_{2\mu}, \xi-\eta_1-\cdots-\eta_{2\mu})
\\
&\cdot\hat\vp_{j_1}(\eta_1)\cdots\hat\vp_{j_{2\mu}}(\eta_{2\mu})\hat\vp_{j_{2\mu+1}}(\xi-\eta_1-\cdots-\eta_{2\mu})\diff \eta_1\cdots\diff \eta_{{2\mu}}.
\end{align*}
Then we linearize the equation as
\begin{align}
\label{eqn4_2}
&\hat\vp_t
+
\frac{i\xi}{\sqrt{1+\xi^2}}\hat\vp
=
\frac{1}{2\pi}
\sum_{\mu=1}^\infty \frac1{2\mu+1}i\xi
\\
&\cdot \sum_Q\iint_{\R^{2\mu}} \Tb_{\mu}(\eta_1,\cdots,\eta_{2\mu}, \xi-\eta_1-\cdots-\eta_{2\mu})\hat \vp_{j_1}(\eta_1)\hat u_{j_2}(\eta_2)\cdots\hat u_{j_{2\mu}}(\eta_{2\mu})\hat{u}_{j_{2\mu+1}}(\xi-\eta_1-\cdots-\eta_{2\mu})\diff \eta_1\cdots\diff \eta_{{2\mu}}.
\nonumber
\end{align}
This equation can be written in an abstract form as 
\[
\vp_t=\Ac(u)\vp,\quad  \vp(0,x)=\vp_0(x)\in H^s(\R),~~s\geq 8,
\]
where $\Ac(u)$ is a first order pseudo-differential operator.
The right-hand-side has one-order derivative loss, which can be controlled by Kato-Ponce type commutator estimate \cite{kato-ponce}. 

Denote the map $\B: L^\infty(0,T; H^s(\R)) \to L^\infty(0,T; H^s(\R)),  u\mapsto \vp$.
In the following, we use a symmetrization argument to show that the map $\B$ is bounded in Sobolev space. Assume $\|\vp_0\|_{H^s}<\bar C$, $\|\vp_0\|_{H^3}<1$. Denote 
\[
X_T=\{u\in L^\infty(0,T; H^s(\R))\mid \|u\|_{L^\infty_TH^s}\leq 2\bar C,   \|u\|_{L^\infty_TH^3}<1\}.
\]
We will prove in the following that there is a positive $T$ such that if $u\in X_T$, the solution of the linearized equation $\vp=\B u$ is also in $X_T$.

\subsection{Sobolev energy estimate}
For the higher order energy estimate, multiply $(1+|\xi|^2)^s\hat \vp(t, -\xi)$ to \eqref{front}, and take integral with $\xi$, 
\begin{align*}
\frac{\diff}{\diff t}
&\int_{\R}\frac12(1+|\xi|^2)^s|\hat\varphi(t, \xi)|^2 \diff \xi
\\
=&
\frac{1}{2\pi}
\sum\limits_{\mu=1}^\infty \frac1{2\mu+1}i\sum_{Q}\iiint_{\R^{{2\mu+1}}}\xi(1+|\xi|^2)^s\Tb_{\mu}(\eta_1,\cdots,\eta_{2\mu}, \xi-\eta_1-\cdots-\eta_{2\mu})\\
&\qquad\qquad\cdot\hat \vp_{j_1}(\eta_1)\hat u_{j_2}(\eta_2)\cdots\hat u_{j_{2\mu}}(\eta_{2\mu})\hat{u}_{j_{2\mu+1}}(\xi-\eta_1-\cdots-\eta_{2\mu})\hat \vp(-\xi)\diff \eta_1\cdots\diff \eta_{{2\mu}}\diff\xi.
\end{align*}
Furthermore, making the change of variables $\eta_{2\mu+1}
=\eta_{2\mu+1}(\eta_1)=\xi-\eta_1-\cdots-\eta_{2\mu}$, we get 
\begin{align*}
&\frac{\diff}{\diff t}
\int_{\R}\frac12(1+|\xi|^2)^s|\hat\varphi(t, \xi)|^2 \diff \xi
\\
=&
\frac{1}{2\pi}
\sum\limits_{\mu=1}^\infty \frac1{2\mu+1}i\sum_{Q}\iiint_{\R^{{2\mu+1}}}\xi(1+|\xi|^2)^s 
\,
\Tb_{\mu}(\xi-\eta_2-\cdots-\eta_{2\mu+1},\eta_{2\mu},\cdots,\eta_{2\mu+1})
\\
&\hat \vp_{j_1}(\xi-\eta_2-\cdots-\eta_{2\mu+1})
\hat u_{j_{2}}(\eta_2)
\cdots \hat u_{j_{2\mu+1}}(\eta_{2\mu+1})\hat{\vp}(-\xi)\diff \eta_2\cdots\diff \eta_{{2\mu+1}}\diff\xi.
\end{align*}
Interchanging the variables $-\xi$ and $\xi-\eta_2-\cdots-\eta_{2\mu+1}$, and then taking the average, we have
\begin{align*}
&
\iiint_{\R^{{2\mu+1}}}\xi(1+|\xi|^2)^s 
\,
\Tb_{\mu}(\xi-\eta_2-\cdots-\eta_{2\mu+1}, \eta_{2}, \cdots,\eta_{2\mu+1})
\\
&\hat \vp_{j_1}(\xi-\eta_2-\cdots-\eta_{2\mu+1})
\hat u_{j_{2}} (\eta_2)
\cdots \hat u_{j_{2\mu+1}}(\eta_{2\mu+1})\hat{\vp}(-\xi)\diff \eta_2\cdots\diff \eta_{{2\mu+1}}\diff\xi
\\
=&
\iiint_{\R^{{2\mu+1}}}-(\xi-\eta_2-\cdots-\eta_{2\mu+1})(1+|\xi-\eta_2-\cdots-\eta_{2\mu+1}|^2)^s
\Tb_{\mu}(-\xi,\eta_2,\cdots, \eta_{2\mu+1})
 \\
&\qquad\hat \vp_{j_1}(-\xi) \hat u_{j_2}(\eta_2)\cdots\hat u_{j_{2\mu+1}}(\eta_{2\mu+1})\hat \vp(\xi-\eta_2-\cdots-\eta_{2\mu+1})\diff \eta_2\cdots\diff \eta_{{2\mu+1}}\diff\xi
\\
=&\frac12\iiint_{\R^{{2\mu+1}}}
\Big[\xi(1+|\xi|^2)^s
\Tb_{\mu}(\xi-\eta_2-\cdots-\eta_{2\mu+1}, \eta_{2}, \cdots,\eta_{2\mu+1})
\psi_{j_{1}}(\xi-\eta_2-\cdots-\eta_{2\mu+1}) 
\\
&\qquad\qquad -(\xi-\eta_2-\cdots-\eta_{2\mu+1})(1+|\xi-\eta_2-\cdots-\eta_{2\mu+1}|^2)^s \Tb_{\mu}(-\xi, \eta_2,\cdots,\eta_{2\mu+1})\psi_{j_{1}}(-\xi)\Big]
\\
&\qquad\cdot\hat u_{j_2}(\eta_2)\cdots\hat u_{j_{2\mu+1}}(\eta_{2\mu+1})\hat\vp(-\xi)\hat \vp(\xi-\eta_2-\cdots-\eta_{2\mu+1})\diff \eta_2\cdots\diff \eta_{{2\mu+1}}\diff\xi.
\end{align*}
We can split the symbol into three parts
\begin{align*}
&\Big[\xi(1+|\xi|^2)^s
\Tb_{\mu}(\xi-\eta_2-\cdots-\eta_{2\mu+1}, \eta_{2}, \cdots,\eta_{2\mu+1})
\psi_{j_{1}}(\xi-\eta_2-\cdots-\eta_{2\mu+1}) 
\\
&\qquad\qquad -(\xi-\eta_2-\cdots-\eta_{2\mu+1})(1+|\xi-\eta_2-\cdots-\eta_{2\mu+1}|^2)^s \Tb_{\mu}(-\xi, \eta_2,\cdots,\eta_{2\mu+1})\psi_{j_{1}}(-\xi)\Big]
\\
=&\Big(\xi(1+|\xi|^2)^s-(\xi-\eta_2-\cdots-\eta_{2\mu+1})(1+|\xi-\eta_2-\cdots-\eta_{2\mu+1}|^2)^s\Big)
\\
&\cdot
\Tb_{\mu}(\xi-\eta_2-\cdots-\eta_{2\mu+1}, \eta_{2}, \cdots,\eta_{2\mu+1})
\psi_{j_{1}}(\xi-\eta_2-\cdots-\eta_{2\mu+1}) 
\\
&+(\xi-\eta_2-\cdots-\eta_{2\mu+1})(1+|\xi-\eta_2-\cdots-\eta_{2\mu+1}|^2)^s
\\
&\cdot\Big(
\Tb_{\mu}(\xi-\eta_2-\cdots-\eta_{2\mu+1}, \eta_{2}, \cdots,\eta_{2\mu+1})
-
\Tb_{\mu}(-\xi, \eta_{2}, \cdots,\eta_{2\mu+1})
\Big)\psi_{j_{1}}(\xi-\eta_2-\cdots-\eta_{2\mu+1})
\\
&+(\xi-\eta_2-\cdots-\eta_{2\mu+1})(1+|\xi-\eta_2-\cdots-\eta_{2\mu+1}|^2)^s
\Tb_{\mu}(-\xi, \eta_{2}, \cdots,\eta_{2\mu+1})
\Big(\psi_{j_{1}}(\xi-\eta_2-\cdots-\eta_{2\mu+1})-\psi_{j_{1}}(-\xi)\Big).
\end{align*}
When writing it as the symbol of a multilinear Fourier integral operator, we use $\eta_{1}$ to replace $\xi-\eta_2-\cdots-\eta_{2\mu+1}$. By Proposition \ref{Prop_symb_Tmu} and the algebraic property of $S^\infty$ norm, 
\begin{align*}
&
\left\|\left(\left(\sum\limits_{i=1}^{2\mu+1}\eta_i\right)\left(1+\left|\sum\limits_{i=1}^{2\mu+1}\eta_i\right|^2\right)^s
-\eta_{1}\left(1+|\eta_{1}|^2\right)^s\right)\cdot\Tb_{\mu}(\eta_1,\cdots,\eta_{2\mu}, \eta_{2\mu+1})\psi_{j_1}(\eta_1)\cdots\psi_{j_{2\mu}}(\eta_{2\mu})\psi_{j_{2\mu+1}}(\eta_{2\mu+1})\right\|_{S^\infty}
\\
\lesssim~& 
2^{2sj_{1}+(j_2+\cdots+j_{2\mu+1})}
(1+2^{2j_{2}})
(1+2^{j_2+\cdots+j_{2\mu+1}})(2^{j_2}+\cdots+2^{j_{2\mu+1}}).
\end{align*}
Since the symbol $\Tb_\mu$ is real, it is an even function with its variables, i.e. $$\Tb_\mu(\eta_1, \cdots, \eta_i,\cdots,\eta_{2\mu+1}) = \Tb_\mu(\eta_1, \cdots, -\eta_i,\cdots, \eta_{2\mu+1}),$$ for any $i=1,\cdots, 2\mu+1$. So by \eqref{Tmu_est5} we have
\begin{align*}
&\left\|\eta_{1}\left(1+|\eta_{1}|^2\right)^s
\cdot
\Big(\Tb_{\mu}(\eta_1, \eta_{2}, \cdots, \eta_{2\mu+1})
-
\Tb_{\mu}(-\sum\limits_{i=1}^{2\mu+1}\eta_i, \eta_{2}, \cdots, \eta_{2\mu+1}) \Big)\psi_{j_1}(\eta_1)\cdots\psi_{j_{2\mu}}(\eta_{2\mu})\psi_{j_{2\mu+1}}(\eta_{2\mu+1})\right\|_{S^\infty}
\\
=&\Big\|\eta_{1}(1+|\eta_{1}|^2)^s\cdot
\Big(
\Tb_{\mu}(\eta_1,\eta_{2}, \cdots, \eta_{2\mu+1})-\Tb_{\mu}(\sum\limits_{i=1}^{2\mu+1}\eta_i, \eta_{2}, \cdots, \eta_{2\mu+1}) \Big)\psi_{j_1}(\eta_1)\cdots\psi_{j_{2\mu}}(\eta_{2\mu})\psi_{j_{2\mu+1}}(\eta_{2\mu+1})\Big\|_{S^\infty}
\\
\lesssim& 
(1+2^{2sj_{1}})(1+2^{3j_{2}})
\prod\limits_{k=2}^{2\mu+1}[(1+2^{j_k})2^{j_k}].
\end{align*}
Also, using $\psi_{j_1}'(\xi)=\frac1{2^{j_1}}\psi'(\xi/2^{j_1})$ and Proposition \ref{Prop_symb_Tmu} we get 
\begin{align*}
&\Big\|\eta_{1}(1+|\eta_{1}|^2)^s
\Tb_{\mu}(-\sum\limits_{i=1}^{2\mu+1}\eta_i,\eta_{2}\cdots,\eta_{2\mu+1}) \psi_{j_2}(\eta_2)\cdots\psi_{j_{2\mu+1}}(\eta_{2\mu+1})\cdot\big(\psi_{j_{1}}(\eta_{1})-\psi_{j_{1}}(-\sum\limits_{i=1}^{2\mu+1}\eta_i)\big)\Big\|_{S^\infty}
\\
\lesssim & 2^{2sj_{1}+(j_2+\cdots+j_{2\mu+1})}(1+2^{2j_{2}})(1+2^{j_2+\cdots+j_{2\mu+1}})(2^{j_2}+\cdots+2^{j_{2\mu+1}}).
\end{align*}
Combining the above three parts, we have
\begin{align*}
&\left\|
\left(
\sum\limits_{i=1}^{2\mu+1}\eta_i
\right)
\left(1+\left|\sum\limits_{i=1}^{2\mu+1}\eta_i\right|^2\right)^s
\Tb_{\mu}(\eta_1, \eta_{2}, \cdots,\eta_{2\mu+1})
\psi_{j_{1}}(\eta_1) 
\right.
\\
&\qquad\qquad 
\left.-(\eta_1)(1+|\eta_1|^2)^s \Tb_{\mu}\left(-\sum\limits_{i=1}^{2\mu+1}\eta_i, \eta_2,\cdots,\eta_{2\mu+1}\right)\psi_{j_{1}}\left(-\sum\limits_{i=1}^{2\mu+1}\eta_i\right)
\right\|_{S^\infty}
\\
\lesssim~~& 2^{2sj_{1}+(j_2+\cdots+j_{2\mu+1})}(1+2^{4j_{2}})(1+2^{j_2+\cdots+j_{2\mu+1}}).
\end{align*}
Therefore, 
\begin{align*}
\bigg|&\iiint_{\R^{{2\mu+1}}}\xi(1+|\xi|^2)^s
\Tb_{\mu}(\eta_1,\cdots,\eta_{2\mu}, \xi-\eta_1-\cdots-\eta_{2\mu}) 
\\
&\qquad \hat \vp_{j_1}(\eta_1)\hat u_{j_{2}}(\eta_2)\cdots\hat u_{j_{2\mu}}(\eta_{2\mu})\hat u_{j_{2\mu+1}}(\xi-\eta_1-\cdots-\eta_{2\mu})\hat \vp(-\xi)\diff \eta_1\cdots\diff \eta_{{2\mu}}\diff\xi\bigg|
\\
\lesssim~&
\|\vp_{j_{1}}\|_{H^s}^2\|\partial_x^2u_{j_{2}}\|_{W^{4,\infty}}\prod_{i=3}^{2\mu+1}\|\partial_x u_{j_i}\|_{W^{1,\infty}}.
\end{align*}
After taking the summation for integers $j_1\geq \cdots\geq j_{2\mu+1}$, we obtain
\begin{align*}
\frac{\diff}{\diff t} \|\vp(t)\|_{H^s}^2\lesssim \sum_{\mu=1}^\infty \|\vp(t)\|_{H^s}^2\sum
\limits_{\tiny
\begin{array}{c}
 j_i\in\Z
 \\ 
 i=2,\cdots,2\mu+1
\end{array}}
\left(\|\partial_x^2u_{j_{2}}\|_{W^{4,\infty}}\prod_{i=3}^{2\mu+1}\|\partial_x u_{j_i}\|_{W^{1,\infty}}\right).
\end{align*}
Replacing $u$ by $\vp$, it leads to the energy inequality \ref{EnergyIneq}. 
By Sobolev embedding theorem, we can obtain
\begin{align*}
\frac{\diff}{\diff t} \|\vp(t)\|_{H^s}^2
\lesssim &
\sum_{\mu=1}^\infty \|\vp(t)\|_{H^s}^2\sum\limits_{\tiny
\begin{array}{c} 
j_i\in\Z
\\ 
i=2,\cdots,2\mu+1
\end{array}}
\left(\|\partial_x^2u_{j_{2}}\|_{H^5}\prod_{i=3}^{2\mu+1}\|\partial_x u_{j_i}\|_{H^2}
\right)
\\
\lesssim~&
\sum_{\mu=1}^\infty \|\vp(t)\|_{H^s}^2 \|\partial_x^2 u(t)\|_{H^5}\|\partial_xu(t)\|_{H^2}^{2\mu-1}.
\end{align*}
After taking integral with $t$, we obtain
\[
\|\vp(t)\|_{H^s}\leq \|\vp_0\|_{H^s}+C\int_0^t \sum_{\mu=1}^\infty \|\vp(\tau)\|_{H^s}^2 \|\partial_x^2 u(\tau)\|_{H^5}\|\partial_xu(\tau)\|_{H^2}^{2\mu-1}\diff \tau,
\]
for some constant $C$.
Since the initial data satisfies $\|\vp_0\|_{H^s}<\bar C$, and $\|\vp_0\|_{H^3}<1$, there exist $T_1>0$, such that $\vp\in X_{T_1}$.

\subsection{Contraction in lower space}
Assume for $i=1,2$, $u_i \in X_T$, and $\vp_i$ is the solution of
\[
\partial_t\vp_i=\Ac(u_i)\vp_i, \quad \vp_i(0,x)=\vp_0(x)\in H^s(\R), s\geq 8.
\]
We will prove the contraction of $\B$ in $L^\infty(0,T; H^7(\R))$.
By taking difference of the two equations,
\[
\partial_t(\vp_1-\vp_2)=\Ac(u_1)(\vp_1-\vp_2)+[\Ac(u_1)-\Ac(u_2)]\vp_2.
\]
Multiply $(1+|\xi|^2)^7(\hat \vp_1(t, -\xi)-\hat \vp_2(t, -\xi))$ to the Fourier transform of the above equation, and take integral with $\xi$,
\begin{align*}
&\frac{\diff}{\diff t}\|(\vp_1-\vp_2)(t)\|_{H^7}^2\\
\lesssim~ &\sum_{\mu=1}^\infty \|(\vp_1-\vp_2)(t)\|_{H^7}^2 \|\partial_x^2 u_1(t)\|_{H^5}\|\partial_xu_1(t)\|_{H^2}^{2\mu-1}+\Big[\|\px^2(u_1-u_2)(t)\|_{H^5}(\|\partial_xu_1(t)\|_{H^2}^{2\mu-1}+\|\partial_xu_2(t)\|_{H^2}^{2\mu-1})\\
&+(\|\px^2 u_1(t)\|_{H^5}+\|\px^2 u_2(t)\|_{H^5})\|\partial_x(u_1-u_2)(t)\|_{H^2}(\|\partial_xu_1(t)\|_{H^2}^{2\mu-2}+\|\partial_xu_2(t)\|_{H^2}^{2\mu-2})\Big]\|\vp_2(t)\|_{H^8}\|(\vp_1-\vp_2)(t)\|_{H^7}.
\end{align*}
By Gronwall inequality,
\begin{align*}
\|(\vp_1-\vp_2)(t)\|_{H^7}\leq \int_0^t e^{\int_0^s \|\partial_x^2 u_1(\tau)\|_{H^5}\sum_{\mu}\|\partial_xu_1(\tau)\|_{H^2}^{2\mu-1}\diff \tau}\Big[\|\px^2(u_1-u_2)(s)\|_{H^5}(\|\partial_xu_1(s)\|_{H^2}^{2\mu-1}+\|\partial_xu_2(s)\|_{H^2}^{2\mu-1})\\
+(\|\px^2 u_1(s)\|_{H^5}+\|\px^2 u_2(s)\|_{H^5})\|\partial_x(u_1-u_2)(s)\|_{H^2}(\|\partial_xu_1(s)\|_{H^2}^{2\mu-2}+\|\partial_xu_2(s)\|_{H^2}^{2\mu-2})\Big]\|\vp_2(s)\|_{H^8}\diff s.
\end{align*}
Since $u_1, u_2$ are both bounded in $L^\infty(0,T; H^s(\R)), s\geq 8$, there exists $T_2>0$, such that there exist $L<1$, 
\[
\|(\vp_1-\vp_2)\|_{L^\infty_{T_2}H^7}\leq L\|u_1-u_2\|_{L^\infty_{T_2}H^7}.
\]
This contraction argument also leads to the uniqueness of solutions to Cauchy problem \eqref{front}.

\subsection{Iteration scheme}
We then construct a sequence $\{\vp^{(i)}\}$ of approximate solutions of \eqref{front} by
\begin{equation}
\vp^{(0)}(x,t)=\vp_0(x),\qquad \vp^{(i)}=\B(\vp^{(i-1)})\quad \text{for $i\in \mathbb N$}.
\label{defseq}
\end{equation}
For sufficiently small $T > 0$, we have proved that this sequence is bounded in  $L^\infty([0,T]; H^s(\R))$ and Cauchy in
$L^\infty([0,T]; H^7(\R))$, which implies that its limit exists in $L^\infty([0,T]; H^7(\R))$ and is a local solution of the initial value problem \eqref{front}.

\subsection{Continuity in time}
Next, we prove that the solution constructed above is a continuous function of time with values in $H^s$.
First, we notice that $\vp_t\in L^{\infty}([0,T];H^{s'}(\R))$ for any
$s' < s-1$, which implies that $\vp\in C([0,T]; H^{s'}(\R))$.

The equation is time reversible and translation invariant in time, so it suffices to prove that
\[
\lim\limits_{t\to0+}\|\vp(t)-\vp(0)\|_{H^s}=0.
\]
Since $\|\vp(t)\|_{H^s}$ is bounded on $[0,T]$ and $\vp(t) \to \vp(0)$ strongly in $H^{s'}$, the weak $H^s$-limit of any convergent subsequence is unique, and we see that
 $\vp(t)$ converges to $\vp(0)$ in the weak $H^s$-topology.
To show convergence in the strong $H^s$-topology, we only need to prove the norm-convergence
\begin{equation}
\lim\limits_{t\to0+}\|\vp(t)\|_{H^s}=\|\vp(0)\|_{H^s}.
\label{normcon}
\end{equation}
This is directly from the energy inequality and the Gronwall's inequality.

Finally, we obtained the following existence and uniqueness theorem of Cauchy problem \eqref{front}.
\begin{theorem}[Local solutions]
If the initial data $\vp_0\in H^s(\R)$ for $s\geq 8$ satisfying $\| \vp_0\|_{H^3}<1$, then there exist $T>0$ such that there is a unique solution to the Cauchy problem \eqref{front} in $C([0,T]; H^s(\R))$. The solution satisfies the following energy inequality
\begin{equation}\label{EnergyIneq}
\|\vp(t)\|_{H^s}^2\lesssim \|\vp_0\|_{H^s}^2+\int_0^t \|\vp(\tau)\|_{H^s}^2\sum_{\mu=1}^\infty \|\vp(\tau)\|_{B^{2,6}}\|\vp(\tau)\|_{B^{1,2}}^{2\mu-1}\diff \tau.
\end{equation}
\end{theorem}

\section{Global solutions for small localized initial data}\label{sec:Glo}
In the remaining sections, we will prove the global boundedness of the energy. We formulate the main theorem first. 
Define the profile function $h(t,x)=e^{t\partial_x(1-\partial_x^2)^{-1/2}}\vp(t,x)$ and the dispersion relation $p(\xi)=-\xi(1+\xi^2)^{-1/2}$. Thus $\hat h(t, \xi)=e^{-itp(\xi)}\hat \vp(t, \xi)$. Define some constants:
\begin{align}\label{param_vals}
r=0.4, ~~w=11,~~ s=130,~~ p_0=10^{-4}.
\end{align}
$Z$-norm is defined as 
\begin{equation}\label{Znorm}
\|f\|_Z:=\|(|\xi|^{r}+|\xi|^{w})\hat f(\xi)\|_{L^\infty_\xi}.
\end{equation}
Our main theorem is 
\begin{theorem}\label{Thm:glo}
Let $s$, $p_0$ be defined as in \eqref{param_vals}. There exists a constant $0<\ve \ll 1$, such that if $\vp_0\in H^s(\R)$ and satisfies
\[
\|\vp_0\|_{H^s} + \|\partial_\xi\hat\vp_0(\xi)\|_{L^2}+\|\vp_0\|_{Z}\leq \ve_0,
\]
for some $0 < \ve_0 \leq \ve$, then there exists a unique global solution $\vp\in C([0,\infty); H^s(\R))$ of \eqref{front}.
Moreover, this solution satisfies
\[
(t+1)^{-p_0} \|\vp(t)\|_{H^s}+(t+1)^{-0.01}\|\partial_\xi\hat h(t,\xi)\|_{L^2_\xi} + \|\vp\|_{Z} \lesssim\ve_0.
\]
\end{theorem}
The proof of this theorem follows directly from the following bootstrap proposition.

\begin{proposition}[Bootstrap]\label{bootstrap}
Let $T>1$ and suppose that $\vp\in C([0,T]; H^s)$ is a solution of \eqref{front}, where the initial data satisfies
\[
\|\vp_0\|_{H^s} + \|\partial_\xi\hat\vp_0(\xi)\|_{L^2}+\|\vp_0\|_{Z}\leq \ve_0
\]
for some $0 < \ve_0 \ll 1$. If there exists $\ve_1$ with $\ve_0 \leq \ve_1 \lesssim \ve_0^{2/3}$ such that the solution satisfies
\[
(t+1)^{-p_0} \|\vp(t)\|_{H^s}+(t+1)^{-0.01}\|\partial_\xi\hat h(t,\xi)\|_{L^2_\xi} +\|\vp(t)\|_{Z}\leq \ve_1,
\]
for every $t\in [0,T]$, then the solution satisfies an improved bound
\[
(t+1)^{-p_0} \|\vp(t)\|_{H^s}+(t+1)^{-0.01}\|\partial_\xi\hat h(t,\xi)\|_{L^2_\xi} + \|\vp(t)\|_{Z} \lesssim\ve_0.
\]
\end{proposition}

\begin{lemma}[Nonlinear pointwise decay]\label{NonDis}
Under the bootstrap assumptions,
\[
\|\vp\|_{B^{1,6}} \lesssim (t+1)^{-1/2}\ve_1.
\]
\end{lemma}

\begin{lemma}[Weighted energy estimate]\label{weightedE}
Under the bootstrap assumptions,
\[
\|\partial_\xi\hat h(t,\xi)\|_{L^2_\xi}\lesssim \ve_0(t + 1)^{0.01}.
\]
\end{lemma}

\begin{lemma}[Z-norm estimate]\label{Z-norm_Est}
Under the bootstrap assumptions,
\[
\|\vp(t)\|_{Z}\lesssim \ve_0.
\]
\end{lemma}

\begin{proof}[Proof of Proposition \ref{bootstrap}]
By Lemma \ref{NonDis} and the energy inequality \eqref{EnergyIneq}, we obtain
\[
\|\vp(t)\|_{H^s}\lesssim \ve_0(t+1)^{p_0}.
\]
Then using Lemma \ref{weightedE} and Lemma \ref{Z-norm_Est}, we complete the proof.
\end{proof}

In the following sections, we will prove Lemma \ref{NonDis}, Lemma \ref{weightedE} and Lemma \ref{Z-norm_Est} separately.

\section{Dispersive estimate}\label{sec:Dis}
\begin{lemma}[Linear dispersive estimate]\label{disp}
 For $t> 0$ and $h \in L^2$, we have the linear dispersive estimates
\begin{align}\nonumber
&\|e^{itp(D)} P_kh \|_{L^\infty}\\\label{LocDis}
\lesssim~& \frac1{\sqrt t} 2^{-\frac12k}(1+2^{2k})^{5/2}  \|\psi_k(\xi)\hat h(t, \xi)\|_{L^\infty_\xi}+(t+1)^{-3/4} 2^{-\frac34k}(1+2^{2k})^{5/2}\|\partial_\xi\hat h (t, \xi)\psi_k(\xi)\|_{L^2_\xi},
\end{align}
where $D=-i\px$.
\end{lemma}
\begin{proof} Using the inverse Fourier transform, we can write the solution as
\begin{align*}
e^{itp(D)}  P_kh
= \int_\R e^{ix\xi+ip(\xi)t}  \psi_k(\xi)\hat h (\xi)\diff{\xi}.
\end{align*}
Notice that
\begin{equation}\label{exp}
\partial_\xi e^{ix\xi+ip(\xi)t}=i[x+tp'(\xi)]  e^{ix\xi+ip(\xi)t}.
\end{equation}
Denote $\xi_0^\pm$ to be the solution of $p'(\xi_0^\pm)=-x/t$. Then we discuss different situations for $\xi_0^\pm$. 

\noindent{\bf Case I.} When $|\xi_0^\pm|>\frac85 2^{k+1}$ or $|\xi_0^\pm|<\frac58 2^{k-1}$, integral by part with respect to $\xi$,
\begin{align}\nonumber
&\int_\R e^{ix\xi+ip(\xi)t}\hat h (\xi)\psi_k(\xi)\diff{\xi}\\\label{eqn6.3}
=&\int_\R e^{ix\xi+ip(\xi)t}\frac{-ip''(\xi)}{-t[\frac{x}t+p'(\xi)]^2}\hat h(\xi)\psi_k(\xi)\diff\xi+\int_\R e^{ix\xi+ip(\xi)t}\frac{ 1}{it(\frac{x}t+p'(\xi))} \partial_\xi \left(\hat h (\xi)\psi_k(\xi)\right)\diff\xi.
\end{align}
Since $p'(\xi)=-\frac1{(1+\xi^2)^{3/2}}$ is an even function and is increasing on $[0,\infty)$,
\begin{align*}
\Big|\frac{x}t+p'(\xi)\Big|=|p'(\xi)-p'(\xi_0)|\gtrsim |p''(2^k)|2^k\gtrsim 2^{2k}(1+2^{2k})^{-5/2}.
\end{align*}
Therefore, the two integrals in \eqref{eqn6.3} satisfy
\begin{align*}
&\left|\int_\R e^{ix\xi+ip(\xi)t}\frac{-ip''(\xi)}{-t[\frac{x}t+p'(\xi)]^2}\hat h(\xi)\psi_k(\xi)\diff\xi\right|\\
\lesssim~& \frac1t 2^{-4k}(1+2^{2k})^{5} 2^k(1+2^{2k})^{-5/2}2^{k} \|\psi_k(\xi)\hat h(\xi)\|_{L^\infty}\\
\lesssim~& \frac1t 2^{-2k}(1+2^{k})^{5}  \|\psi_k\hat h\|_{L^\infty},
\end{align*}
and
\begin{align*}
&\left|\int_\R e^{ix\xi+ip(\xi)t}\frac{ 1}{it(\frac{x}t+p'(\xi))} \partial_\xi \left(\hat h (\xi)\psi_k(\xi)\right)\diff\xi\right|\\
\lesssim~&\frac1t 2^{-k}(1+2^{2k})^{5/2}2^{-\frac{k}2}\|\partial_\xi(\hat h (\xi)\psi_k(\xi))\|_{L^2_\xi}\\
\lesssim~&\frac1t 2^{-\frac32k}(1+2^{2k})^{5/2}(\|\partial_\xi\hat h (\xi)\psi_k(\xi)\|_{L^2_\xi}+\|\hat h (\xi)\psi_k(\xi)\|_{L^\infty_\xi}2^{-k/2})\\
\lesssim~&\frac1t 2^{-\frac32k}(1+2^{2k})^{5/2}\|\partial_\xi\hat h (\xi)\psi_k(\xi)\|_{L^2_\xi}+ \frac1t 2^{-2k}(1+2^{2k})^{5/2}  \|\psi_k\hat h\|_{L^\infty}.
\end{align*}
So we have
\begin{align*}
&\left|\int_\R e^{ix\xi+ip(\xi)t}\hat h (\xi)\psi_k(\xi)\diff{\xi}\right|\\
\lesssim~&\frac1t 2^{-\frac32k}(1+2^{2k})^{5/2}\|\partial_\xi\hat h (\xi)\psi_k(\xi)\|_{L^2_\xi}+ \frac1t 2^{-2k}(1+2^{2k})^{5/2}  \|\psi_k\hat h\|_{L^\infty}.
\end{align*}
When $t>2^{-3k}$, we have $t^{1/4}>2^{-\frac34k}$, $t^{1/2}>2^{-\frac32k}$, thus
\begin{align}\nonumber
&\left|\int_\R e^{ix\xi+ip(\xi)t}\hat h (\xi)\psi_k(\xi)\diff{\xi}\right|\\\label{case1}
\lesssim~&\frac1{t^{3/4}}  2^{-\frac34k}(1+2^{2k})^{5/2}\|\partial_\xi\hat h (\xi)\psi_k(\xi)\|_{L^2_\xi}+ \frac1{t^{1/2}} 2^{-\frac12k}(1+2^{2k})^{5/2}  \|\psi_k\hat h\|_{L^\infty}.
\end{align}
When $t<2^{-3k}$, we have $2^{\frac32k}<t^{-1/2}$, thus
\begin{align}\label{lowfre}
\left|\int_\R e^{ix\xi+ip(\xi)t}\hat h (\xi)\psi_k(\xi)\diff{\xi}\right|
\lesssim~ 2^{k} \|\psi_k\hat h\|_{L^\infty}\lesssim \frac1{t^{1/2}} 2^{-\frac12k}(1+2^{2k})^{5/2}  \|\psi_k\hat h\|_{L^\infty}.
\end{align}

\noindent{\bf Case II.} When $\frac58 2^{k-1}\leq |\xi_0^\pm|\leq \frac85 2^{k+1}$, we make further dyadic decomposition
\begin{align*}
\int_\R e^{ix\xi+ip(\xi)t}\hat h (\xi)\psi_k(\xi)\diff{\xi}
=\sum\limits_{l}\int_\R e^{ix\xi+ip(\xi)t}\hat h (\xi)\psi_k(\xi)\psi_{l}(|\xi|-|\xi_0^\pm|)\diff{\xi}.
\end{align*}
Considering the support of the cut-off functions, we can assume the above summation is for $l\leq k+1$.

After integration by part, we have
\begin{align*}
&\left|\int_\R e^{ix\xi+ip(\xi)t}\hat h (\xi)\psi_k(\xi)\psi_{l}(|\xi|-|\xi_0^\pm|)\diff{\xi}\right|\\
\lesssim~& \left|\int_\R e^{ix\xi+ip(\xi)t}\frac{-ip''(\xi)}{-t[\frac{x}t+p'(\xi)]^2}\hat h(\xi)\psi_k(\xi)\psi_{l}(|\xi|-|\xi_0^\pm|)\diff\xi\right|\\
&+\left|\int_\R e^{ix\xi+ip(\xi)t}\frac{1}{it(\frac{x}t+p'(\xi))} \partial_\xi \left(\hat h (\xi)\psi_k(\xi)\psi_{l}(|\xi|-|\xi_0^\pm|)\right)\diff\xi\right|.
\end{align*}
On the support of the cut-off function $\psi_{l}(|\xi|-|\xi_0^\pm|)$, 
\begin{align*}
\left|\frac{x}t+p'(\xi)\right|=|p'(\xi)-p'(\xi_0)|\gtrsim |p''(2^k)|2^l\gtrsim 2^{k+l}(1+2^{2k})^{-5/2}.
\end{align*}
Therefore,
\begin{align*}
&\left|\int_\R e^{ix\xi+ip(\xi)t}\frac{-ip''(\xi)}{-t[\frac{x}t+p'(\xi)]^2}\hat h(\xi)\psi_k(\xi)\psi_{l}(|\xi|-|\xi_0^\pm|)\diff\xi\right|\\
\lesssim~& \frac1t 2^{-2k-2l}(1+2^{2k})^{5} 2^k(1+2^{2k})^{-5/2}2^l \|\psi_k(\xi)\psi_{l}(|\xi|-|\xi_0^\pm|)\hat h(\xi)\|_{L^\infty_\xi}\\
\lesssim~& \frac1t 2^{-k-l}(1+2^{2k})^{5/2}   \|\psi_k(\xi)\psi_{l}(|\xi|-|\xi_0^\pm|)\hat h(\xi)\|_{L^\infty_\xi},
\end{align*}
and
\begin{align*}
&\left|\int_\R e^{ix\xi+ip(\xi)t}\frac{1}{it(\frac{x}t+p'(\xi))} \partial_\xi \left(\hat h (\xi)\psi_k(\xi)\psi_{l}(|\xi|-|\xi_0^\pm|)\right)\diff\xi\right|\\
\lesssim~&\frac1t 2^{-k-l}(1+2^{2k})^{5/2}2^{\frac{l}2}\|\partial_\xi (\hat h (\xi)\psi_k(\xi)\psi_{l}(|\xi|-|\xi_0^\pm|))\|_{L^2_\xi}\\
\lesssim~&\frac1t 2^{-k-\frac l2}(1+2^{2k})^{5/2}\left[\|\partial_\xi\hat h (\xi)\psi_k(\xi)\psi_{l}(|\xi|-|\xi_0^\pm|)\|_{L^2_\xi}+\|\hat h (\xi)\psi_k(\xi)\psi_{l}(|\xi|-|\xi_0^\pm|)\|_{L^\infty_\xi}2^{-l/2})\right].
\end{align*}
Take the summation for $l$ from $-\frac12\log_2(t+1)-\frac k2$ to $k+1$, 
\begin{align*}
&\left|\sum\limits_{l=-\frac12\log_2(t+1)}^{k+1}\int_\R e^{ix\xi+ip(\xi)t}\hat h (\xi)\psi_k(\xi)\diff{\xi}\right|\\
\lesssim~& t^{-1/2} 2^{-\frac12 k}(1+2^{2k})^{5/2}  \|\psi_k(\xi)\hat h(\xi)\|_{L^\infty_\xi}+(t+1)^{-3/4} 2^{-\frac34 k}(1+2^{2k})^{5/2}\|\partial_\xi\hat h (\xi)\psi_k(\xi)\|_{L^2_\xi}.
\end{align*}
For $l<-\frac12\log_2(t+1)-\frac k2$,
\[
\left|\sum\limits_{l<-\frac12\log_2(t+1)-\frac k2}\int_\R e^{ix\xi+ip(\xi)t}\hat h (\xi)\psi_k(\xi)\psi_{l}(|\xi|-|\xi_0^\pm|)\diff{\xi}\right|\lesssim (t+1)^{-1/2} 2^{-\frac12 k}\|\psi_k(\xi)\hat h(\xi)\|_{L^\infty_\xi}.
\]
So we have
\begin{align}\nonumber
&\left|\int_\R e^{ix\xi+ip(\xi)t}\hat h (\xi)\psi_k(\xi)\diff{\xi}\right|\\
\lesssim~& \frac1{\sqrt t} 2^{-\frac12k}(1+2^{2k})^{5/2}  \|\psi_k(\xi)\hat h(\xi)\|_{L^\infty_\xi}+(t+1)^{-3/4} 2^{-\frac34k}(1+2^{2k})^{5/2}\|\partial_\xi\hat h (\xi)\psi_k(\xi)\|_{L^2_\xi}.\label{case2}
\end{align}
Hence we obtain the conclusion by combining \eqref{case1}, \eqref{lowfre} and \eqref{case2}.
\end{proof}

\begin{proof}[Proof of Lemma \ref{NonDis}]
By Lemma \ref{disp}, for $t>1$, we have
\begin{align*}
\|\vp\|_{B^{1,6}}\lesssim &\sum\limits_{2^k\leq t^{1/100}}(2^k+2^{6k})\|e^{t\partial_x(1-\partial_x^2)^{-1/2}} P_kh \|_{L^\infty}+\sum\limits_{2^k> t^{1/100}}\|\partial_x^{6}P_k\vp\|_{L^\infty}\\
\lesssim &\sum\limits_{2^k\leq t^{1/100}}(1+2^{5k})\left[ \frac1{\sqrt t} 2^{\frac12k}(1+2^{2k})^{5/2}  \|\psi_k(\xi)\hat h(\xi)\|_{L^\infty_\xi}+(t+1)^{-3/4} 2^{\frac14k}(1+2^{2k})^{5/2}\|\partial_\xi\hat h (\xi)\psi_k(\xi)\|_{L^2_\xi}\right]\\
&+\sum\limits_{2^k> t^{1/100}} 2^{6k}2^{k/2-sk}\||\xi|^s\psi_k\hat\vp\|_{L^2_\xi}\\
\lesssim &t^{-1/2} \|(|\xi|^{0.4}+|\xi|^{11}) \hat h(\xi)\|_{L^\infty_\xi}+(t+1)^{-\frac34+\frac{11}{100}}\|\partial_\xi \hat h\|_{L^2_\xi}+t^{-1}\|\vp\|_{H^s}\\
\lesssim &t^{-1/2} \ve_1,
\end{align*}
where the last step is from the bootstrap assumption. 
\end{proof}

\section{Weighted energy estimate}\label{sec:WE}
In this section, we prove Lemma \ref{weightedE}.
Recall that $\hat h(t, \xi)=e^{-itp(\xi)}\hat \vp(t, \xi)$. Then the equation \eqref{front} can be expressed as 
\begin{equation}\label{eqhhat}
\hat h_t= e^{-itp(\xi)}i\xi\widehat{\Nc(\vp)} = e^{-itp(\xi)}i\xi\left(\widehat{\Nc_3(\vp)}+\widehat{\Nc_{\geq5}(\vp)}\right).
\end{equation}
Denote 
\begin{equation}\label{defPhi}
\Phi(\eta_1,\eta_2,\xi)=p(\eta_1)+p(\eta_2)+p(\xi-\eta_1-\eta_2)-p(\xi)=
-\frac{\eta_1}{\sqrt{1+\eta_1^2}}-\frac{\eta_2}{\sqrt{1+\eta_2^2}}-\frac{\xi-\eta_1-\eta_2}{\sqrt{1+(\xi-\eta_1-\eta_2)^2}}+\frac{\xi}{\sqrt{1+\xi^2}}.
\end{equation}
By \eqref{N3}, we have
\begin{align*}
e^{-itp(\xi)}\widehat{\Nc_3(\vp)}(\xi)&=\frac13  e^{-itp(\xi)}\iint_{\R^2}\Tb_1(\eta_1,\eta_2,\xi-\eta_1-\eta_2)\hat\vp(\eta_1)\hat\vp(\eta_2)\hat\vp(\xi-\eta_1-\eta_2)\diff \eta_1\diff \eta_2\\
&=\frac13 \iint_{\R^2}\Tb_1(\eta_1,\eta_2,\xi-\eta_1-\eta_2)e^{it\Phi}\hat h(\eta_1)\hat h(\eta_2)\hat h(\xi-\eta_1-\eta_2)\diff \eta_1\diff \eta_2\\
&=\frac13  \sum\limits_{j_1,j_2,j_3\in\Z}\iint_{\R^2}\Tb_1(\eta_1,\eta_2,\xi-\eta_1-\eta_2)e^{it\Phi}\hat h_{j_1}(\eta_1)\hat h_{j_2}(\eta_2)\hat h_{j_3}(\xi-\eta_1-\eta_2)\diff \eta_1\diff \eta_2.
\end{align*}
Since this integral is symmetric with the three variables $\eta_1, \eta_2, \xi-\eta_1-\eta_2$, after the possible change of variables, we assume $j_1\geq j_2\geq j_3$.
Denote the set $\P$ as all the indices $j_1,j_2,j_3\in\Z$, such that $j_1\geq j_2\geq j_3$ with possible repetition.

Take $\partial_\xi$ to \eqref{eqhhat},
\begin{equation}\label{eqn7.2}
\begin{aligned}
\partial_{\xi}\hat h_t
=~&\frac i3\xi \iint_{\R^2}\partial_{\xi}\Tb_1(\eta_1,\eta_2,\xi-\eta_1-\eta_2) e^{it\Phi} \hat h(\eta_1)\hat h(\eta_2)\hat h(\xi-\eta_1-\eta_2)\diff \eta_1\diff \eta_2\\
&+\frac i3\xi  \sum\limits_{\P}\iint_{\R^2}\Tb_1(\eta_1,\eta_2,\xi-\eta_1-\eta_2) e^{it\Phi} \hat h_{j_1}(\eta_1)\hat h_{j_2}(\eta_2)\partial_{\xi}\hat h_{j_3}(\xi-\eta_1-\eta_2)\diff \eta_1\diff \eta_2\\
&-\frac13\xi t\sum\limits_{\P}\iint_{\R^2}\Tb_1(\eta_1,\eta_2,\xi-\eta_1-\eta_2)\partial_{\xi}\Phi e^{it\Phi} \hat h_{j_1}(\eta_1)\hat h_{j_2}(\eta_2)\hat h_{j_3}(\xi-\eta_1-\eta_2)\diff \eta_1\diff \eta_2\\
&+\frac i3 \iint_{\R^2}\Tb_1(\eta_1,\eta_2,\xi-\eta_1-\eta_2) e^{it\Phi} \hat h(\eta_1)\hat h(\eta_2)\hat h(\xi-\eta_1-\eta_2)\diff \eta_1\diff \eta_2\\
&+i\partial_\xi[ \xi e^{-itp(\xi)}\widehat{\Nc_{\geq5}(\vp)}]\\
=:~&\rm{I+II+III+IV+V}.
\end{aligned}\end{equation}
We split the index set $\P$ into $\P_1\bigcup\P_2$, where $\P_1$ includes the indices satisfying $j_1\geq j_2\geq j_3$ and $ j_1-j_3> 1$, which correspond to nonresonant frequencies. $\P_2$ includes indices satisfying $j_1\geq j_2\geq j_3$ and $ j_1-j_3\leq 1$, which contains the resonant frequencies.
Thus the term $\rm{III}$ can be written as
\begin{align*}
\rm{III}=&-\frac 13\xi t \sum\limits_{\P_1}\iint_{\R^2}\Tb_1(\eta_1,\eta_2,\xi-\eta_1-\eta_2)\partial_{\xi}\Phi e^{it\Phi} \hat h_{j_1}(\eta_1)\hat h_{j_2}(\eta_2)\hat h_{j_3}(\xi-\eta_1-\eta_2)\diff \eta_1\diff \eta_2\\
&-\frac 13\xi t \sum\limits_{\P_2}\iint_{\R^2}\Tb_1(\eta_1,\eta_2,\xi-\eta_1-\eta_2)\partial_{\xi}\Phi e^{it\Phi} \hat h_{j_1}(\eta_1)\hat h_{j_2}(\eta_2)\hat h_{j_3}(\xi-\eta_1-\eta_2)\diff \eta_1\diff \eta_2\\
=:&\rm{III}_1+\rm{III}_2.
\end{align*}
For $\rm{III}_2$, we need the further decomposition. Define
\[
\upsilon_\pm(\eta)=\left\{
\begin{aligned}
&1,\quad {\text {if } } \pm\eta\geq 0,\\
&0, \quad {\text {if } } \pm\eta<0,
\end{aligned}
\right.
\]
and $\hat h_{j}^\pm(\eta)=\hat h_j(\eta)\upsilon_\pm(\eta)$. 
The integrals in $\rm{III}_2$ can be written as
\begin{align*}
&-\frac 13\xi t \sum\limits_{\P_2}\iint_{\R^2}\Tb_1(\eta_1,\eta_2,\xi-\eta_1-\eta_2)\partial_{\xi}\Phi e^{it\Phi} \hat h_{j_1}(\eta_1)\hat h_{j_2}(\eta_2)\hat h_{j_3}(\xi-\eta_1-\eta_2)\diff \eta_1\diff \eta_2\\
=~&-\sum\limits_{\iota_1, \iota_2, \iota_3\in\{\pm\}}\frac 13\xi t \sum\limits_{\P_2}\iint_{\R^2}\Tb_1(\eta_1,\eta_2,\xi-\eta_1-\eta_2)\partial_{\xi}\Phi e^{it\Phi} \hat h_{j_1}^{\iota_1}(\eta_1)\hat h_{j_2}^{\iota_2}(\eta_2)\hat h_{j_3}(\xi-\eta_1-\eta_2)\upsilon_{\iota_3}(\xi)\diff \eta_1\diff \eta_2.
\end{align*}
Define the nonresonant part
\begin{align*}
{\rm{III}}_{21}=-\sum\limits_{\mbox{$\scriptsize\begin{array}{cc}(\iota_1, \iota_2, \iota_3)=\\ (+,+,-)\\
\text{ or } (-, -, +)\end{array}$}}\frac 13\xi t \sum\limits_{\P_2}\iint_{\R^2}\Tb_1(\eta_1,\eta_2,\xi-\eta_1-\eta_2)\partial_{\xi}\Phi e^{it\Phi} \hat h_{j_1}^{\iota_1}(\eta_1)\hat h_{j_2}^{\iota_2}(\eta_2)\hat h_{j_3}(\xi-\eta_1-\eta_2)\upsilon_{\iota_3}(\xi)\diff \eta_1\diff \eta_2,
\end{align*}
and the space-time resonance terms around $(\xi, -\xi)$ or $(-\xi, \xi)$,
\begin{align*}
{\rm{III}}_{22}=-\sum\limits_{\scriptsize\begin{array}{cc}\mbox{$(\iota_1, \iota_2)=(\pm, \mp)$}\\ \mbox{$\iota_3\in \{+, -\}$}\end{array}}\frac 13\xi t \sum\limits_{\P_2}\iint_{\R^2}\Tb_1(\eta_1,\eta_2,\xi-\eta_1-\eta_2)\partial_{\xi}\Phi e^{it\Phi} \hat h_{j_1}^{\iota_1}(\eta_1)\hat h_{j_2}^{\iota_2}(\eta_2)\hat h_{j_3}(\xi-\eta_1-\eta_2)\upsilon_{\iota_3}(\xi)\diff \eta_1\diff \eta_2.
\end{align*}
When $(\iota_1, \iota_2, \iota_3)=(+,+,+)$ or $(-,-,-)$, we separate the space resonance and the space-time resonance:
\begin{align*}
&\iint_{\R^2}\Tb_1(\eta_1,\eta_2,\xi-\eta_1-\eta_2)\partial_{\xi}\Phi e^{it\Phi} \hat h_{j_1}^{\iota_1}(\eta_1)\hat h_{j_2}^{\iota_2}(\eta_2)\hat h_{j_3}(\xi-\eta_1-\eta_2)\upsilon_{\iota_3}(\xi)\diff \eta_1\diff \eta_2\\
=~&\iint_{\R^2}\Tb_1(\eta_1,\eta_2,\xi-\eta_1-\eta_2)\partial_{\xi}\Phi e^{it\Phi} \hat h_{j_1}^{\iota_1}(\eta_1)\hat h_{j_2}^{\iota_2}(\eta_2)\hat h_{j_3}(\xi-\eta_1-\eta_2)\upsilon_{\iota_3}(\xi)\psi_{j_1-3}(\eta_1+\eta_2-2\xi)\diff \eta_1\diff \eta_2
\\
&+\iint_{\R^2}\Tb_1(\eta_1,\eta_2,\xi-\eta_1-\eta_2)\partial_{\xi}\Phi e^{it\Phi} \hat h_{j_1}^{\iota_1}(\eta_1)\hat h_{j_2}^{\iota_2}(\eta_2)\hat h_{j_3}(\xi-\eta_1-\eta_2)\upsilon_{\iota_3}(\xi)[1-\psi_{j_1-3}(\eta_1+\eta_2-2\xi)]\diff \eta_1\diff \eta_2.
\end{align*}
Define the space-time resonance around $(\xi, \xi)$ as
\begin{align*}
{\rm{III}}_{23}=&-\sum\limits_{\scriptsize\begin{array}{cc}\mbox{$(\iota_1, \iota_2,\iota_3)=$}\\ \mbox{$(+, +, +)$ or $(-, -, -)$}\end{array}}\frac 13\xi t \sum\limits_{\P_2}\\
&\iint_{\R^2}\Tb_1(\eta_1,\eta_2,\xi-\eta_1-\eta_2)\partial_{\xi}\Phi e^{it\Phi} \hat h_{j_1}^{\iota_1}(\eta_1)\hat h_{j_2}^{\iota_2}(\eta_2)\hat h_{j_3}(\xi-\eta_1-\eta_2)\upsilon_{\iota_3}(\xi)\psi_{j_1-3}(\eta_1+\eta_2-2\xi)\diff \eta_1\diff \eta_2.
\end{align*}
After taking out the time derivative, we can write the space resonance term as
\begin{align*}
&
\frac13\xi t
\iint_{\R^2}\Tb_1(\eta_1,\eta_2,\xi-\eta_1-\eta_2)\partial_{\xi}\Phi e^{it\Phi} \hat h_{j_1}^{\iota_1}(\eta_1)\hat h_{j_2}^{\iota_2}(\eta_2)\hat h_{j_3}(\xi-\eta_1-\eta_2)\upsilon_{\iota_3}(\xi)[1-\psi_{j_1-3}(\eta_1+\eta_2-2\xi)]\diff \eta_1\diff \eta_2
\\
=&\frac{\diff}{\diff t}
\left(
\frac13\xi t
\iint_{\R^2}\Tb_1(\eta_1,\eta_2,\xi-\eta_1-\eta_2)\frac{\partial_{\xi}\Phi}{i\Phi} e^{it\Phi} \hat h_{j_1}^{\iota_1}(\eta_1)\hat h_{j_2}^{\iota_2}(\eta_2)\hat h_{j_3}(\xi-\eta_1-\eta_2)\upsilon_{\iota_3}(\xi)[1-\psi_{j_1-3}(\eta_1+\eta_2-2\xi)]\diff \eta_1\diff \eta_2
\right)
\\
&-\frac13\xi
\iint_{\R^2}\Tb_1(\eta_1,\eta_2,\xi-\eta_1-\eta_2)\frac{\partial_{\xi}\Phi}{i\Phi} e^{it\Phi} \hat h_{j_1}^{\iota_1}(\eta_1)\hat h_{j_2}^{\iota_2}(\eta_2)\hat h_{j_3}(\xi-\eta_1-\eta_2)\upsilon_{\iota_3}(\xi)[1-\psi_{j_1-3}(\eta_1+\eta_2-2\xi)]\diff \eta_1\diff \eta_2
\\
&-\iint_{\R^2}\Tb_1(\eta_1,\eta_2,\xi-\eta_1-\eta_2)\frac{\partial_{\xi}\Phi}{i\Phi} e^{it\Phi} \frac{\diff}{\diff t}[\hat h_{j_1}^{\iota_1}(\eta_1)\hat h_{j_2}^{\iota_2}(\eta_2)\hat h_{j_3}(\xi-\eta_1-\eta_2)]\upsilon_{\iota_3}(\xi)[1-\psi_{j_1-3}(\eta_1+\eta_2-2\xi)]\diff \eta_1\diff \eta_2.
\end{align*}
Define
\begin{align*}
&{\rm{VI}}:=-\frac {\xi}3 t \sum\limits_{\scriptsize\begin{array}{cc}\mbox{$(\iota_1, \iota_2,\iota_3)=$}\\ \mbox{$(+, +, +)$ or $(-, -, -)$}\end{array}}\sum\limits_{\P_2}\\
&\quad\iint_{\R^2}\Tb_1(\eta_1,\eta_2,\xi-\eta_1-\eta_2)\frac{\partial_{\xi}\Phi}{i\Phi} e^{it\Phi} \hat h_{j_1}^{\iota_1}(\eta_1)\hat h_{j_2}^{\iota_2}(\eta_2)\hat h_{j_3}(\xi-\eta_1-\eta_2)\upsilon_{\iota_3}(\xi)[1-\psi_{j_1-3}(\eta_1+\eta_2-2\xi)]\diff \eta_1\diff \eta_2,
\end{align*}
\begin{equation}\label{III241}
\begin{aligned}
&{\rm{III}}_{241}:=
-\frac 13\xi  \sum\limits_{\scriptsize\begin{array}{cc}\mbox{$(\iota_1, \iota_2,\iota_3)=$}\\ \mbox{$(+, +, +)$ or $(-, -, -)$}\end{array}}\sum\limits_{\P_2}\\
&\quad\iint_{\R^2}\Tb_1(\eta_1,\eta_2,\xi-\eta_1-\eta_2)\frac{\partial_{\xi}\Phi}{i\Phi} e^{it\Phi} \hat h_{j_1}^{\iota_1}(\eta_1)\hat h_{j_2}^{\iota_2}(\eta_2)\hat h_{j_3}(\xi-\eta_1-\eta_2)\upsilon_{\iota_3}(\xi)[1-\psi_{j_1-3}(\eta_1+\eta_2-2\xi)]\diff \eta_1\diff \eta_2,
\end{aligned}
\end{equation}
\begin{equation}\label{III242}
\begin{aligned}
&{\rm{III}}_{242}=-\frac 13\xi t\sum\limits_{\scriptsize\begin{array}{cc}\mbox{$(\iota_1, \iota_2,\iota_3)=$}\\ \mbox{$(+, +, +)$ or $(-, -, -)$}\end{array}}\sum\limits_{\P_2}\\
&\iint_{\R^2}\Tb_1(\eta_1,\eta_2,\xi-\eta_1-\eta_2)\frac{\partial_{\xi}\Phi}{i\Phi} e^{it\Phi} \frac{\diff}{\diff t}[\hat h_{j_1}^{\iota_1}(\eta_1)\hat h_{j_2}^{\iota_2}(\eta_2)\hat h_{j_3}(\xi-\eta_1-\eta_2)]\upsilon_{\iota_3}(\xi)[1-\psi_{j_1-3}(\eta_1+\eta_2-2\xi)]\diff \eta_1\diff \eta_2.
\end{aligned}
\end{equation}
So, we write the equation as 
\begin{align*}
\partial_t[\partial_\xi\hat h+\rm{VI}]=\rm{I}+\rm{II}+\rm{III}_1+\rm{III}_{21}+\rm{III}_{22}+\rm{III}_{23}+\rm{III}_{241}+\rm{III}_{242}+\rm{IV}+\rm{V}.
\end{align*}
We multiply the equation by $\partial_\xi\hat h+\rm{VI}$ and then take integral. 
\begin{equation}\label{7.5}
\begin{aligned}
&\left|\frac{\diff}{\diff t}\int \frac12[\partial_\xi \hat h(\xi)+{\rm{VI}}]\overline{[\partial_\xi \hat h(\xi)+{\rm{VI}}]}\diff\xi\right|\\
=~&\left|\Re\int\partial_t[\partial_\xi \hat h(\xi)+{\rm{VI}}]\overline{[\partial_\xi \hat h(\xi)+{\rm{VI}}]}\diff\xi\right|
\\
=~&\left|\Re\int[{\rm{I}+\rm{II}+\rm{III}_1+\rm{III}_{21}+\rm{III}_{22}+\rm{III}_{23}+\rm{III}_{241}+\rm{III}_{242}+\rm{IV+V}}]\overline{[\partial_\xi \hat h(\xi)+{\rm{VI}}]}\diff\xi\right|
\\
\lesssim~&\|[{\rm{I}+\rm{II}+\rm{III}_{21}+\rm{III}_{22}+\rm{III}_{23}+\rm{III}_{241}+\rm{III}_{242}+\rm{IV+V}}]\|_{L^2}\|\partial_\xi \hat h(\xi)+{\rm{VI}}\|_{L^2}\\
&+\left|\Re\int {\rm{III}}_1 \overline{[\partial_\xi \hat h(\xi)+{\rm{VI}}]} 
d\xi
\right|.
\end{aligned}
\end{equation}
We claim that
\begin{equation}\label{claim}
\begin{aligned}
\|\partial_\xi\hat h(t, \xi)\|_{L^2}^2\lesssim ~&\|\partial_\xi\hat h(0, \xi)\|_{L^2}^2+t\|\vp\|_{B^{1,6}}^2\|\vp\|_{L^2}\\
&\quad+\int_0^t \|\vp(\tau)\|_{B^{1,6}}^2 F(\|\vp(\tau)\|_{B^{1,6}})(\|\partial_\xi\hat h(\tau, \xi)\|_{L^2}^2+\|\vp(\tau)\|_{L^2}^2)\diff \tau,
\end{aligned}\end{equation}
where $F$ is a positive polynomial.

In the following, we prove this claim by estimating each term on the right of \eqref{7.5}, and then estimating
the term $\rm{VI}$ to show that $\|\partial_\xi\hat h+\rm{VI}]\|_{L^2_\xi}$ is equivalent to $\|\partial_\xi\hat h\|_{L^2_\xi}$.

\subsection{Term $\rm{I}$ estimate}

By the symbol estimate of $\partial_{\eta_3}\Tb_1$ in Proposition \ref{Prop_symb_Tmu}, we have
\begin{align*}
&\|{\rm{I}}\|_{L^2}\\
=~&\left\|\frac13\xi \iint_{\R^2}\partial_{\xi}\Tb_1(\eta_1,\eta_2,\xi-\eta_1-\eta_2) e^{it\Phi} \hat h(\eta_1)\hat h(\eta_2)\hat h(\xi-\eta_1-\eta_2)\diff \eta_1\diff \eta_2\right\|_{L^2}\\
\lesssim~& \sum\limits_{j_1,j_2,j_3} (1+2^{\min\{j_1,j_2,j_3\}+\med\{j_1,j_2,j_3\}})2^{4\max\{0, \med\{j_1,j_2,j_3\}\}}\\
&\qquad \cdot\|\partial_x\vp_{\min\{j_1,j_2,j_3\}}\|_{L^\infty}\|\partial_x\vp_{\med\{j_1,j_2,j_3\}}\|_{L^\infty}\|\vp_{\max\{j_1,j_2,j_3\}}\|_{H^1}\\
\lesssim~& \left(\sum\limits_{j\in\Z}\|\partial_x\vp_j\|_{L^{\infty}}\right)^2\|\vp\|_{H^{7}}\\
\lesssim~&\|\vp\|_{B^{1,6}}^2\|\vp\|_{H^7}.
\end{align*}

\subsection{Term $\rm{II}$ estimate}
By the definition of $\P$, $j_1=\max\{j_1, j_2, j_3\}$.
\begin{align*}
&\left\|\frac13\xi \iint_{\R^2}\Tb_1(\eta_1,\eta_2,\xi-\eta_1-\eta_2)e^{it\Phi(\eta_1,\eta_2,\xi)} \hat h_{j_1}(\eta_1)\hat h_{j_2}(\eta_2)\partial_{\xi}\hat h_{j_3}(\xi-\eta_1-\eta_2)\diff \eta_1\diff \eta_2\right\|_{L^2}\\
\lesssim~&2^{j_1}\|\partial_\xi\hat h_{j_3}\|_{L^2}\|\Tb_1(\eta_1,\eta_2, \eta_3)\psi_{j_1}(\eta_1)\psi_{j_2}(\eta_2)\psi_{j_3}(\eta_3)\|_{S^{\infty}}\| h_{j_1}\|_{L^\infty}\| h_{j_2}\|_{L^\infty}\\
\lesssim~&2^{j_1+j_2+j_3}(1+2^{j_2})(1+2^{j_3})2^{2\max\{j_2, j_3, 0\}}\|\partial_\xi\hat h_{j_3}\|_{L^2}\| \vp_{j_1}\|_{L^\infty}\| \vp_{j_2}\|_{L^\infty}\\
\lesssim~&2^{(j_1+j_2)/2}\|\partial_\xi\hat h_{j_3}\|_{L^2}\|\partial_x\vp_{j_1}\|_{L^{\infty}}\| \partial_x\vp_{j_2}\|_{W^{4,\infty}}.
\end{align*}
By taking the summation for $(j_1, j_2, j_3)\in \P$, 
\begin{align*}
&\sum\limits_{\P}\left\|\frac13 \xi\iint_{\R^2}\Tb_1(\eta_1,\eta_2,\xi-\eta_1-\eta_2) e^{it\Phi(\eta_1,\eta_2,\xi)} \hat h_{j_1}(\eta_1)\hat h_{j_2}(\eta_2)\partial_{\xi}\hat h_{j_3}(\xi-\eta_1-\eta_2)\diff \eta_1\diff \eta_2\right\|_{L^2}\\
\lesssim~&\|\vp\|_{B^{1,6}}^2\|\partial_\xi\hat h(\xi)\|_{L^2_\xi}.
\end{align*}

\subsection{Term $\rm{III}$'s estimates}
We estimate $\rm{III}_1, \rm{III}_{21}, \rm{III}_{22}, \rm{III}_{23}, \rm{III}_{241}, \rm{III}_{242}$ separately. Among these terms, $\rm{III}_1, \rm{III}_{21}$ are the terms away from resonances. $\rm{III}_{22}, \rm{III}_{23}$ are the terms near the space-time resonances. $\rm{III}_{241}, \rm{III}_{242}$ are related to the terms close to the space resonance. 

\subsubsection{$\rm{III}_1$ estimate}
 When $(j_1, j_2, j_3)\in \P_1$,  $j_1-j_3>1$, this is the nonresonant case.
\begin{align}
\label{eqn7_1}
&\frac 13\xi t \iint_{\R^2}\Tb_1(\eta_1,\eta_2,\xi-\eta_1-\eta_2)\partial_{\xi}\Phi e^{it\Phi} \hat h_{j_1}(\eta_1)\hat h_{j_2}(\eta_2)\hat h_{j_3}(\xi-\eta_1-\eta_2)\diff \eta_1\diff \eta_2
\nonumber
\\
\nonumber
=~&\frac 13\xi  \iint_{\R^2}\Tb_1(\eta_1,\eta_2,\xi-\eta_1-\eta_2)\frac{\partial_{\xi}\Phi}{i\partial_{\eta_1}\Phi} \partial_{\eta_1}e^{it\Phi} \hat h_{j_1}(\eta_1)\hat h_{j_2}(\eta_2)\hat h_{j_3}(\xi-\eta_1-\eta_2)\diff \eta_1\diff \eta_2
\\
=~&-\frac 13\xi  \iint_{\R^2}\partial_{\eta_1}\left[\Tb_1(\eta_1,\eta_2,\xi-\eta_1-\eta_2)\frac{\partial_{\xi}\Phi}{i\partial_{\eta_1}\Phi} \right]e^{it\Phi} \hat h_{j_1}(\eta_1)\hat h_{j_2}(\eta_2)\hat h_{j_3}(\xi-\eta_1-\eta_2)\diff \eta_1\diff \eta_2
\\
&-\frac 13\xi  \iint_{\R^2}\Tb_1(\eta_1,\eta_2,\xi-\eta_1-\eta_2)\frac{\partial_{\xi}\Phi}{i\partial_{\eta_1}\Phi} e^{it\Phi} \partial_{\eta_1}\hat h_{j_1}(\eta_1)\hat h_{j_2}(\eta_2)\hat h_{j_3}(\xi-\eta_1-\eta_2)\diff \eta_1\diff \eta_2
\nonumber
\\
\nonumber
&-\frac 13\xi  \iint_{\R^2}\Tb_1(\eta_1,\eta_2,\xi-\eta_1-\eta_2)\frac{\partial_{\xi}\Phi}{i\partial_{\eta_1}\Phi} e^{it\Phi} \hat h_{j_1}(\eta_1)\hat h_{j_2}(\eta_2)\partial_{\eta_1}\hat h_{j_3}(\xi-\eta_1-\eta_2)\diff \eta_1\diff \eta_2,
\end{align}
where
$\Phi$ is defined in \eqref{defPhi} and 
\begin{align*}
\partial_\xi\Phi(\eta_1,\eta_2,\xi)=-\frac{1}{[1+(\xi-\eta_1-\eta_2)^2]^{3/2}}+\frac{1}{(1+\xi^2)^{3/2}},\\
\partial_{\eta_1}\Phi(\eta_1,\eta_2,\xi)=-\frac{1}{(1+\eta_1^2)^{3/2}}+\frac{1}{[1+(\xi-\eta_1-\eta_2)^2]^{3/2}}.
\end{align*}
By direct calculation
\begin{equation}\label{eqn7_2}
\begin{aligned}
&\partial_{\eta_1}\left[\Tb_1(\eta_1,\eta_2,\xi-\eta_1-\eta_2)\frac{\partial_{\xi}\Phi}{i\partial_{\eta_1}\Phi} \right]\psi_{j_1}(\eta_1)\psi_{j_2}(\eta_2)\psi_{j_3}(\xi-\eta_1-
\eta_2)\\
=~&(\partial_{1}\Tb_1-\partial_3\Tb_1)(\eta_1,\eta_2,\xi-\eta_1-\eta_2)\frac{\partial_{\xi}\Phi}{i\partial_{\eta_1}\Phi} \psi_{j_1}(\eta_1)\psi_{j_2}(\eta_2)\psi_{j_3}(\xi-\eta_1-
\eta_2)\\
&+\Tb_1(\eta_1,\eta_2,\xi-\eta_1-\eta_2)\frac{\partial_{\eta_1}\partial_{\xi}\Phi}{i\partial_{\eta_1}\Phi} \psi_{j_1}(\eta_1)\psi_{j_2}(\eta_2)\psi_{j_3}(\xi-\eta_1-
\eta_2)\\
&-\Tb_1(\eta_1,\eta_2,\xi-\eta_1-\eta_2)\frac{\partial_{\eta_1}^2\Phi\partial_{\xi}\Phi}{i(\partial_{\eta_1}\Phi)^2} \psi_{j_1}(\eta_1)\psi_{j_2}(\eta_2)\psi_{j_3}(\xi-\eta_1-
\eta_2).
\end{aligned}
\end{equation}
The first term on the right-hand-side of equation \eqref{eqn7_1} is a Fourier multilinear integral with the symbol 
\begin{align*}
&-i\psi_{j_1}(\eta_1)\psi_{j_2}(\eta_2)\psi_{j_3}(\eta_3)\\
&\cdot\Big[(\partial_{1}\Tb_1-\partial_3\Tb_1)(\eta_1,\eta_2, \eta_3)\frac{\frac{1}{[1+\eta_3^2]^{3/2}}-\frac{1}{(1+(\eta_1+\eta_2+\eta_3)^2)^{3/2}}}{\frac{1}{(1+\eta_1^2)^{3/2}}-\frac{1}{[1+\eta_3^2]^{3/2}}} 
+\Tb_1(\eta_1,\eta_2, \eta_3)\frac{\frac{3\eta_3}{2[1+\eta_3^2]^{5/2}}}{\frac{1}{(1+\eta_1^2)^{3/2}}-\frac{1}{[1+\eta_3^2]^{3/2}}}\\
&\qquad -\Tb_1(\eta_1,\eta_2, \eta_3)\frac{\left[-\frac{3\eta_1}{(1+\eta_1^2)^{5/2}}-\frac{3\eta_3}{(1+\eta_3^2)^{5/2}}\right]\left[\frac{1}{[1+\eta_3^2]^{3/2}}-\frac{1}{(1+(\eta_1+\eta_2+\eta_3)^2)^{3/2}}\right]}{\left[\frac{1}{(1+\eta_1^2)^{3/2}}-\frac{1}{[1+\eta_3^2]^{3/2}}\right]^2}\Big].
\end{align*}
By direct calculation, we have the following estimates for the symbols:
\begin{align}
\label{symbolphixi}
&\left\|\psi_{j_1}(\eta_1)\psi_{j_2}(\eta_2)\psi_{j_3}(\eta_3)\left[\frac{1}{[1+\eta_3^2]^{3/2}}-\frac{1}{(1+(\eta_1+\eta_2+\eta_3)^2)^{3/2}}\right]\right\|_{S^\infty}
\hskip-0.2in
\lesssim 
2^{2j_1}(1+2^{2j_1})^{-1}(1+2^{3j_3})^{-1},
\end{align}
\begin{align}\label{symbolpartialphi}
\left\|
\psi_{j_1}(\eta_1)\psi_{j_2}(\eta_2)\psi_{j_3}(\eta_3)
\left[\frac{1}{(1+\eta_1^2)^{3/2}}-\frac{1}{[1+\eta_3^2]^{3/2}} \right]^{-1}\right\|_{S^\infty}
\lesssim& 
(1+2^{3j_3})(1+2^{-2j_1}),
\end{align}
\begin{align}\label{symbol2d}
\left\|\psi_{j_1}(\eta_1)\psi_{j_2}(\eta_2)\psi_{j_3}(\eta_3)\left[-\frac{3\eta_1}{(1+\eta_1^2)^{5/2}}-\frac{3\eta_3}{(1+\eta_3^2)^{5/2}}\right]\right\|_{S^\infty}
\lesssim&
2^{j_1}(1+2^{j_1})^{-1}(1+2^{j_3})^{-4}
,
\end{align}
\begin{align*}
\left\|\psi_{j_1}(\eta_1)\psi_{j_2}(\eta_2)\psi_{j_3}(\eta_3)\left[\frac{3\eta_3}{(1+\eta_3^2)^{5/2}}\right]\right\|_{S^\infty}\lesssim 2^{j_3}(1+2^{j_3})^{-5}.
\end{align*}
By the algebraic property of $S^\infty$ norm, combining the above estimates and Proposition \ref{Prop_symb_Tmu}, 
\begin{align*}
&\left\|\Big\{\partial_{\eta_1}[\Tb_1(\eta_1,\eta_2,\xi-\eta_1-\eta_2)\frac{\partial_{\xi}\Phi}{i\partial_{\eta_1}\Phi}]\Big\}_{\xi=\eta_1+\eta_2+\eta_3}\psi_{j_1}(\eta_1)\psi_{j_2}(\eta_2)\psi_{j_3}(\eta_3)\right\|_{S^\infty}\\
\lesssim~&  [(1+2^{j_3})(1+2^{5j_2})2^{j_2+j_3}][1+2^{3j_3}](1+2^{3j_3})^{-1}\\
& +[(1+2^{j_3})(1+2^{3j_2})2^{j_2+j_3}][2^{j_3}(1+2^{j_3})^{-1}][(1+2^{3j_3})(1+2^{-2j_1})]\\
&+[(1+2^{j_3})(1+2^{3j_2})2^{j_2+j_3}][2^{j_1}(1+2^{j_1})^{-1}][2^{2j_1}(1+2^{2j_1})^{-1}(1+2^{3j_3})^{-1}][(1+2^{3j_3})(1+2^{-2j_1})]^2\\
\lesssim~&2^{j_2+j_3}(1+2^{5j_3})(1+2^{5j_2})(1+2^{-j_1}).
\end{align*}
The first term of equation \eqref{eqn7_1} satisfies
\begin{align*}
\left\|\frac13\xi\iint_{\R^2}\partial_{\eta_1}\left[\Tb_1(\eta_1,\eta_2,\xi-\eta_1-\eta_2)\frac{\partial_{\xi}\Phi}{i\partial_{\eta_1}\Phi} \right]e^{it\Phi} \hat h_{j_1}(\eta_1)\hat h_{j_2}(\eta_2)\hat h_{j_3}(\xi-\eta_1-\eta_2)\diff \eta_1\diff \eta_2\right\|_{L^2}\\
\lesssim \|\vp_{j_1}\|_{H^{1}}\|\vp_{j_2}\|_{B^{1,6}}\|\vp_{j_3}\|_{B^{1,6}}.
\end{align*}
The second term of equation \eqref{eqn7_1} is
\begin{align*}
&\xi\iint_{\R^2}\Tb_1(\eta_1,\eta_2,\xi-\eta_1-\eta_2)\frac{\partial_{\xi}\Phi}{i\partial_{\eta_1}\Phi} e^{it\Phi} \partial_{\eta_1}\hat h_{j_1}(\eta_1)\hat h_{j_2}(\eta_2)\hat h_{j_3}(\xi-\eta_1-\eta_2)\diff \eta_1\diff \eta_2.
\end{align*}
For this term, we consider the following two cases:
Case 1: $j_2\le j_1\le j_2+1$,
and
Case 2: $j_2< j_1-1$.

\vskip.1in
\noindent
{\bf Estimate in Case 1.} Since $|\xi|\approx 2^{j_2}\approx 2^{j_1}>2^{j_3+1}$, using Proposition \ref{Prop_symb_Tmu}
 we have 
\begin{align*}
&\left\|
\psi_{j_1}(\eta_1)\psi_{j_2}(\eta_2)\psi_{j_3}(\eta_3)
(\eta_1+\eta_2+\eta_3)\Tb_1(\eta_1,\eta_2,\eta_3)\frac{\partial_{\xi}\Phi}{i\partial_{\eta_1}\Phi}
\right\|_{S^{\infty}}
\\
\lesssim~&
2^{j_2}
\|
\psi_{j_1}(\eta_1)\psi_{j_2}(\eta_2)\psi_{j_3}(\eta_3)
\Tb_1(\eta_1,\eta_2,\eta_3)
\|_{S^{\infty}}
\lesssim
 2^{(2j_2+j_3)}(1+2^{3 j_2}) (1+2^{j_3}),
\end{align*}
which implies that 
\begin{align*}
&
\Big|\Re\int_{\R} \overline{[\partial_\xi \hat h(\xi)]}\iint_{\R^2}\xi\Tb_1(\eta_1,\eta_2,\xi-\eta_1-\eta_2)\frac{\partial_{\xi}\Phi}{i\partial_{\eta_1}\Phi} e^{it\Phi} \psi_{j_1}(\eta_1)\partial_{\eta_1}\hat h(\eta_1)\hat h_{j_2}(\eta_2)\hat h_{j_3}(\xi-\eta_1-\eta_2)\diff \eta_1\diff \eta_2\diff\xi
\Big|
\\
\lesssim~&
 \|\psi_{j_1}(\xi)\partial_\xi\hat h(\xi)\|_{L^2}\|\tilde\psi_{j_1}(\xi)\partial_\xi\hat h(\xi)\|_{L^2}\|\vp_{j_2}\|_{B^{1,6}}\|\vp_{j_3}\|_{B^{1,6}}.
\end{align*}
{\bf Estimate in Case 2.} In this case, the second term of equation \eqref{eqn7_1}  has the trouble of loss of derivatives in the high frequency part. One can transfer one-order of the high frequency derivative to the low frequency by doing the symmetrization. Thus we need to estimate
\begin{align*}
&\Re\int_{\R} \overline{[\partial_\xi \hat h(\xi)]}\iint_{\R^2}\xi\Tb_1(\eta_1,\eta_2,\xi-\eta_1-\eta_2)\frac{\partial_{\xi}\Phi}{i\partial_{\eta_1}\Phi} e^{it\Phi} \psi_{j_1}(\eta_1)\partial_{\eta_1}\hat h(\eta_1)\hat h_{j_2}(\eta_2)\hat h_{j_3}(\xi-\eta_1-\eta_2)\diff \eta_1\diff \eta_2\diff\xi
\\
=~&-\Im \iiint_{\R^3}\xi\Tb_1(\eta_1,\eta_2,\xi-\eta_1-\eta_2)\frac{p'(\xi-\eta_1-\eta_2)-p'(\xi)}{p'(\eta_1)-p'(\xi-\eta_1-\eta_2)} \psi_{j_1}(\eta_1)\\
& \cdot e^{it\Phi(\eta_1,\eta_2,\xi)}\overline{[\partial_\xi \hat h(\xi)]}\partial_{\eta_1}\hat h(\eta_1)\hat h_{j_2}(\eta_2)\hat h_{j_3}(\xi-\eta_1-\eta_2)\diff \eta_1\diff \eta_2\diff\xi\\
=~-&\Im\iiint_{\R^3}-\eta_1 \Tb_1(-\xi,\eta_2,\xi-\eta_1-\eta_2)\frac{p'(\eta_1)-p'(\xi-\eta_1-\eta_2)} {p'(\xi-\eta_1-\eta_2)-p'(\xi)}\psi_{j_1}(\xi) \\
&\cdot e^{it\Phi(\eta_1,\eta_2,\xi)}\overline{[\partial_\xi \hat h(\xi)]} \partial_{\eta_1}\hat h(\eta_1)\hat h_{j_2}(\eta_2)\hat h_{j_3}(\xi-\eta_1-\eta_2)\diff \eta_1\diff \eta_2\diff\xi\\
=~-&\Im\iiint_{\R^3} m(\eta_1, \eta_2, \xi-\eta_1-\eta_2))e^{it\Phi(\eta_1,\eta_2,\xi)}\overline{[\partial_\xi \hat h(\xi)]} \partial_{\eta_1}\hat h(\eta_1)\hat h_{j_2}(\eta_2)\hat h_{j_3}(\xi-\eta_1-\eta_2)\diff \eta_1\diff \eta_2\diff\xi,
\end{align*}
where 
\begin{align*}
&m(\eta_1, \eta_2, \xi-\eta_1-\eta_2)\\
=~&\frac12\Bigg[ -\eta_1\Tb_1(-\xi,\eta_2,\xi-\eta_1-\eta_2)\frac{p'(\eta_1)-p'(\xi-\eta_1-\eta_2)}{p'(\xi-\eta_1-\eta_2)-p'(\xi)}\psi_{j_1}(\xi)\\
&\qquad+ \xi\Tb_1(\eta_1,\eta_2,\xi-\eta_1-\eta_2)\frac{p'(\xi-\eta_1-\eta_2)-p'(\xi)}{p'(\eta_1)-p'(\xi-\eta_1-\eta_2)}\psi_{j_1}(\eta_1)\Bigg]\\
 =~& \frac12  \Bigg[-\eta_1\big(\Tb_1(-\xi,\eta_2,\xi-\eta_1-\eta_2)-\Tb_1(\eta_1,\eta_2,\xi-\eta_1-\eta_2)\big)\frac{p'(\eta_1)-p'(\xi-\eta_1-\eta_2)}{p'(\xi-\eta_1-\eta_2)-p'(\xi)}\psi_{j_1}(\xi)\\
 &~~\quad +\eta_1 \Tb_1(\eta_1,\eta_2,\xi-\eta_1-\eta_2)\left(\frac{p'(\xi-\eta_1-\eta_2)-p'(\xi)}{p'(\eta_1)-p'(\xi-\eta_1-\eta_2)}-\frac{p'(\eta_1)-p'(\xi-\eta_1-\eta_2)}{p'(\xi-\eta_1-\eta_2)-p'(\xi)}\right)\psi_{j_1}(\xi)\\
 &~~\quad +\eta_1 \Tb_1(\eta_1,\eta_2,\xi-\eta_1-\eta_2)\frac{p'(\xi-\eta_1-\eta_2)-p'(\xi)}{p'(\eta_1)-p'(\xi-\eta_1-\eta_2)}(\psi_{j_1}(\eta_1)-\psi_{j_1}(\xi))\\
 &~~\quad+(\xi-\eta_1)\Tb_1(\eta_1,\eta_2,\xi-\eta_1-\eta_2)\frac{p'(\xi-\eta_1-\eta_2)-p'(\xi)}{p'(\eta_1)-p'(\xi-\eta_1-\eta_2)}\psi_{j_1}(\eta_1)\Bigg]\\
 =:~&m_1+m_2+m_3+m_4.
\end{align*}
We will estimate the $S^\infty$ norm of each symbol in the dyadic support.

For the symbol $m_1$, using \eqref{Tmu_est5}, 
\begin{align*}
\|m_1\psi_{j_1}\psi_{j_2}\psi_{j_3}\|_{S^\infty}&\lesssim 2^{j_1}[2^{j_2+j_3}(1+2^{j_1})^{-1}(1+2^{j_2})^3(1+2^{j_3})]\\
&\lesssim 2^{j_2+j_3}(1+2^{j_2})^3(1+2^{j_3}).
\end{align*}

For the symbol $m_2$, we calculate the difference of the fractions
\[
\frac{p'(\eta_3)-p'(\eta_1+\eta_2+\eta_3)}{p'(\eta_1)-p'(\eta_3)}-\frac{p'(\eta_1)-p'(\eta_3)}{p'(\eta_3)-p'(\eta_1+\eta_2+\eta_3)}
=:\frac{\mathcal{A}}{\mathcal{B}},
\] 
where
\begin{align*}
\mathcal{A}=~&[(1+(\eta_1+\eta_2+\eta_3)^2)^{3/2}-(1+\eta_3^2)^{3/2}]^2(1+\eta_1^2)^{3}-[(1+\eta_3^2)^{3/2}-(1+\eta_1^2)^{3/2}]^2(1+(\eta_1+\eta_2+\eta_3)^2)^{3}\\
=~&(1+\eta_3^2)^{3}(1+\eta_1^2)^{3}-2(1+(\eta_1+\eta_2+\eta_3)^2)^{3/2}(1+\eta_3^2)^{3/2}(1+\eta_1^2)^{3}\\
&-(1+\eta_3^2)^{3}(1+(\eta_1+\eta_2+\eta_3)^2)^{3}+2(1+\eta_3^2)^{3/2}(1+\eta_1^2)^{3/2}(1+(\eta_1+\eta_2+\eta_3)^2)^{3}\\
=~&(1+\eta_3^2)^{3}[(1+\eta_1^2)^{3}-(1+(\eta_1+\eta_2+\eta_3)^2)^{3}]\\
&+2(1+\eta_3^2)^{3/2}(1+\eta_1^2)^{3/2}(1+(\eta_1+\eta_2+\eta_3)^2)^{3/2}[(1+(\eta_1+\eta_2+\eta_3)^2)^{3/2}-(1+\eta_1^2)^{3/2}],
\end{align*}
\begin{align*}
\mathcal{B}=~[(1+(\eta_1+\eta_2+\eta_3)^2)^{3/2}-(1+\eta_3^2)^{3/2}](1+\eta_1^2)^{3/2}[(1+\eta_3^2)^{3/2}-(1+\eta_1^2)^{3/2}](1+(\eta_1+\eta_2+\eta_3)^2)^{3/2}.
\end{align*}
It's easy to estimate the $S^\infty$ norm of $\mathcal{A}$ and $\mathcal{B}$ as
\[
\|\mathcal{A}\psi_{j_1}\psi_{j_2}\psi_{j_3}\|_{S^\infty}\lesssim 2^{2j_1+2j_2}(1+2^{j_1})^7(1+2^{j_2}),
\]
\[
\|\mathcal{B}\psi_{j_1}\psi_{j_2}\psi_{j_3}\|_{S^\infty}\approx (1+2^{j_1})^{8}2^{4j_1}.
\]
Therefore, by algebraic property of $S^\infty$ norm,
\[
\Big\|\frac{\mathcal{A}}{\mathcal{B}}\psi_{j_1}\psi_{j_2}\psi_{j_3}\Big\|_{S^\infty}\lesssim (1+2^{j_1})^{-1}(1+2^{j_2}).
\]
By above estimate and proposition \ref{Prop_symb_Tmu},
\begin{align*}
\|m_2\psi_{j_1}\psi_{j_2}\psi_{j_3}\|_{S^\infty}&\lesssim 2^{j_1}[2^{j_1+j_2+j_3}(1+2^{j_1})^{-1}(1+2^{j_2})^3(1+2^{j_3})][(1+2^{j_1})^{-1}(1+2^{j_2})]\\
&\lesssim 2^{j_2+j_3}(1+2^{j_2})^4(1+2^{j_3}).
\end{align*}
For the symbol $m_3, m_4$, they satisfy the estimates
\begin{align*}
\|m_3\psi_{j_1}\psi_{j_2}\psi_{j_3}\|_{S^\infty}&\lesssim 2^{j_1}[2^{j_1+j_2+j_3}(1+2^{j_1})^{-1}(1+2^{j_2})^3(1+2^{j_3})]2^{-j_1}\\
&\lesssim 2^{j_2+j_3}(1+2^{j_2})^3(1+2^{j_3}),
\end{align*}
\begin{align*}
\|m_4\psi_{j_1}\psi_{j_2}\psi_{j_3}\|_{S^\infty}&\lesssim 2^{j_2}[2^{j_1+j_2+j_3}(1+2^{j_1})^{-1}(1+2^{j_2})^3(1+2^{j_3})]\\
&\lesssim 2^{2j_2+j_3}(1+2^{j_2})^3(1+2^{j_3}).
\end{align*}
Therefore the $S^\infty$-norm of $m(\eta_1, \eta_2, \eta_3)\psi_{j_1}(\eta_1)\psi_{j_2}(\eta_2)\psi_{j_3}(\eta_3)$ satisfies
\begin{align*}
&\|m\psi_{j_1}\psi_{j_2}\psi_{j_3}\|_{S^\infty}\lesssim 2^{j_2+j_3}(1+2^{j_2})^4(1+2^{j_3}).
\end{align*}
The integral is bounded by
\begin{align*}
&\left|\Re\int_{\R} \overline{[\partial_\xi \hat h(\xi)]}\iint_{\R^2}\Tb_1(\eta_1,\eta_2,\xi-\eta_1-\eta_2)\frac{\partial_{\xi}\Phi}{i\partial_{\eta_1}\Phi} e^{it\Phi} \psi_{j_1}(\eta_1)\partial_{\eta_1}\hat h(\eta_1)\hat h_{j_2}(\eta_2)\hat h_{j_3}(\xi-\eta_1-\eta_2)\diff \eta_1\diff \eta_2\diff\xi\right|\\
\lesssim~& \|\psi_{j_1}(\xi)\partial_\xi\hat h(\xi)\|_{L^2}\|\tilde\psi_{j_1}(\xi)\partial_\xi\hat h(\xi)\|_{L^2}\|\vp_{j_2}\|_{B^{1,6}}\|\vp_{j_3}\|_{B^{1,6}}.
\end{align*}
The other term for $\rm{III}_{1}$ estimate in \eqref{7.5} is
\begin{align*}
&\left|\Re\int_{\R} \overline{\rm{VI}(\xi)}\iint_{\R^2}\xi \Tb_1(\eta_1,\eta_2,\xi-\eta_1-\eta_2)\frac{\partial_{\xi}\Phi}{i\partial_{\eta_1}\Phi} e^{it\Phi} \psi_{j_1}(\eta_1)\partial_{\eta_1}\hat h(\eta_1)\hat h_{j_2}(\eta_2)\hat h_{j_3}(\xi-\eta_1-\eta_2)\diff \eta_1\diff \eta_2\diff\xi\right|\\
\lesssim~& (1+2^{j_1})\|\tilde \psi_{j_1}(\xi){\rm{VI}}\|_{L^2}\|\tilde\psi_{j_1}\partial_\xi\hat h\|_{L^2}\|\vp_{j_2}\|_{B^{1,6}}\|\vp_{j_3}\|_{B^{1,6}}.
\end{align*}
The fourth term in \eqref{eqn7_1} satisfies
\begin{align*}
&\left|\iiint_{\R^3} \frac13\overline{[\partial_\xi \hat h(\xi)]}\xi\Tb_1(\eta_1,\eta_2,\xi-\eta_1-\eta_2)\frac{\partial_{\xi}\Phi}{i\partial_{\eta_1}\Phi} e^{it\Phi} \hat h_{j_1}(\eta_1)\hat h_{j_2}(\eta_2)\partial_{\eta_1}\hat h_{j_3}(\xi-\eta_1-\eta_2)\diff \eta_1\diff \eta_2\diff \xi \right|\\
\lesssim~& \|\partial_\xi\hat h_{j_3}\|_{L^2_\xi}\|\partial_\xi\hat h \tilde \psi_{j_1}\|_{L^2} \|\vp_{j_2}\|_{B^{1,6}}\|\vp_{j_1}\|_{B^{1,6}}.
\end{align*}

\subsubsection{Term $\rm{III}_{21}$ estimate}
We will estimate the $L^2$-norm of
\begin{align*}
\frac \xi3 t \iint_{\R^2}\Tb_1(\eta_1,\eta_2,\xi-\eta_1-\eta_2)\partial_{\xi}\Phi e^{it\Phi} \hat h_{j_1}^{\iota_1}(\eta_1)\hat h_{j_2}^{\iota_2}(\eta_2)\hat h_{j_3}(\xi-\eta_1-\eta_2)\upsilon_{\iota_3}(\xi)\diff \eta_1\diff \eta_2,
\end{align*}
when $(\iota_1, \iota_2, \iota_3)=(+,+,-)$ or $(-, -, +)$ and $j_1, j_2, j_3\in \P_2$.

In the support of the integrand, $|\eta_1^2-(\xi-\eta_1-\eta_2)^2|\gtrsim \min\{1, 2^{2j_1}\}$. This is away from resonances.
By the same operation as \eqref{eqn7_1} and  \eqref{eqn7_2}, we need to estimate
\begin{equation}\label{eqn714}
\begin{aligned}
&\frac\xi3  \iint_{\R^2}\partial_{\eta_1}[\Tb_1(\eta_1,\eta_2,\xi-\eta_1-\eta_2)\frac{\partial_{\xi}\Phi}{i\partial_{\eta_1}\Phi} ]e^{it\Phi} \hat h_{j_1}^{\iota_1}(\eta_1)\hat h_{j_2}^{\iota_2}(\eta_2)\hat h_{j_3}(\xi-\eta_1-\eta_2)\upsilon_{\iota_3}(\xi)\diff \eta_1\diff \eta_2\\
&+\frac\xi3  \iint_{\R^2}\Tb_1(\eta_1,\eta_2,\xi-\eta_1-\eta_2)\frac{\partial_{\xi}\Phi}{i\partial_{\eta_1}\Phi} e^{it\Phi} \partial_{\eta_1}\hat h_{j_1}^{\iota_1}(\eta_1)\hat h_{j_2}^{\iota_2}(\eta_2)\hat h_{j_3}(\xi-\eta_1-\eta_2)\upsilon_{\iota_3}(\xi)\diff \eta_1\diff \eta_2\\
&+\frac\xi3  \iint_{\R^2}\Tb_1(\eta_1,\eta_2,\xi-\eta_1-\eta_2)\frac{\partial_{\xi}\Phi}{i\partial_{\eta_1}\Phi} e^{it\Phi} \hat h_{j_1}^{\iota_1}(\eta_1)\hat h_{j_2}^{\iota_2}(\eta_2)\partial_{\eta_1}\hat h_{j_3}(\xi-\eta_1-\eta_2)\diff \eta_1\diff \eta_2.
\end{aligned}\end{equation}
The symbols satisfy
\begin{align*}
&\left\|\psi_{j_1}(\eta_1)\psi_{j_2}(\eta_2)\psi_{j_3}(\eta_3)\upsilon_{\iota_1}(\eta_1)\upsilon_{\iota_2}(\eta_2)\upsilon_{\iota_3}(\eta_1+\eta_2+\eta_3)\left[\frac{1}{[1+\eta_3^2]^{3/2}}-\frac{1}{(1+(\eta_1+\eta_2+\eta_3)^2)^{3/2}}\right]\right\|_{S^\infty}\\
\lesssim~& 2^{2j_1}(1+2^{5j_1})^{-1},
\end{align*}
\begin{align}\nonumber
&\left\|\psi_{j_1}(\eta_1)\psi_{j_2}(\eta_2)\psi_{j_3}(\eta_3)\upsilon_{\iota_1}(\eta_1)\upsilon_{\iota_2}(\eta_2)\upsilon_{\iota_3}(\eta_1+\eta_2+\eta_3)\left[\frac{1}{(1+\eta_1^2)^{3/2}}-\frac{1}{[1+\eta_3^2]^{3/2}} \right]^{-1}\right\|_{S^\infty}\\\label{eqn715}
\lesssim~ &(1+2^{2j_1})^{5/2}\max\{1, 2^{-2j_1}\}\lesssim (1+2^{j_1})^{7}2^{-2j_1},
\end{align}
\begin{align}\label{eqn716}
\left\|\psi_{j_1}(\eta_1)\psi_{j_2}(\eta_2)\psi_{j_3}(\eta_3)\upsilon_{\iota_1}(\eta_1)\upsilon_{\iota_2}(\eta_2)\upsilon_{\iota_3}(\eta_1+\eta_2+\eta_3)\left[-\frac{3\eta_1}{(1+\eta_1^2)^{5/2}}-\frac{3\eta_3}{(1+\eta_3^2)^{5/2}}\right]\right\|_{S^\infty}\lesssim 2^{j_1}(1+2^{j_1})^{-5},
\end{align}
\begin{align}\label{eqn717}
\left\|\psi_{j_1}(\eta_1)\psi_{j_2}(\eta_2)\psi_{j_3}(\eta_3)\upsilon_{\iota_1}(\eta_1)\upsilon_{\iota_2}(\eta_2)\upsilon_{\iota_3}(\eta_1+\eta_2+\eta_3)\left[\frac{3\eta_3}{(1+\eta_3^2)^{5/2}}\right]\right\|_{S^\infty}\lesssim 2^{j_1}(1+2^{j_1})^{-5}.
\end{align}
Thus
\begin{align*}
&\left\|\frac\xi3  \iint_{\R^2}\partial_{\eta_1}\left[\Tb_1(\eta_1,\eta_2,\xi-\eta_1-\eta_2)\frac{\partial_{\xi}\Phi}{i\partial_{\eta_1}\Phi} \right]e^{it\Phi} \hat h_{j_1}^{\iota_1}(\eta_1)\hat h_{j_2}^{\iota_2}(\eta_2)\hat h_{j_3}(\xi-\eta_1-\eta_2)\upsilon_{\iota_3}(\xi)\diff \eta_1\diff \eta_2\right\|_{L^2}\\
\lesssim~& \left\|\frac\xi3  \iint_{\R^2}\partial_{\eta_1}\Tb_1(\eta_1,\eta_2,\xi-\eta_1-\eta_2)\frac{\partial_{\xi}\Phi}{i\partial_{\eta_1}\Phi} e^{it\Phi}\hat h_{j_1}^{\iota_1}(\eta_1)\hat h_{j_2}^{\iota_2}(\eta_2)\hat h_{j_3}(\xi-\eta_1-\eta_2)\upsilon_{\iota_3}(\xi)\diff \eta_1\diff \eta_2\right\|_{L^2}\\
& +\left\|\frac\xi3  \iint_{\R^2}\Tb_1(\eta_1,\eta_2,\xi-\eta_1-\eta_2)\frac{\partial_{\xi}\Phi \partial_{\eta_1}^2\Phi}{i(\partial_{\eta_1}\Phi)^2} e^{it\Phi}\hat h_{j_1}^{\iota_1}(\eta_1)\hat h_{j_2}^{\iota_2}(\eta_2)\hat h_{j_3}(\xi-\eta_1-\eta_2)\upsilon_{\iota_3}(\xi)\diff \eta_1\diff \eta_2\right\|_{L^2}\\
& +\left\|\frac\xi3  \iint_{\R^2}\Tb_1(\eta_1,\eta_2,\xi-\eta_1-\eta_2)\frac{\partial_{\xi}\partial_{\eta_1}\Phi}{i\partial_{\eta_1}\Phi} e^{it\Phi} \hat h_{j_1}^{\iota_1}(\eta_1)\hat h_{j_2}^{\iota_2}(\eta_2)\hat h_{j_3}(\xi-\eta_1-\eta_2)\upsilon_{\iota_3}(\xi)\diff \eta_1\diff \eta_2\right\|_{L^2}\\
\lesssim~&2^{j_1}\Big\{[2^{2j_1} (1+2^{j_1})^6][(1+2^{2j_1})]\\
&+[2^{3j_1}(1+2^{j_1})^3][2^{2j_1}(1+2^{5j_1})^{-1}][(1+2^{7j_1})2^{-2j_1}]^2[2^{j_1}(1+2^{5j_1})^{-1}]\\
&+[2^{3j_1}(1+2^{j_1})^3] [(1+2^{7j_1})2^{-2j_1}][2^{j_1}(1+2^{5j_1})^{-1}]\Big\}\cdot\|\vp_{j_1}\|_{L^2}\|\vp_{j_2}\|_{L^\infty}\|\vp_{j_3}\|_{L^\infty}\\
\lesssim~&\|\vp_{j_1}\|_{L^2}\|\vp_{j_2}\|_{B^{1,6}}\|\vp_{j_3}\|_{B^{1,6}}.
\end{align*}

Similarly, the other two integrals satisfy
\begin{align*}
&\left\|\frac\xi3  \iint_{\R^2}\Tb_1(\eta_1,\eta_2,\xi-\eta_1-\eta_2)\frac{\partial_{\xi}\Phi}{i\partial_{\eta_1}\Phi} e^{it\Phi} \partial_{\eta_1}\hat h_{j_1}^{\iota_1}(\eta_1)\hat h_{j_2}^{\iota_2}(\eta_2)\hat h_{j_3}(\xi-\eta_1-\eta_2)\upsilon_{\iota_3}(\xi)\diff \eta_1\diff \eta_2\right\|_{L^2}\\
&+\left\|\frac\xi3  \iint_{\R^2}\Tb_1(\eta_1,\eta_2,\xi-\eta_1-\eta_2)\frac{\partial_{\xi}\Phi}{i\partial_{\eta_1}\Phi} e^{it\Phi} \hat h_{j_1}^{\iota_1}(\eta_1)\hat h_{j_2}^{\iota_2}(\eta_2)\partial_{\eta_1}\hat h_{j_3}(\xi-\eta_1-\eta_2)\diff \eta_1\diff \eta_2\right\|_{L^2}\\
\lesssim~&2^{j_1}[2^{3j_1}(1+2^{j_1})^3][(1+2^{2j_1})]
\big(\|\partial_\xi\hat h_{j_1}(\xi)\|_{L^2_\xi}\|\vp_{j_2}\|_{L^\infty}\|\vp_{j_3}\|_{L^\infty}+\|\partial_\xi\hat h_{j_3}(\xi)\|_{L^2_\xi}\|\vp_{j_2}\|_{L^\infty}\|\vp_{j_1}\|_{L^\infty}\big)\\
\lesssim~&\|\partial_{\xi}\hat h_{j_1}(\xi)\|_{L^2_\xi}\|\vp_{j_2}\|_{B^{1,6}}\|\vp_{j_3}\|_{B^{1,6}}+\|\partial_{\xi}\hat h_{j_3}\|_{L^2_\xi}\|\vp_{j_2}\|_{B^{1,6}}\|\vp_{j_1}\|_{B^{1,6}}.
\end{align*}

Then after taking the summation for the related sub-indices $(j_1, j_2, j_3)\in\P_2$, we obtain
\begin{align*}
\|{\rm{III}}_{21}\|_{L^2}\lesssim \|\vp\|_{B^{1,6}}^2(\|\vp\|_{L^2}+\|\partial_\xi\hat h(\xi)\|_{L^2_\xi}).
\end{align*}

\subsubsection{Term $\rm{III}_{22}$ estimate}
In this case, $(\eta_1, \eta_2)$ is around $(\xi, -\xi)$ or $(-\xi, \xi)$. Since $\eta_1$ and $\eta_2$ are symmetric in the integral, we assume $(\eta_1, \eta_2)$ is around $(\xi, -\xi)$ in the following. Under this assumption, $0\leq |\eta_1+\eta_2|\lesssim 2^{j_1}$, $|\eta_1-\eta_2|\approx 2^{j_1}$.

Since 
\begin{equation}\label{eqn718}
(\partial_{\eta_1}-\partial_{\eta_2})e^{it\Phi(\eta_1, \eta_2, \xi)}=(p'(\eta_1)-p'(\eta_2))ite^{it\Phi(\eta_1, \eta_2, \xi)},
\end{equation}
we have
\begin{align*}
&-\frac\xi3 t \iint_{\R^2}\Tb_1(\eta_1,\eta_2,\xi-\eta_1-\eta_2)\partial_{\xi}\Phi e^{it\Phi} \hat h_{j_1}^{\iota_1}(\eta_1)\hat h_{j_2}^{\iota_2}(\eta_2)\hat h_{j_3}(\xi-\eta_1-\eta_2)\upsilon_{\iota_3}(\xi)\diff \eta_1\diff \eta_2\\
=~&\frac {i\xi}3  \iint_{\R^2}\Tb_1(\eta_1,\eta_2,\xi-\eta_1-\eta_2)\partial_{\xi}\Phi \frac{1}{p'(\eta_1)-p'(\eta_2)}(\partial_{\eta_1}-\partial_{\eta_2})e^{it\Phi} \hat h_{j_1}^{\iota_1}(\eta_1)\hat h_{j_2}^{\iota_2}(\eta_2)\hat h_{j_3}(\xi-\eta_1-\eta_2)\upsilon_{\iota_3}(\xi)\diff \eta_1\diff \eta_2\\\nonumber
=~&\frac {i\xi}3  \iint_{\R^2}e^{it\Phi} (\partial_{\eta_1}-\partial_{\eta_2})\left[\Tb_1(\eta_1,\eta_2,\xi-\eta_1-\eta_2) \frac{p'(\xi-\eta_1-\eta_2)-p'(\xi)}{p'(\eta_1)-p'(\eta_2)}\hat h_{j_1}^{\iota_1}(\eta_1)\hat h_{j_2}^{\iota_2}(\eta_2)\hat h_{j_3}(\xi-\eta_1-\eta_2)\right]\nonumber\\
&\qquad\cdot\upsilon_{\iota_3}(\xi)\diff \eta_1\diff \eta_2\\\nonumber
=~&\frac {i\xi}3  \iint_{\R^2}e^{it\Phi} \Big[(\partial_1-\partial_2)\Tb_1(\eta_1,\eta_2,\xi-\eta_1-\eta_2) \frac{p'(\xi-\eta_1-\eta_2)-p'(\xi)}{p'(\eta_1)-p'(\eta_2)}\\
&+\Tb_1(\eta_1,\eta_2,\xi-\eta_1-\eta_2) \frac{(p''(\eta_1)+p''(\eta_2))\left(p'(\xi-\eta_1-\eta_2)-p'(\xi)\right)}{(p'(\eta_1)-p'(\eta_2))^2}\Big]\hat h_{j_1}^{\iota_1}(\eta_1)\hat h_{j_2}^{\iota_2}(\eta_2)\hat h_{j_3}(\xi-\eta_1-\eta_2)\nonumber\\
&\qquad\cdot\upsilon_{\iota_3}(\xi)\diff \eta_1\diff \eta_2\\\nonumber
&+\frac {i\xi}3  \iint_{\R^2}e^{it\Phi}\Tb_1(\eta_1,\eta_2,\xi-\eta_1-\eta_2) \frac{p'(\xi-\eta_1-\eta_2)-p'(\xi)}{p'(\eta_1)-p'(\eta_2)} (\partial_1-\partial_2)\left[\hat h_{j_1}^{\iota_1}(\eta_1)\hat h_{j_2}^{\iota_2}(\eta_2)\right]\\
&\qquad\cdot\hat h_{j_3}(\xi-\eta_1-\eta_2)\upsilon_{\iota_3}(\xi)\diff \eta_1\diff \eta_2\\\nonumber
=:~& I_1+I_2.
\end{align*}
By estimates \eqref{eqnB7} and \eqref{eqnB8}, we obtain the bounds
\begin{align*}
\|I_1\|_{L^2}&\lesssim  2^{2j_1}(1+2^{j_1})^6  \|\vp_{j_1}\|_{L^\infty}\|\vp_{j_2}\|_{L^\infty}\|\vp_{j_3}\|_{L^2}\\
&\lesssim    \|\vp_{j_1}\|_{B^{1,6}}\|\vp_{j_2}\|_{B^{1,6}}\|\vp_{j_3}\|_{L^2},
\end{align*} 
and
\begin{align*}
\|I_2\|_{L^2}&\lesssim  (2^{4j_1}(1+2^{j_1})^6 )\cdot(\|\partial_\xi\hat h_{j_1}(\xi)\|_{L^2_\xi}\|\vp_{j_2}\|_{L^\infty}\|\vp_{j_3}\|_{L^\infty}+\|\partial_\xi\hat h_{j_2}(\xi)\|_{L^2_\xi}\|\vp_{j_1}\|_{L^\infty}\|\vp_{j_3}\|_{L^\infty})\\
&\lesssim \|\partial_{\xi}\hat h_{j_1}(\xi)\|_{L^2_\xi}\|\vp_{j_2}\|_{B^{1,6}}\|\vp_{j_3}\|_{B^{1,6}}+\|\partial_{\xi}\hat h_{j_2}\|_{L^2_\xi}\|\vp_{j_1}\|_{B^{1,6}}\|\vp_{j_3}\|_{B^{1,6}}.
\end{align*}
Combining the estimate of each term, we obtain
\begin{equation}
\|{\rm{III}}_{22}\|_{L^2}\lesssim \|\vp\|_{B^{1,6}}^2(\|\partial_\xi\hat h(t,\xi)\|_{L^2}+\|\vp\|_{L^2}).
\end{equation}

\subsubsection{Term $\rm{III}_{23}$ estimate}
$\rm{III}_{23}$ is the space-time resonance around $(\xi, \xi)$. We will estimate the $L^2$-norm of
\begin{align*}
\frac \xi3 t \iint_{\R^2}\Tb_1(\eta_1,\eta_2,\xi-\eta_1-\eta_2)\partial_{\xi}\Phi e^{it\Phi} \hat h_{j_1}^{\iota_1}(\eta_1)\hat h_{j_2}^{\iota_2}(\eta_2)\hat h_{j_3}(\xi-\eta_1-\eta_2)\upsilon_{\iota_3}(\xi)\psi_{j-3}(\eta_1+\eta_2-2\xi)\diff \eta_1\diff \eta_2,
\end{align*}
when $(\iota_1, \iota_2,\iota_3)=(+, +, +)$ or $(-, -, -)$ and $(j_1, j_2, j_3)\in\P_2$.

After integration by part, we split the integral into several terms.
\begin{align*}
&t \iint_{\R^2}\Tb_1(\eta_1,\eta_2,\xi-\eta_1-\eta_2)\partial_{\xi}\Phi e^{it\Phi} \hat h_{j_1}^{\iota_1}(\eta_1)\hat h_{j_2}^{\iota_2}(\eta_2)\hat h_{j_3}(\xi-\eta_1-\eta_2)\upsilon_{\iota_3}(\xi)\psi_{j-3}(\eta_1+\eta_2-2\xi)\diff \eta_1\diff \eta_2\\
=~& t \iint_{\R^2}\Tb_1(\eta_1,\eta_2,\xi-\eta_1-\eta_2)[p'(\xi-\eta_1-\eta_2)-p'(\eta_1)] e^{it\Phi} \hat h_{j_1}^{\iota_1}(\eta_1)\hat h_{j_2}^{\iota_2}(\eta_2)\hat h_{j_3}(\xi-\eta_1-\eta_2)\upsilon_{\iota_3}(\xi)\\
&\qquad\cdot \psi_{j-3}(\eta_1+\eta_2-2\xi)\diff \eta_1\diff \eta_2\\
& +t \iint_{\R^2}\Tb_1(\eta_1,\eta_2,\xi-\eta_1-\eta_2)[p'(\eta_1)-p'(\xi)] e^{it\Phi} \hat h_{j_1}^{\iota_1}(\eta_1)\hat h_{j_2}^{\iota_2}(\eta_2)\hat h_{j_3}(\xi-\eta_1-\eta_2)\upsilon_{\iota_3}(\xi)\\
&\qquad\cdot \psi_{j-3}(\eta_1+\eta_2-2\xi)\diff \eta_1\diff \eta_2\\
=~&i  \iint_{\R^2}\Tb_1(\eta_1,\eta_2,\xi-\eta_1-\eta_2)\partial_{\eta_1} e^{it\Phi} \hat h_{j_1}^{\iota_1}(\eta_1)\hat h_{j_2}^{\iota_2}(\eta_2)\hat h_{j_3}(\xi-\eta_1-\eta_2)\upsilon_{\iota_3}(\xi)\psi_{j-3}(\eta_1+\eta_2-2\xi)\diff \eta_1\diff \eta_2\\
&+i  \iint_{\R^2}\Tb_1(\eta_1,\eta_2,\xi-\eta_1-\eta_2)\frac{p'(\eta_1)-p'(\xi)}{p'(\eta_2)-p'(\xi-\eta_1-\eta_2)} \partial_{\eta_2}e^{it\Phi} \hat h_{j_1}^{\iota_1}(\eta_1)\hat h_{j_2}^{\iota_2}(\eta_2)\hat h_{j_3}(\xi-\eta_1-\eta_2)\upsilon_{\iota_3}(\xi)\\
&\qquad\qquad\psi_{j-3}(\eta_1+\eta_2-2\xi)\diff \eta_1\diff \eta_2\\
=~&i \iint_{\R^2}\partial_{\eta_1}[\psi_{j-3}(\eta_1+\eta_2-2\xi)\Tb_1(\eta_1,\eta_2,\xi-\eta_1-\eta_2)] e^{it\Phi} \hat h_{j_1}^{\iota_1}(\eta_1)\hat h_{j_2}^{\iota_2}(\eta_2)\hat h_{j_3}(\xi-\eta_1-\eta_2)\upsilon_{\iota_3}(\xi)\diff \eta_1\diff \eta_2\\
&+i  \iint_{\R^2}\psi_{j-3}(\eta_1+\eta_2-2\xi)\Tb_1(\eta_1,\eta_2,\xi-\eta_1-\eta_2) e^{it\Phi} \partial_{\eta_1}\hat h_{j_1}^{\iota_1}(\eta_1)\hat h_{j_2}^{\iota_2}(\eta_2)\hat h_{j_3}(\xi-\eta_1-\eta_2)\upsilon_{\iota_3}(\xi)\diff \eta_1\diff \eta_2\\
&+i  \iint_{\R^2}\psi_{j-3}(\eta_1+\eta_2-2\xi)\Tb_1(\eta_1,\eta_2,\xi-\eta_1-\eta_2) e^{it\Phi}\hat h_{j_1}^{\iota_1}(\eta_1)\hat h_{j_2}^{\iota_2}(\eta_2) \partial_{\eta_1}\hat h_{j_3}(\xi-\eta_1-\eta_2)\upsilon_{\iota_3}(\xi)\diff \eta_1\diff \eta_2\\
&-i  \iint_{\R^2}\partial_{\eta_2}[\psi_{j-3}(\eta_1+\eta_2-2\xi)\Tb_1(\eta_1,\eta_2,\xi-\eta_1-\eta_2)]\frac{p'(\eta_1)-p'(\xi)}{p'(\eta_2)-p'(\xi-\eta_1-\eta_2)} \\
&\qquad \cdot e^{it\Phi} \hat h_{j_1}^{\iota_1}(\eta_1)\hat h_{j_2}^{\iota_2}(\eta_2)\hat h_{j_3}(\xi-\eta_1-\eta_2)\upsilon_{\iota_3}(\xi)\diff \eta_1\diff \eta_2\\
&+i  \iint_{\R^2}\Tb_1(\eta_1,\eta_2,\xi-\eta_1-\eta_2)\frac{(p'(\eta_1)-p'(\xi))[p''(\eta_2)+p''(\xi-\eta_1-\eta_2)]}{(p'(\eta_2)-p'(\xi-\eta_1-\eta_2))^2} e^{it\Phi}\\
&\qquad\cdot \hat h_{j_1}^{\iota_1}(\eta_1)\hat h_{j_2}^{\iota_2}(\eta_2)\hat h_{j_3}(\xi-\eta_1-\eta_2)\psi_{j-3}(\eta_1+\eta_2-2\xi)\upsilon_{\iota_3}(\xi)\diff \eta_1\diff \eta_2\\
&-i \iint_{\R^2}\Tb_1(\eta_1,\eta_2,\xi-\eta_1-\eta_2)\frac{p'(\eta_1)-p'(\xi)}{p'(\eta_2)-p'(\xi-\eta_1-\eta_2)} e^{it\Phi} \hat h_{j_1}^{\iota_1}(\eta_1)\psi_{j-3}(\eta_1+\eta_2-2\xi)\upsilon_{\iota_3}(\xi)\\ 
&\qquad\cdot \left[\partial_{\eta_2}\hat h_{j_2}(\eta_2)\hat h_{j_3}(\xi-\eta_1-\eta_2)+h_{j_2}^{\iota_2}(\eta_2)\partial_{\eta_2}\hat h_{j_3}(\xi-\eta_1-\eta_2)\right]\diff \eta_1\diff \eta_2.
\end{align*}
Then we estimate each term. By the symbol estimate Proposition \ref{Prop_symb_Tmu},
\begin{align*}
&\left\|\xi \iint_{\R^2}\psi_{j_3}(\eta_1+\eta_2-2\xi)\partial_{\eta_1}[\Tb_1(\eta_1,\eta_2,\xi-\eta_1-\eta_2)] e^{it\Phi} \hat h_{j_1}^{\iota_1}(\eta_1)\hat h_{j_2}^{\iota_2}(\eta_2)\hat h_{j_3}(\xi-\eta_1-\eta_2)\upsilon_{\iota_3}(\xi)\diff \eta_1\diff \eta_2\right\|_{L^2}\\
\lesssim ~& 2^{3j_1}(1+2^{j_1})^6 \|\vp_{j_1}\|_{L^\infty} \|\vp_{j_2}\|_{L^\infty} \|\vp_{j_3}\|_{L^2}\\
\lesssim ~& \|\vp_{j_1}\|_{B^{1,6}} \|\vp_{j_2}\|_{B^{1,6}} \|\vp_{j_3}\|_{L^2},
\end{align*}
\begin{align*}
&\left\| \xi \iint_{\R^2}\partial_{\eta_1}\psi_{j_3}(\eta_1+\eta_2-2\xi)[\Tb_1(\eta_1,\eta_2,\xi-\eta_1-\eta_2)] e^{it\Phi} \hat h_{j_1}^{\iota_1}(\eta_1)\hat h_{j_2}^{\iota_2}(\eta_2)\hat h_{j_3}(\xi-\eta_1-\eta_2)\upsilon_{\iota_3}(\xi)\diff \eta_1\diff \eta_2\right\|_{L^2}\\
\lesssim ~& 2^{4j_1}(1+2^{j_1})^4 2^{-j_1} \|\vp_{j_1}\|_{L^\infty} \|\vp_{j_2}\|_{L^\infty} \|\vp_{j_3}\|_{L^2}\\
\lesssim ~& \|\vp_{j_1}\|_{B^{1,6}} \|\vp_{j_2}\|_{B^{1,6}} \|\vp_{j_3}\|_{L^2},
\end{align*}
\begin{align*}
&\left\|\xi\iint_{\R^2}\psi_{j-3}(\eta_1+\eta_2-2\xi)\Tb_1(\eta_1,\eta_2,\xi-\eta_1-\eta_2) e^{it\Phi} \partial_{\eta_1}\hat h_{j_1}^{\iota_1}(\eta_1)\hat h_{j_2}^{\iota_2}(\eta_2)\hat h_{j_3}(\xi-\eta_1-\eta_2)\upsilon_{\iota_3}(\xi)\diff \eta_1\diff \eta_2\right\|_{L^2}\\
\lesssim~&2^{4j_1}(1+2^{j_1})^4  \|\vp_{j_2}\|_{L^\infty} \|\vp_{j_3}\|_{L^\infty} \|\partial_{\xi}\hat h_{j_1}(\xi)\|_{L^2_\xi}\\
\lesssim~& \|\vp_{j_2}\|_{B^{1,6}} \|\vp_{j_3}\|_{B^{1,6}} \|\partial_{\xi}\hat h_{j_1}(\xi)\|_{L^2_\xi},
\end{align*}
\begin{align*}
&\left\|\xi\iint_{\R^2}\psi_{j-3}(\eta_1+\eta_2-2\xi)\Tb_1(\eta_1,\eta_2,\xi-\eta_1-\eta_2) e^{it\Phi}\hat h_{j_1}^{\iota_1}(\eta_1)\hat h_{j_2}^{\iota_2}(\eta_2) \partial_{\eta_1}\hat h_{j_3}(\xi-\eta_1-\eta_2)\upsilon_{\iota_3}(\xi)\diff \eta_1\diff \eta_2\right\|_{L^2}\\
\lesssim~& \|\vp_{j_2}\|_{B^{1,6}} \|\vp_{j_1}\|_{B^{1,6}} \|\partial_{\xi}\hat h_{j_3}(\xi)\|_{L^2_\xi}.
\end{align*}
By Proposition \ref{Prop_symb_Tmu} and estimates \eqref{eqnB9}, 
\begin{align*}
&\Bigg\|\xi\iint_{\R^2}\partial_{\eta_2}\psi_{j-3}(\eta_1+\eta_2-2\xi)\Tb_1(\eta_1,\eta_2,\xi-\eta_1-\eta_2)\frac{p'(\eta_1)-p'(\xi)}{p'(\eta_2)-p'(\xi-\eta_1-\eta_2)} \\
&\qquad \cdot e^{it\Phi} \hat h_{j_1}^{\iota_1}(\eta_1)\hat h_{j_2}^{\iota_2}(\eta_2)\hat h_{j_3}(\xi-\eta_1-\eta_2)\upsilon_{\iota_3}(\xi)\diff \eta_1\diff \eta_2\Bigg\|_{L^2}\\
\lesssim ~& \|\vp_{j_1}\|_{B^{1,6}} \|\vp_{j_2}\|_{B^{1,6}} \|\vp_{j_3}\|_{L^2},
\end{align*}
\begin{align*}
&\Bigg\|\xi\iint_{\R^2}\psi_{j-3}(\eta_1+\eta_2-2\xi)\partial_{\eta_2}\Tb_1(\eta_1,\eta_2,\xi-\eta_1-\eta_2)\frac{p'(\eta_1)-p'(\xi)}{p'(\eta_2)-p'(\xi-\eta_1-\eta_2)} \\
&\qquad \cdot e^{it\Phi} \hat h_{j_1}^{\iota_1}(\eta_1)\hat h_{j_2}^{\iota_2}(\eta_2)\hat h_{j_3}(\xi-\eta_1-\eta_2)\upsilon_{\iota_3}(\xi)\diff \eta_1\diff \eta_2\Bigg\|_{L^2}\\
\lesssim ~& \|\vp_{j_1}\|_{B^{1,6}} \|\vp_{j_2}\|_{B^{1,6}} \|\vp_{j_3}\|_{L^2},
\end{align*}
\begin{align*}
&\Bigg\|\xi\iint_{\R^2}\Tb_1(\eta_1,\eta_2,\xi-\eta_1-\eta_2)\frac{p'(\eta_1)-p'(\xi)}{p'(\eta_2)-p'(\xi-\eta_1-\eta_2)} e^{it\Phi} \hat h_{j_1}^{\iota_1}(\eta_1)\psi_{j-3}(\eta_1+\eta_2-2\xi)\upsilon_{\iota_3}(\xi)\\ 
&\qquad\cdot \left[\partial_{\eta_2}\hat h_{j_2}(\eta_2)\hat h_{j_3}(\xi-\eta_1-\eta_2)+\hat h_{j_2}^{\iota_2}(\eta_2)\partial_{\eta_2}\hat h_{j_3}(\xi-\eta_1-\eta_2)\right]\diff \eta_1\diff \eta_2\Bigg\|_{L^2}\\
\lesssim~&\|\vp_{j_1}\|_{B^{1,6}}( \|\vp_{j_3}\|_{B^{1,6}} \|\partial_{\xi}\hat h_{j_2}(\xi)]\|_{L^2_\xi}+\|\vp_{j_2}\|_{B^{1,6}} \|\partial_{\xi}\hat h_{j_3}(\xi)\|_{L^2_\xi}).
\end{align*}
By Proposition \ref{Prop_symb_Tmu}, estimates \eqref{eqnB9} and \eqref{eqnB10}, 
\begin{align*}
&\Bigg\| \iint_{\R^2}\Tb_1(\eta_1,\eta_2,\xi-\eta_1-\eta_2)\frac{(p'(\eta_1)-p'(\xi))[p''(\eta_2)+p''(\xi-\eta_1-\eta_2)]}{(p'(\eta_2)-p'(\xi-\eta_1-\eta_2))^2} e^{it\Phi}\\
&\qquad\cdot \hat h_{j_1}^{\iota_1}(\eta_1)\hat h_{j_2}^{\iota_2}(\eta_2)\hat h_{j_3}(\xi-\eta_1-\eta_2)\psi_{j-3}(\eta_1+\eta_2-2\xi)\upsilon_{\iota_3}(\xi)\diff \eta_1\diff \eta_2\Bigg\|_{L^2}\\
\lesssim ~&   \|\vp_{j_1}\|_{B^{1,6}} \|\vp_{j_2}\|_{B^{1,6}} \|\vp_{j_3}\|_{L^2}.
\end{align*}
Combining the estimate of each term, we obtain
\begin{equation}
\|{\rm{III}}_{23}\|_{L^2}\lesssim \|\vp\|_{B^{1,6}}^2(\|\partial_\xi\hat h(t,\xi)\|_{L^2}+\|\vp\|_{L^2}).
\end{equation}

\subsubsection{Terms $\rm{III}_{241}$ and $\rm{III}_{242}$ estimates}
By the definition \eqref{III241} and the symbol estimates Proposition \ref{Prop_symb_Tmu} and Proposition \ref{PropB3},
\begin{align*}
\|(1+|\xi|)^r{\rm{III}_{241}}\|_{L^2}&\lesssim \sum\limits_{(j_1,j_2,j_3)\in\P_{2}}2^{2j_1}(1+2^{j_1})^4   \|\vp_{j_1}\|_{L^\infty} \|\vp_{j_2}\|_{L^\infty}\|\hat h_{j_3}(\xi)\|_{L^2_\xi}\\
&\lesssim \|\vp\|_{B^{1,6}}^{2} \|\vp\|_{L^{2}}.
\end{align*}
Since
\begin{align*}
\|\partial_t \hat h_j\|_{L^2}\lesssim& \sum\limits_{\mu=1}^\infty\sum\limits_{j_1, \cdots, j_{2\mu+1}}2^{(j_2+\cdots+j_{2\mu+1})}\prod\limits_{k=2}^{2\mu+1} (1+2^{j_k})2^{2\max\{0, j_2, \cdots, j_{2\mu+1}\}} \|\vp_{j_1}\|_{L^2}\|\vp_{j_2}\|_{L^\infty}\cdots\|\vp_{j_{2\mu+1}}\|_{L^\infty}\\
\lesssim& \sum\limits_{\mu=1}^\infty \|\vp\|_{B^{1,6}}^{2\mu}\|\vp\|_{H^1},
\end{align*}
by the definition \eqref{III242} and the symbol estimates Proposition \ref{Prop_symb_Tmu} and Proposition \ref{PropB3},
\begin{align*}
\|{\rm{III}_{242}}\|_{L^2}&\lesssim t \sum\limits_{(j_1,j_2,j_3)\in\P_{2}}2^{2j_1}(1+2^{j_1})^4 \|\vp_{j_1}\|_{L^\infty} \|\vp_{j_2}\|_{L^\infty}\|\partial_t \hat h_{j_3}\|_{L^2}\\
&\lesssim t\|\vp\|_{B^{1,6}}^2\sum\limits_{\mu=1}^\infty \|\vp\|_{B^{1,6}}^{2\mu} \|\vp\|_{H^{r+1}}.
\end{align*}

\subsection{Term $\rm{V}$ estimate}
Now we estimate the term
\begin{align*}
{\rm{V}}=~&\partial_\xi[\xi e^{-it\frac{\xi}{\sqrt{1+\xi^2}}}\widehat{\Nc_{\geq5}(\vp)}],
\end{align*}
where $e^{-it\frac{\xi}{\sqrt{1+\xi^2}}}\widehat{\Nc_{\geq5}(\vp)}$ is the sum of the following terms with $\mu=2,3, \cdots$,
\begin{align*}
e^{-it\frac{\xi}{\sqrt{1+\xi^2}}}&\widehat{\Nc_{2\mu+1}}(\xi)=\frac1{2\mu+1}i\iiint_{\R^{{2\mu}}} \Tb_{\mu}(\eta_1,\cdots,\eta_{2\mu}, \xi-\eta_1-\cdots-\eta_{2\mu})\\
&\cdot e^{it\left[\frac{\eta_1}{\sqrt{1+\eta_1^2}}+\cdots+\frac{\eta_{2\mu}}{\sqrt{1+\eta_{2\mu}^2}}+\frac{\xi-\eta_1-\cdots-\eta_{2\mu}}{\sqrt{1+(\xi-\eta_1-\cdots-\eta_{2\mu})^2}}-\frac{\xi}{\sqrt{1+\xi^2}}\right]}\hat h(\eta_1)\cdots\hat h(\eta_{2\mu})\hat h(\xi-\eta_1-\cdots-\eta_{2\mu})\diff \eta_1\cdots\diff \eta_{{2\mu}}.
\end{align*}
By dyadic decomposition, this integral can be written as the sum of
\begin{align*}
&\mathfrak N_{2\mu+1}^{j_1\cdots j_{2\mu+1}}(\xi)\\
=~&\frac1{2\mu+1}i\iiint_{\R^{{2\mu}}} \Tb_{\mu}(\eta_1,\cdots,\eta_{2\mu}, \xi-\eta_1-\cdots-\eta_{2\mu})e^{it\left[\frac{\eta_1}{\sqrt{1+\eta_1^2}}+\cdots+\frac{\eta_{2\mu}}{\sqrt{1+\eta_{2\mu}^2}}+\frac{\xi-\eta_1-\cdots-\eta_{2\mu}}{\sqrt{1+(\xi-\eta_1-\cdots-\eta_{2\mu})^2}}-\frac{\xi}{\sqrt{1+\xi^2}}\right]}\\
&\cdot \hat h_{j_1}(\eta_1)\cdots\hat h_{j_{2\mu}}(\eta_{2\mu})\hat h_{j_{2\mu+1}}(\xi-\eta_1-\cdots-\eta_{2\mu})\diff \eta_1\cdots\diff \eta_{{2\mu}},
\end{align*}
for $j_1\ge j_2\ge \cdots\ge j_{2\mu+1}$, with possible repetetion.
Then
\begin{align*}
&\partial_\xi\mathfrak N_{2\mu+1}^{j_1\cdots j_{2\mu+1}}(\xi)\\
=~&
i\iiint_{\R^{{2\mu}}} \partial_{2\mu+1}\Tb_{\mu}(\eta_1,\cdots,\eta_{2\mu}, \xi-\eta_1-\cdots-\eta_{2\mu})e^{it\left[\frac{\eta_1}{\sqrt{1+\eta_1^2}}+\cdots+\frac{\eta_{2\mu}}{\sqrt{1+\eta_{2\mu}^2}}+\frac{\xi-\eta_1-\cdots-\eta_{2\mu}}{\sqrt{1+(\xi-\eta_1-\cdots-\eta_{2\mu})^2}}-\frac{\xi}{\sqrt{1+\xi^2}}\right]}\\
&\qquad\cdot \hat h_{j_1}(\eta_1)\cdots\hat h_{j_{2\mu}}(\eta_{2\mu})\hat h_{j_{2\mu+1}}(\xi-\eta_1-\cdots-\eta_{2\mu})\diff \eta_1\cdots\diff \eta_{{2\mu}}\\
&+i\iiint_{\R^{{2\mu}}} \Tb_{\mu}(\eta_1,\cdots,\eta_{2\mu}, \xi-\eta_1-\cdots-\eta_{2\mu})e^{it\left[\frac{\eta_1}{\sqrt{1+\eta_1^2}}+\cdots+\frac{\eta_{2\mu}}{\sqrt{1+\eta_{2\mu}^2}}+\frac{\xi-\eta_1-\cdots-\eta_{2\mu}}{\sqrt{1+(\xi-\eta_1-\cdots-\eta_{2\mu})^2}}-\frac{\xi}{\sqrt{1+\xi^2}}\right]}\\
&\qquad\cdot(\xi-\eta_1-\cdots-\eta_{2\mu})\hat h_{j_1}(\eta_1)\cdots\hat h_{j_{2\mu}}(\eta_{2\mu})\partial_\xi\frac{\hat h_{j_{2\mu+1}}(\xi-\eta_1-\cdots-\eta_{2\mu})}{\xi-\eta_1-\cdots-\eta_{2\mu}}\diff \eta_1\cdots\diff \eta_{{2\mu}}\\
&+i\iiint_{\R^{{2\mu}}} \Tb_{\mu}(\eta_1,\cdots,\eta_{2\mu}, \xi-\eta_1-\cdots-\eta_{2\mu})e^{it\left[\frac{\eta_1}{\sqrt{1+\eta_1^2}}+\cdots+\frac{\eta_{2\mu}}{\sqrt{1+\eta_{2\mu}^2}}+\frac{\xi-\eta_1-\cdots-\eta_{2\mu}}{\sqrt{1+(\xi-\eta_1-\cdots-\eta_{2\mu})^2}}-\frac{\xi}{\sqrt{1+\xi^2}}\right]}\\
&\qquad\cdot \hat h_{j_1}(\eta_1)\cdots\hat h_{j_{2\mu}}(\eta_{2\mu})\frac{\hat h_{j_{2\mu+1}}(\xi-\eta_1-\cdots-\eta_{2\mu})}{\xi-\eta_1-\cdots-\eta_{2\mu}}\diff \eta_1\cdots\diff \eta_{{2\mu}}\\
&-t\iiint_{\R^{{2\mu}}} \Tb_{\mu}(\eta_1,\cdots,\eta_{2\mu}, \xi-\eta_1-\cdots-\eta_{2\mu})e^{it\left[\frac{\eta_1}{\sqrt{1+\eta_1^2}}+\cdots+\frac{\eta_{2\mu}}{\sqrt{1+\eta_{2\mu}^2}}+\frac{\xi-\eta_1-\cdots-\eta_{2\mu}}{\sqrt{1+(\xi-\eta_1-\cdots-\eta_{2\mu})^2}}-\frac{\xi}{\sqrt{1+\xi^2}}\right]}\\
&\quad\cdot \left[(1+(\xi-\eta_1-\cdots-\eta_{2\mu})^2)^{-3/2}-(1+\xi^2)^{-3/2}\right] \hat h_{j_1}(\eta_1)\cdots\hat h_{j_{2\mu}}(\eta_{2\mu})\hat h_{j_{2\mu+1}}(\xi-\eta_1-\cdots-\eta_{2\mu})\diff \eta_1\cdots\diff \eta_{{2\mu}}.
\end{align*}
By the symbol estimates Proposition \ref{Prop_symb_Tmu} and the symbol $(1+(\xi-\eta_1-\cdots-\eta_{2\mu})^2)^{-3/2}-(1+\xi^2)^{-3/2}$ is bounded and smooth, we have the following estimate for each term. 
\begin{align*}
&\left\|\xi\iiint_{\R^{{2\mu}}} \partial_{2\mu+1}\Tb_{\mu}(\eta_1,\cdots,\eta_{2\mu}, \xi-\eta_1-\cdots-\eta_{2\mu})e^{it\left[\frac{\eta_1}{\sqrt{1+\eta_1^2}}+\cdots+\frac{\eta_{2\mu}}{\sqrt{1+\eta_{2\mu}^2}}+\frac{\xi-\eta_1-\cdots-\eta_{2\mu}}{\sqrt{1+(\xi-\eta_1-\cdots-\eta_{2\mu})^2}}-\frac{\xi}{\sqrt{1+\xi^2}}\right]}\right.\\
&\qquad\left.\cdot \hat h_{j_1}(\eta_1)\cdots\hat h_{j_{2\mu}}(\eta_{2\mu})\hat h_{j_{2\mu+1}}(\xi-\eta_1-\cdots-\eta_{2\mu})\diff \eta_1\cdots\diff \eta_{{2\mu}}\right\|_{L^2}\\
\lesssim~ & \|\vp_{j_1}\|_{H^1}\|\vp_{j_2}\|_{B^{1,6}}\cdots\|\vp_{j_{2\mu+1}}\|_{B^{1,6}},
\end{align*}

\begin{align*}
&\left\|\xi\iiint_{\R^{{2\mu}}} \Tb_{\mu}(\eta_1,\cdots,\eta_{2\mu}, \xi-\eta_1-\cdots-\eta_{2\mu})e^{it\left[\frac{\eta_1}{\sqrt{1+\eta_1^2}}+\cdots+\frac{\eta_{2\mu}}{\sqrt{1+\eta_{2\mu}^2}}+\frac{\xi-\eta_1-\cdots-\eta_{2\mu}}{\sqrt{1+(\xi-\eta_1-\cdots-\eta_{2\mu})^2}}-\frac{\xi}{\sqrt{1+\xi^2}}\right]}\right.\\
&\qquad\left.\hat h_{j_1}(\eta_1)\cdots\hat h_{j_{2\mu}}(\eta_{2\mu})\partial_\xi \hat h_{j_{2\mu+1}}(\xi-\eta_1-\cdots-\eta_{2\mu})\diff \eta_1\cdots\diff \eta_{{2\mu}}\right\|_{L^2}\\
\lesssim~ &\|\vp_{j_1}\|_{B^{1,6}}\|\vp_{j_2}\|_{B^{1,6}}\cdots\|\vp_{j_{2\mu}}\|_{B^{1,6}}\|\partial_\xi\hat h_{j_{2\mu+1}}(\xi)\|_{L^2_\xi},
\end{align*}
\begin{align*}
&\left\| it\xi\iiint_{\R^{{2\mu}}} \Tb_{\mu}(\eta_1,\cdots,\eta_{2\mu}, \xi-\eta_1-\cdots-\eta_{2\mu})e^{it\left[\frac{\eta_1}{\sqrt{1+\eta_1^2}}+\cdots+\frac{\eta_{2\mu}}{\sqrt{1+\eta_{2\mu}^2}}+\frac{\xi-\eta_1-\cdots-\eta_{2\mu}}{\sqrt{1+(\xi-\eta_1-\cdots-\eta_{2\mu})^2}}-\frac{\xi}{\sqrt{1+\xi^2}}\right]}\right.\\
&~~\left.\left[(1+(\xi-\eta_1-\cdots-\eta_{2\mu})^2)^{-3/2}-(1+\xi^2)^{-3/2}\right]  \hat h_{j_1}(\eta_1)\cdots\hat h_{j_{2\mu}}(\eta_{2\mu})\hat h_{j_{2\mu+1}}(\xi-\eta_1-\cdots-\eta_{2\mu})\diff \eta_1\cdots\diff \eta_{{2\mu}}\right\|_{L^2}\\
\lesssim~ &t\|\vp_{j_1}\|_{H^1}\|\vp_{j_2}\|_{B^{1,6}}\cdots\|\vp_{j_{2\mu+1}}\|_{B^{1,6}},
\end{align*}
\begin{align*}
&\left\|\xi\iiint_{\R^{{2\mu}}} \Tb_{\mu}(\eta_1,\cdots,\eta_{2\mu}, \xi-\eta_1-\cdots-\eta_{2\mu})e^{it\left[\frac{\eta_1}{\sqrt{1+\eta_1^2}}+\cdots+\frac{\eta_{2\mu}}{\sqrt{1+\eta_{2\mu}^2}}+\frac{\xi-\eta_1-\cdots-\eta_{2\mu}}{\sqrt{1+(\xi-\eta_1-\cdots-\eta_{2\mu})^2}}-\frac{\xi}{\sqrt{1+\xi^2}}\right]}\right.\\
&\qquad\left.\cdot \hat h_{j_1}(\eta_1)\cdots\hat h_{j_{2\mu}}(\eta_{2\mu})\hat h_{j_{2\mu+1}}(\xi-\eta_1-\cdots-\eta_{2\mu})\diff \eta_1\cdots\diff \eta_{{2\mu}}\right\|_{L^2}\\
\lesssim~ &\|\vp_{j_1}\|_{H^1}\|\vp_{j_2}\|_{B^{1,6}}\cdots\|\vp_{j_{2\mu}}\|_{B^{1,6}}\|\vp_{j_{2\mu+1}}\|_{B^{1,6}}.
\end{align*}
Combining the above estimates, by taking summation for suitable sub-indices, 
\[
\|{\rm{V}}\|_{L^2}\lesssim t\sum\limits_{\mu=2}^\infty \|\vp\|_{B^{1,6}}^{2\mu}(\|\vp\|_{H^1}+\|\partial_\xi\hat h(\xi)\|_{L^2}).
\]

\subsection{Term $\rm{VI}$ estimate}
The term $\rm{VI}$ is bounded by
\begin{align*}
\|{\rm{VI}}\|_{L^2}&\lesssim \sum\limits_{(j_1,j_2,j_3)\in\P_{2}}2^{2j_1}(1+2^{j_1})^2 2^{2\max\{0, j_1\}}  \|\vp_{j_1}\|_{L^\infty} \|\vp_{j_2}\|_{L^\infty}\| \vp_{j_3}\|_{L^2}\\
&\lesssim \|\vp\|_{B^{1,6}}^{2} \|\vp\|_{H^{6}}.
\end{align*}

Combining all the estimates above, we have proved the claim \eqref{claim}.

\section{$Z$-norm estimates}\label{sec:Znorm}
In this section, we prove Lemma \ref{Z-norm_Est}. We will estimate $|(|\xi|^{r}+|\xi|^{w})\hat\vp(t,\xi)|$ for each fixed $\xi$. According to different values of $\xi$, the estimate is done in different ways. 
\subsection{Large and small frequencies}
\label{sec:large}
When $|\xi|$ is very big or small, we can estimate $|(|\xi|^{r}+|\xi|^{w})\hat\vp(t,\xi)|$ from the bootstrap assumption without using the equation.  Let $p_1=10^{-5}$. When $|\xi|\geq (t + 1)^{p_1}$, we get similarly from Lemma~\ref{interpolation} and the bootstrap assumptions that \begin{align*}
|(|\xi|^{r}+|\xi|^{w})\hat\vp(t,\xi)|^2&\lesssim (|\xi|^{r}+|\xi|^{w})^2[\|\vp\|_{L^2}\|\partial_\xi\hat h\|_{L^2}+|\xi|^{-1}\|\vp\|_{L^2}^2]\\
&\lesssim\frac{(|\xi|^{r}+|\xi|^{w})^2}{|\xi|^{s}}\|\vp\|_{H^s}\|\partial_\xi\hat h\|_{L^2_\xi}+\frac{(|\xi|^{r}+|\xi|^{w})^2}{|\xi|^{s+1}}\|\vp\|_{H^s}\|\vp\|_{L^2}\\
&\lesssim |\xi|^{2w-s}(t+1)^{2p_0} \ve_0^2\\
&\lesssim (t+1)^{-(s-2w)p_1+2p_0}\ve_0^2 \\
&\lesssim \ve_0^2,
\end{align*}
since $-(s-2w)p_1+2p_0<0$ for the parameter values in \eqref{param_vals}.

When $|\xi|<(t + 1)^{- 0.01/r}=(t+1)^{-0.025}$, by Sobolev embedding, the bootstrap assumptions, Lemma~\ref{weightedE}, and the conservation of the $L^2$-norm of $\vp$ gives
\begin{align*}
|(|\xi|^{r} + |\xi|^{w})\hat\vp(t, \xi)|^2&\lesssim (t+1)^{-0.01}(\|\hat\vp\|_{L^2_\xi}^2+\|\partial_\xi\hat h\|_{L^2_\xi}^2)\\
&\lesssim \ve_0^2.
\end{align*}

Thus, we only need to consider the frequency range $(t+1)^{-0.025}\leq |\xi|\leq (t + 1)^{p_1}$ in the following. From now on, we fix $\xi$ in this range, and use $\cutoffxi$ to denote a smooth cutoff function such that
\begin{align}
\begin{split}
&\text{$\cutoffxi(t, \xi) = 1$ on $\left\{(t, \xi) \mid (t+1)^{-0.025}\leq|\xi|\le(t+1)^{p_1}\right\}$,}
\\
&\text{$\cutoffxi(t, \xi)$ is supported on a small neighborhood of $\left\{(t, \xi) \mid (t+1)^{-0.025}\leq|\xi|\le(t+1)^{p_1}\right\}$.}
\end{split}
\label{defdelta}
\end{align}

\subsection{Modified scattering}
Define $\mathfrak{A}:=\frac{3\xi}{(1+\xi^2)^{5/2}}$ and the phase correction
\begin{equation}
\Theta(t, \xi) =- \frac{\pi \xi }{3\mathfrak{A}} [\Tb_1(\xi, \xi, -\xi) + \Tb_1(\xi,-\xi,\xi) + \Tb_1(-\xi, \xi, \xi)] \int_0^t \frac{|\hat \vp(\tau, \xi)|^2}{\tau + 1} \diff{\tau}.
\label{defTheta}
\end{equation}
We then let
\[
\hat v(t, \xi)=e^{i\Theta(t, \xi)}\hat{h}(t, \xi).
\]
Using \eqref{eqhhat} and \eqref{defTheta}, we find that
\begin{align}
\hat v_t(t, \xi) = e^{i \Theta(t, \xi)} [\hat{h}_t(t, \xi) + i \Theta_t(t, \xi) \hat{h}(t, \xi)] = U(t, \xi) - e^{-  i t \xi (1+\xi^2)^{-1/2}} e^{i \Theta(t, \xi)}i\xi \widehat{\Nc_{\geq 5}(\vp)}(t, \xi),
\label{vteq}
\end{align}
where 
\begin{align}
\begin{split}
U(t, \xi) &=e^{i\Theta(t, \xi)}\bigg\{ \frac13 i\xi \iint_{\R^2} \Tb_1(\eta_1, \eta_2, \xi - \eta_1 - \eta_2) e^{it\Phi(\eta_1,\eta_2, \xi)} \hat h(\xi-\eta_1-\eta_2) \hat h(\eta_1)\hat h(\eta_2)  \diff{\eta_1} \diff{\eta_2}\\
&\qquad - \frac{\pi i \xi }{3\mathfrak{A}} \left[\Tb_1(\xi, \xi, -\xi) + \Tb_1(\xi, -\xi, \xi) + \Tb_1(-\xi,\xi,\xi)\right] \frac{|\hat{h}(t, \xi)|^2\hat{h}(t, \xi)}{t + 1}\bigg\}.
\end{split}
\label{defU}
\end{align}
Then we obtain from \eqref{vteq} that
\begin{align*}
\|\vp\|_{Z}&=\|(|\xi|^{r}+|\xi|^{w}) \hat\vp(t, \xi)\|_{L^\infty_\xi}=\|(|\xi|^{r}+|\xi|^{w}) \hat v(t, \xi)\|_{L^\infty_\xi}
\\
&\lesssim \int_0^t \left\{\|(|\xi|^{r}+|\xi|^{w})U(\xi, \tau)\|_{L^\infty_\xi}+\|(|\xi|^{r+1}+|\xi|^{w+1})\widehat{\mathcal{N}_{\geq5}(\vp)}(\xi, \tau)\|_{L^\infty_\xi}\right\} \diff \tau.
\end{align*}
We first estimate the term
\begin{align*}
&\|(|\xi|^{r}+|\xi|^{w})U(\tau, \xi)\|_{L^\infty_\xi}\\
=~&\Bigg\|(|\xi|^{r+1}+|\xi|^{w+1})\bigg\{ \frac13  \iint_{\R^2} \Tb_1(\eta_1, \eta_2, \xi - \eta_1 - \eta_2) e^{it\Phi(\eta_1,\eta_2, \xi)} \hat h(\xi-\eta_1-\eta_2) \hat h(\eta_1)\hat h(\eta_2)  \diff{\eta_1} \diff{\eta_2}\\
& - \frac{\pi  }{3\mathfrak{A}} \left[\Tb_1(\xi, \xi, -\xi) + \Tb_1(\xi, -\xi, \xi) + \Tb_1(-\xi,\xi,\xi)\right] \frac{|\hat{h}(t, \xi)|^2\hat{h}(t, \xi)}{t + 1}\bigg\}\Bigg\|.
\end{align*}
Taking the dyadic decomposition and assuming $j_1\geq j_2\geq j_3$, we can rewrite the following integral as 
\begin{align*}
\frac13 \iint_{\R^2} \Tb_1(\eta_1, \eta_2, \xi - \eta_1 - \eta_2) e^{it\Phi(\eta_1,\eta_2, \xi)} \hat h(\xi-\eta_1-\eta_2) \hat h(\eta_1)\hat h(\eta_2)  \diff{\eta_1} \diff{\eta_2}\\
=\frac13  \sum\limits_{\P}\iint_{\R^2}\Tb_1(\eta_1,\eta_2,\xi-\eta_1-\eta_2) e^{it\Phi} \hat h_{j_1}(\eta_1)\hat h_{j_2}(\eta_2)\hat h_{j_3}(\xi-\eta_1-\eta_2)\diff \eta_1\diff \eta_2,
\end{align*}
where the summation with $\P$ same as previous section in \eqref{eqn7.2}.


\subsection{Nonresonant frequencies}\label{sec:nonres}
Since $|\xi|\approx 2^{j_1}$, by the assumption of $\cutoff$, we only need to consider $j_1<p_1\log_2(t+1)=10^{-5}\log_2(t+1)$.
The indices satisfying $j_1-j_3>1$  correspond to nonresonant frequencies. Denote the set of these indices as $\P'_1$. We will estimate
\[
\left\|(|\xi|^{r}+|\xi|^{w})\xi\iint_{\R^2}\Tb_1(\eta_1,\eta_2,\xi-\eta_1-\eta_2) e^{it\Phi} \hat h_{j_1}(\eta_1)\hat h_{j_2}(\eta_2)\hat h_{j_3}(\xi-\eta_1-\eta_2)\diff \eta_1\diff \eta_2\right\|_{L^\infty_\xi}.
\]
After integrating by parts, we have
\begin{align*}
&  \iint_{\R^2}    \Tb_1(\eta_1, \eta_2, \xi - \eta_1 - \eta_2) e^{it\Phi(\eta_1,\eta_2, \xi)} \hat h_{j_1}(\eta_1)\hat h_{j_2}(\eta_2) \hat h_{j_3}(\xi-\eta_1-\eta_2) \diff{\eta_1} \diff{\eta_2}\\
 =~&   \iint_{\R^2}    \frac{\Tb_1(\eta_1, \eta_2, \xi - \eta_1 - \eta_2)}{it\partial_{\eta_1}\Phi(\eta_1,\eta_2, \xi)}\partial_{\eta_1} e^{it\Phi(\eta_1,\eta_2, \xi)} \hat h_{j_1}(\eta_1)\hat h_{j_2}(\eta_2) \hat h_{j_3}(\xi-\eta_1-\eta_2) \diff{\eta_1} \diff{\eta_2}\\
 =~& -W_1-W_2-W_3,
\end{align*}
where
\begin{align*}
W_1(t, \xi) & =  \iint_{\R^2} \partial_{\eta_1}\left[ \frac{\Tb_1(\eta_1, \eta_2, \xi - \eta_1 - \eta_2)}{it\partial_{\eta_1}\Phi(\eta_1,\eta_2, \xi)}\right] e^{it\Phi(\eta_1,\eta_2, \xi)}  \hat h_{j_1}(\eta_1)\hat h_{j_2}(\eta_2) \hat h_{j_3}(\xi-\eta_1-\eta_2) \diff{\eta_1} \diff{\eta_2},\\
W_2(t, \xi) & =  \iint_{\R^2} \left[ \frac{\Tb_1(\eta_1, \eta_2, \xi - \eta_1 - \eta_2)}{it\partial_{\eta_1}\Phi(\eta_1,\eta_2, \xi)}\right] e^{it\Phi(\eta_1,\eta_2, \xi)} \hat h_{j_1}(\eta_1)\hat h_{j_2}(\eta_2) \partial_{\eta_1}\hat h_{j_3}(\xi-\eta_1-\eta_2) \diff{\eta_1} \diff{\eta_2},\\
W_3(t, \xi) & =  \iint_{\R^2} \left[ \frac{\Tb_1(\eta_1, \eta_2, \xi - \eta_1 - \eta_2)}{it\partial_{\eta_1}\Phi(\eta_1,\eta_2, \xi)}\right] e^{it\Phi(\eta_1,\eta_2, \xi)} \partial_{\eta_1} \hat h_{j_1}(\eta_1)\hat h_{j_2}(\eta_2) \hat h_{j_3}(\xi-\eta_1-\eta_2) \diff{\eta_1} \diff{\eta_2}.
\end{align*}
In the first integral $W_1$,
\begin{align*}
&\partial_{\eta_1}\left[ \frac{\Tb_1(\eta_1, \eta_2, \xi - \eta_1 - \eta_2)}{\partial_{\eta_1}\Phi(\eta_1,\eta_2, \xi)}\right] \\
=~&\frac{\partial_1\Tb_1(\eta_1, \eta_2, \xi - \eta_1 - \eta_2)-\partial_3\Tb_1(\eta_1, \eta_2, \xi - \eta_1 - \eta_2)}{p'(\eta_1)-p'(\xi-\eta_1-\eta_2)}-\frac{\Tb_1(\eta_1, \eta_2, \xi - \eta_1 - \eta_2)[p''(\eta_1)+p''(\xi-\eta_1-\eta_2)]}{(p'(\eta_1)-p'(\xi-\eta_1-\eta_2))^2}.
\end{align*}
By the symbol estimate \eqref{symbolpartialphi},  \eqref{symbol2d} and \eqref{Tmu_est1} the $S^\infty$ norm of the symbol is bounded by
\begin{align*}
&[(1+2^{2j_1})(1+2^{3j_3})2^{-2j_1}][(1+2^{j_2})^5(1+2^{j_3})2^{j_2+j_3}]\\
&+[(1+2^{2j_1})(1+2^{3j_3})2^{-2j_1}]^2[(1+2^{j_1})^{-1}(1+2^{j_2})^3(1+2^{j_3})2^{j_1+j_2+j_3}][2^{j_1}(1+2^{j_1})^{-1}]\\
\lesssim~&2^{j_2+j_3-2j_1}(1+2^{j_1})^2(1+2^{j_2})^5(1+2^{j_3})^4+2^{j_2+j_3-2j_1}(1+2^{j_2})^3(1+2^{j_3})^{4}\\
\lesssim~&2^{j_2+j_3-2j_1}(1+2^{j_1})^2(1+2^{j_2})^5(1+2^{j_3})^{4}.
\end{align*}
By the symbol estimate \eqref{symbolpartialphi} and \eqref{Tmu_est1}, we obtain the estimate for $W_1$:
\begin{align*}
\|(|\xi|^{r+1}+|\xi|^{w + 1})W_1\|_{L^\infty_\xi}\lesssim~& (1+t)^{wp_1-1}    2^{j_2+j_3-j_1}(1+2^{j_1})^2(1+2^{j_2})^5(1+2^{j_3})^{4}\|\vp_{j_1}\|_{L^2}\|\vp_{j_2}\|_{L^2}\|\vp_{j_3}\|_{L^\infty}\\
\lesssim~& (1+t)^{wp_1-1}\|(1+2^{j_1})^7\vp_{j_1}\|_{L^2}\|\vp_{j_2}\|_{L^2}\|2^{j_3}(1+2^{j_3})^{4}\vp_{j_3}\|_{L^\infty}\\
\lesssim~& (1+t)^{wp_1-1}\|\vp_{j_1}\|_{H^s}\|\vp_{j_2}\|_{L^2}\|\vp_{j_3}\|_{B^{1,6}}.
\end{align*}
By the symbol estimate \eqref{symbolpartialphi} and \eqref{Tmu_est1}, the symbol of $W_2$ and $W_3$ is bounded by 
\[
[(1+2^{2j_1})(1+2^{3j_3})2^{-2j_1}][(1+2^{j_1})^{-1}(1+2^{j_2})^3(1+2^{j_3})2^{j_1+j_2+j_3}]\lesssim (1+2^{j_1})(1+2^{j_2})^3(1+2^{j_3})^42^{j_2+j_3-j_1}.
\]
Therefore,
\begin{align*}
&\|(|\xi|^{r+1}+|\xi|^{w + 1})W_2\|_{L^\infty_\xi}\\
\lesssim &(1+t)^{wp_1}2^{j_1}[(1+2^{j_1})(1+2^{j_2})^3(1+2^{j_3})^42^{j_2+j_3-j_1}]\|\vp_{j_1}\|_{L^\infty}\|\vp_{j_2}\|_{L^2}\|\partial_\xi \hat h_{j_3}\|_{L^2_\xi}\\
\lesssim &(1+t)^{(w+4)p_1}2^{j_2+j_3}(1+2^{j_1})^4\|\vp_{j_1}\|_{L^\infty}\|\vp_{j_2}\|_{L^2}\|\partial_\xi \hat h_{j_3}\|_{L^2}\\
\lesssim &(1+t)^{(w+4)p_1}\|\vp_{j_1}\|_{B^{1,6}}\|\vp_{j_2}\|_{L^2}\|\partial_\xi \hat h_{j_3}\|_{L^2},
\end{align*}
and
\begin{align*}
&\|(|\xi|^{r+1}+|\xi|^{w + 1})W_3\|_{L^\infty_\xi}\\
\lesssim &(1+t)^{wp_1}2^{j_1}[(1+2^{j_1})(1+2^{j_2})^3(1+2^{j_3})^42^{j_2+j_3-j_1}]\|\vp_{j_3}\|_{L^\infty}\|\vp_{j_2}\|_{L^2}\|\partial_\xi \hat h_{j_1}\|_{L^2_\xi}\\
\lesssim &(1+t)^{(w+5)p_1}\|\vp_{j_3}\|_{B^{1,6}}\|\vp_{j_2}\|_{L^2}\|\partial_\xi \hat h_{j_1}\|_{L^2}.
\end{align*}
By taking the summation for $\P_1'$,
\begin{align*}
&\sum\limits_{\P_1'}\left\|(|\xi|^{r+1}+|\xi|^{w + 1}) \iint_{\R^2}    \Tb_1(\eta_1, \eta_2, \xi - \eta_1 - \eta_2) e^{it\Phi(\eta_1,\eta_2, \xi)} \hat h_{j_1}(\eta_1)\hat h_{j_2}(\eta_2) \hat h_{j_3}(\xi-\eta_1-\eta_2) \diff{\eta_1} \diff{\eta_2}\right\|_{L^\infty_\xi}\\
\lesssim &(1+t)^{(w+5)p_1-1}\|\vp\|_{B^{1, 6}}\|\vp\|_{L^2}(\|\vp\|_{H^s}+\|\partial_\xi \hat h(\xi)\|_{L^2})\\
\lesssim &(1+t)^{-1.5+(w+5)p_1}\ve_1^3(t+1)^{p_0+0.01},
\end{align*}
which is integrable for $t\in(0,\infty)$.

\subsection{Close to resonance}\label{sec:near}
When $j_1\geq j_2\geq j_3$, $j_1-j_3\leq 1$ and $ j_1<10^{-5}\log_2(t+1)$, we denote the set of these indices as $\P'_2$. 
 For $\iota_1, \iota_2, \iota_3\in \{+,-\}$, we split the integral into several parts
 \begin{align*}
\iint_{\R^2}    \Tb_1(\eta_1, \eta_2, \xi - \eta_1 - \eta_2) e^{it\Phi(\eta_1,\eta_2, \xi)} \hat h_{j_1}^{\iota_1}(\eta_1)\hat h_{j_2}^{\iota_2}(\eta_2) \hat h_{j_3}(\xi-\eta_1-\eta_2)\upsilon_{\iota_3}(\xi) \diff{\eta_1} \diff{\eta_2}.
 \end{align*}
The we estimate each case in the following.

{\noindent\bf 1. $(\iota_1, \iota_2, \iota_3)=(+, +, -)$ or $(-, -, +)$.} This is the nonresonant case. We integrate by part with respect to $\eta_1$ and obtain 
\begin{equation}
\begin{aligned}
& \iint_{\R^2}\partial_{\eta_1}[\Tb_1(\eta_1,\eta_2,\xi-\eta_1-\eta_2)\frac{1}{it\partial_{\eta_1}\Phi} ]e^{it\Phi} \hat h_{j_1}^{\iota_1}(\eta_1)\hat h_{j_2}^{\iota_2}(\eta_2)\hat h_{j_3}(\xi-\eta_1-\eta_2)\upsilon_{\iota_3}(\xi)\diff \eta_1\diff \eta_2\\
&+  \iint_{\R^2}\Tb_1(\eta_1,\eta_2,\xi-\eta_1-\eta_2)\frac{1}{it\partial_{\eta_1}\Phi} e^{it\Phi} \partial_{\eta_1}\hat h_{j_1}^{\iota_1}(\eta_1)\hat h_{j_2}^{\iota_2}(\eta_2)\hat h_{j_3}(\xi-\eta_1-\eta_2)\upsilon_{\iota_3}(\xi)\diff \eta_1\diff \eta_2\\
&+  \iint_{\R^2}\Tb_1(\eta_1,\eta_2,\xi-\eta_1-\eta_2)\frac{1}{it\partial_{\eta_1}\Phi} e^{it\Phi} \hat h_{j_1}^{\iota_1}(\eta_1)\hat h_{j_2}^{\iota_2}(\eta_2)\partial_{\eta_1}\hat h_{j_3}(\xi-\eta_1-\eta_2)\upsilon_{\iota_3}(\xi)\diff \eta_1\diff \eta_2.
\end{aligned}\end{equation}
By using the symbol estimates \eqref{eqn715}, \eqref{eqn716}, and \eqref{eqn717}, we obtain the estimate for each integrals.
 \begin{align*}
&\left\|(|\xi|^{r+1}+|\xi|^{w+1})\iint_{\R^2}\partial_{\eta_1}[\Tb_1(\eta_1,\eta_2,\xi-\eta_1-\eta_2)\frac{1}{\partial_{\eta_1}\Phi} ]e^{it\Phi} \hat h_{j_1}^{\iota_1}(\eta_1)\hat h_{j_2}^{\iota_2}(\eta_2)\hat h_{j_3}(\xi-\eta_1-\eta_2)\upsilon_{\iota_3}(\xi)\diff \eta_1\diff \eta_2\right\|_{L^\infty_\xi}\\
\lesssim~&(t+1)^{wp_1}2^{j_1}\left\{[2^{j_2+j_3}(1+2^{j_2})^5(1+2^{j_3})][(1+2^{j_1})^72^{-2j_1}]\right.\\
&\left.+[(1+2^{j_1})^72^{-2j_1}]^2[2^{j_1}(1+2^{j_1})^{-5}][2^{j_1+j_2+j_3}(1+2^{j_1})^{-1}(1+2^{j_2})^3(1+2^{j_3})]\right\}\|\vp_{j_1}\|_{L^2}\|\vp_{j_2}\|_{L^2}\|\vp_{j_3}\|_{L^\infty}\\
\lesssim~&(t+1)^{wp_1}2^{j_1}(1+2^{j_1})^{13}\|\vp_{j_1}\|_{L^2}\|\vp_{j_2}\|_{L^2}\|\vp_{j_3}\|_{L^\infty}\\
\lesssim~&(t+1)^{(w+8)p_1}\|\vp_{j_1}\|_{L^2}\|\vp_{j_2}\|_{L^2}\|\vp_{j_3}\|_{B^{1,6}},
 \end{align*}
 \begin{align*}
&\left\|(|\xi|^{r+1}+|\xi|^{w+1}) \iint_{\R^2}\Tb_1(\eta_1,\eta_2,\xi-\eta_1-\eta_2)\frac{1}{it\partial_{\eta_1}\Phi} e^{it\Phi} \partial_{\eta_1}\hat h_{j_1}^{\iota_1}(\eta_1)\hat h_{j_2}^{\iota_2}(\eta_2)\hat h_{j_3}(\xi-\eta_1-\eta_2)\upsilon_{\iota_3}(\xi)\diff \eta_1\diff \eta_2\right\|_{L^\infty_\xi}\\
\lesssim~&(t+1)^{wp_1}2^{j_1}[2^{j_1+j_2+j_3}(1+2^{j_1})^{-1}(1+2^{j_2})^3(1+2^{j_3})][(1+2^{j_1})^72^{-2j_1}]\|\partial_\xi\hat h_{j_1}\|_{L^2}\|\vp_{j_2}\|_{L^2}\|\vp_{j_3}\|_{L^\infty}\\
\lesssim~&(t+1)^{wp_1}2^{2j_1}(1+2^{j_1})^{10}\|\partial_\xi\hat h_{j_1}\|_{L^2}\|\vp_{j_2}\|_{L^2}\|\vp_{j_3}\|_{L^\infty}\\
\lesssim~&(t+1)^{(w+6)p_1}\|\partial_\xi\hat h_{j_1}\|_{L^2}\|\vp_{j_2}\|_{L^2}\|\vp_{j_3}\|_{B^{1,6}},
 \end{align*}
 \begin{align*}
&\left\|(|\xi|^{r+1}+|\xi|^{w+1}) \iint_{\R^2}\Tb_1(\eta_1,\eta_2,\xi-\eta_1-\eta_2)\frac{1}{it\partial_{\eta_1}\Phi} e^{it\Phi} \partial_{\eta_1}\hat h_{j_1}^{\iota_1}(\eta_1)\hat h_{j_2}^{\iota_2}(\eta_2)\hat h_{j_3}(\xi-\eta_1-\eta_2)\upsilon_{\iota_3}(\xi)\diff \eta_1\diff \eta_2\right\|_{L^\infty_\xi}\\
\lesssim~&(t+1)^{(w+6)p_1}\|\vp_{j_1}\|_{B^{1,6}}\|\vp_{j_2}\|_{L^2}\|\partial_\xi\hat h_{j_3}\|_{L^2}.
 \end{align*}
Therefore taking the summation for corresponding indices $(j_1, j_2, j_3)\in\P'_2$, we have
 \begin{align*}
&\sum\limits_{\P'_2}\left\|(|\xi|^{r+1}+|\xi|^{w+1})\iint_{\R^2}    \Tb_1(\eta_1, \eta_2, \xi - \eta_1 - \eta_2) e^{it\Phi(\eta_1,\eta_2, \xi)} \hat h_{j_1}^{\iota_1}(\eta_1)\hat h_{j_2}^{\iota_2}(\eta_2) \hat h_{j_3}(\xi-\eta_1-\eta_2)\upsilon_{\iota_3}(\xi) \diff{\eta_1} \diff{\eta_2}\right\|\\
\lesssim ~&(t+1)^{(w+8)p_1-1}\|\vp\|_{B^{1,6}}\|\vp\|_{L^2}(\|\vp\|_{L^2}+\|\partial_\xi\hat h(\xi)\|_{L^2})\\
\lesssim &(1+t)^{-1.5+(w+8)p_1+0.01}\ve_1^3,
 \end{align*}
 which is integrable in time $t\in (0, \infty)$.
 
{\noindent\bf 2. $(\iota_1, \iota_2, \iota_3)=(+, +, +)$ or $(-, -, -)$.}
We split the integral into two parts:
 \begin{align}\label{eq8.6}
\iint_{\R^2}    \Tb_1(\eta_1, \eta_2, \xi - \eta_1 - \eta_2) e^{it\Phi(\eta_1,\eta_2, \xi)} \hat h_{j_1}^{\iota_1}(\eta_1)\hat h_{j_2}^{\iota_2}(\eta_2) \hat h_{j_3}(\xi-\eta_1-\eta_2)\upsilon_{\iota_3}(\xi)\psi_{j-3}(\eta_1+\eta_2-2\xi) \diff{\eta_1} \diff{\eta_2},
 \end{align}
 and
  \begin{align}\label{eq8.7}
\iint_{\R^2}    \Tb_1(\eta_1, \eta_2, \xi - \eta_1 - \eta_2) e^{it\Phi(\eta_1,\eta_2, \xi)} \hat h_{j_1}^{\iota_1}(\eta_1)\hat h_{j_2}^{\iota_2}(\eta_2) \hat h_{j_3}(\xi-\eta_1-\eta_2)\upsilon_{\iota_3}(\xi)[1-\psi_{j-3}(\eta_1+\eta_2-2\xi)] \diff{\eta_1} \diff{\eta_2}.
 \end{align}
The first integral \eqref{eq8.6} includes the space-time resonance point $(\eta_1, \eta_2)=(\xi, \xi)$ and the second one include the space resonance point $(\eta_1, \eta_2)=(\xi/3, \xi/3)$.
For each resonance, we define the cut-off functions
\begin{align*}
\cutoff_{\lb}^{1}(\eta_1, \eta_2, \xi):=\psi_{l_{1,1}}(\eta_1-\xi)\psi_{l_{1,2}}(\eta_2-\xi),
\end{align*}
\begin{align*}
\cutoff_{\lb}^{2}(\eta_1, \eta_2, \xi):=\psi_{l_{2,1}}(\eta_1-\frac\xi3)\psi_{l_{2,2}}(\eta_2-\frac\xi3),
\end{align*}
where $(l_{1,1}, l_{1,2}, l_{2,1}, l_{2,2})\in \Z^4$. Define
\begin{equation}\label{varrho}
\varrho(t):=(t+1)^{-0.45}.
\end{equation}
We first consider the case when $\max\{l_{1,1}, l_{1,2}\}\geq \log_2(\varrho(t))$.

Considering the support of the integrand, we write the first integral as
\begin{align*}
&\iint_{\R^2}  \Tb_1(\eta_1, \eta_2, \xi - \eta_1 - \eta_2) e^{it\Phi(\eta_1,\eta_2, \xi)} \hat h_{j_1}^{\iota_1}(\eta_1)\hat h_{j_2}^{\iota_2}(\eta_2) \hat h_{j_3}(\xi-\eta_1-\eta_2)\upsilon_{\iota_3}(\xi)\psi_{j-3}(\eta_1+\eta_2-2\xi) \diff{\eta_1} \diff{\eta_2}\\
=~&\sum\limits_{\max\{l_{1,1}, l_{1,2}\}\leq j+5}\iint_{\R^2} \cutoff_{\lb}^{1} \Tb_1(\eta_1, \eta_2, \xi - \eta_1 - \eta_2) e^{it\Phi(\eta_1,\eta_2, \xi)} \hat h_{j_1}^{\iota_1}(\eta_1)\hat h_{j_2}^{\iota_2}(\eta_2) \hat h_{j_3}(\xi-\eta_1-\eta_2)\\
&\qquad\upsilon_{\iota_3}(\xi)\psi_{j-3}(\eta_1+\eta_2-2\xi) \diff{\eta_1} \diff{\eta_2}.
 \end{align*}

Since $\eta_1$ and $\eta_2$ are symmetric, we assume $l_{1,1}\geq l_{1,2}$. We integrate by part with respect to $\eta_1$ and obtain
\begin{equation}\label{eqn8.6}
\begin{aligned}
& \iint_{\R^2}\partial_{\eta_1}[\cutoff_{\lb}^{1}\psi_{j-3}(\eta_1+\eta_2-2\xi)\Tb_1(\eta_1,\eta_2,\xi-\eta_1-\eta_2)\frac{1}{it\partial_{\eta_1}\Phi} ]e^{it\Phi} \hat h_{j_1}^{\iota_1}(\eta_1)\hat h_{j_2}^{\iota_2}(\eta_2)\hat h_{j_3}(\xi-\eta_1-\eta_2)\upsilon_{\iota_3}(\xi)\diff \eta_1\diff \eta_2\\
&+  \iint_{\R^2}\cutoff_{\lb}^{1}\psi_{j-3}(\eta_1+\eta_2-2\xi)\Tb_1(\eta_1,\eta_2,\xi-\eta_1-\eta_2)\frac{1}{it\partial_{\eta_1}\Phi} e^{it\Phi} \partial_{\eta_1}\hat h_{j_1}^{\iota_1}(\eta_1)\hat h_{j_2}^{\iota_2}(\eta_2)\hat h_{j_3}(\xi-\eta_1-\eta_2)\upsilon_{\iota_3}(\xi)\diff \eta_1\diff \eta_2\\
&+  \iint_{\R^2}\cutoff_{\lb}^{1}\psi_{j-3}(\eta_1+\eta_2-2\xi)\Tb_1(\eta_1,\eta_2,\xi-\eta_1-\eta_2)\frac{1}{it\partial_{\eta_1}\Phi} e^{it\Phi} \hat h_{j_1}^{\iota_1}(\eta_1)\hat h_{j_2}^{\iota_2}(\eta_2)\partial_{\eta_1}\hat h_{j_3}(\xi-\eta_1-\eta_2)\diff \eta_1\diff \eta_2.
\end{aligned}\end{equation}

By the symbol estimates Proposition \ref{Prop_symb_Tmu}, \eqref{eqnB11}, and
\begin{align*}
\left\|\psi_{j_1}(\eta_1)\psi_{j_2}(\eta_2)\psi_{j_3}(\eta_3)\psi_{j-3}(\eta_1+\eta_2-2\xi)\upsilon_{\iota_1}(\eta_1)\upsilon_{\iota_2}(\eta_2)\upsilon_{\iota_3}(\eta_1+\eta_2+\eta_3)\cutoff_{\lb}^1\left[\frac{1}{(1+\eta_1^2)^{3/2}}-\frac{1}{[1+\eta_3^2]^{3/2}} \right]^{-1}\right\|_{S^\infty}\\\lesssim (1+2^{2j_1})^{5/2}2^{-j_1}2^{-l_{1,1}},
\end{align*}
we have the Z-norm estimate for each integral terms in \eqref{eqn8.6},
\begin{align*}
&\left\|(|\xi|^{r+1}+|\xi|^{w+1})\iint_{\R^2}    e^{it\Phi(\eta_1,\eta_2, \xi)}\partial_{\eta_1} \left[\frac{\Tb_1(\eta_1, \eta_2, \xi - \eta_1 - \eta_2)}{\partial_{\eta_1}\Phi(\eta_1,\eta_2,\xi)}\right]\cutoff_{\lb}^1\psi_{j-3}(\eta_1+\eta_2-2\xi)\right.\\
&\qquad \left. \hat h_{j_1}(\eta_1)\hat h_{j_2}(\eta_2) \hat h_{j_3}(\xi-\eta_1-\eta_2) \diff{\eta_1} \diff{\eta_2}\right\|_{L^\infty_\xi}\\
\lesssim~&(t+1)^{wp_1}2^{j_1}\Big\{ [(1+2^{2j_1})^{5/2}2^{-j_1}2^{-l_{1,1}}][2^{j_1+j_2+j_3}(1+2^{j_1})^{-1}(1+2^{j_2})^3(1+2^{j_3})]\\
&+[(1+2^{2j_1})^{5/2}2^{-j_1}2^{-l_{1,1}}][2^{j_2+j_3}(1+2^{5j_2})(1+2^{j_3})]\Big\}\|\vp_{j_1}\|_{L^2}\|\vp_{j_2}\|_{L^\infty}\|\vp_{j_3}\|_{L^2}\\
\lesssim~&(t+1)^{wp_1}2^{2j_1}(1+2^{j_1})^{11}2^{-l_{1,1}}\|\vp_{j_1}\|_{L^2}\|\vp_{j_2}\|_{L^\infty}\|\vp_{j_3}\|_{L^2}\\
\lesssim~&(t+1)^{(w+6)p_1}2^{j_1-l_{1,1}}\|\vp_{j_1}\|_{L^2}\|\vp_{j_2}\|_{B^{1,6}}\|\vp_{j_3}\|_{L^2},
\end{align*}

\begin{align*}
&\left\|(|\xi|^{r+1}+|\xi|^{w+1})\iint_{\R^2}    e^{it\Phi(\eta_1,\eta_2, \xi)}\left[\frac{\Tb_1(\eta_1, \eta_2, \xi - \eta_1 - \eta_2)}{\partial_{\eta_1}\Phi(\eta_1,\eta_2,\xi)}\right]\partial_{\eta_1} \cutoff_{\lb}^1\psi_{j-3}(\eta_1+\eta_2-2\xi)\right. \\
&\qquad\left. \hat h_{j_1}(\eta_1)\hat h_{j_2}(\eta_2) \hat h_{j_3}(\xi-\eta_1-\eta_2) \diff{\eta_1} \diff{\eta_2}\right\|_{L^\infty_\xi}\\
\lesssim~&(t+1)^{wp_1}2^{j_1}\Big\{ [(1+2^{2j_1})^{5/2}2^{-j_1}2^{-l_{1,1}}][2^{j_1+j_2+j_3}(1+2^{j_1})^{-1}(1+2^{3j_2})(1+2^{j_3})]2^{-l_{1,1}}\Big\}\|\psi_{l_{1,1}}(\eta_1-\xi)\hat\vp_{j_1}(\eta_1)\|_{L^2_{\eta_1}L^\infty_\xi}\\
&\qquad\cdot\|\psi_{l_{1,2}}(\eta_2-\xi)\hat \vp_{j_2}(\eta_2)\|_{L^2_{\eta_2}L^\infty_\xi}\|\vp_{j_3}\|_{L^\infty}\\
\lesssim~&(t+1)^{wp_1}2^{3j_1}(1+2^{j_1})^{8}2^{-l_{1,1}}
 \|\hat\vp_{j_1}\|_{L^\infty_\xi}\|\hat \vp_{j_2}\|_{L^\infty_\xi}\|\vp_{j_3}\|_{L^\infty}\\
\lesssim~&(t+1)^{(w+4)p_1}2^{j_1-l_{1,1}}\|2^{j_1/2}\hat\vp_{j_1}\|_{L^\infty_\xi}\|2^{j_2/2}\hat \vp_{j_2}\|_{L^\infty_\xi}\|\vp_{j_3}\|_{B^{1,6}},
\end{align*}

\begin{align*}
&\left\|(|\xi|^{r+1}+|\xi|^{w+1})\iint_{\R^2}    e^{it\Phi(\eta_1,\eta_2, \xi)}\left[\frac{\Tb_1(\eta_1, \eta_2, \xi - \eta_1 - \eta_2)}{\partial_{\eta_1}\Phi(\eta_1,\eta_2,\xi)}\right] \cutoff_{\lb}^1\partial_{\eta_1}\psi_{j_1-3}(\eta_1+\eta_2-2\xi) \right.\\
&\qquad\left. \hat h_{j_1}(\eta_1)\hat h_{j_2}(\eta_2) \hat h_{j_3}(\xi-\eta_1-\eta_2) \diff{\eta_1} \diff{\eta_2}\right\|_{L^\infty_\xi}\\
\lesssim~&(t+1)^{wp_1}2^{j_1}\Big\{ [(1+2^{2j_1})^{5/2}2^{-j_1}2^{-l_{1,2}}][2^{j_1+j_2+j_3}(1+2^{j_1})^{-1}(1+2^{3j_2})(1+2^{j_3})]2^{-j_{1}}\Big\}\|\vp_{j_1}\|_{L^2}\|\vp_{j_2}\|_{L^\infty}\|\vp_{j_3}\|_{L^2}\\
\lesssim~&(t+1)^{wp_1}2^{2j_1}(1+2^{j_1})^82^{-l_{1,2}}\|\vp_{j_1}\|_{L^2}\|\vp_{j_2}\|_{L^\infty}\|\vp_{j_3}\|_{L^2}\\
\lesssim~&(t+1)^{(w+4)p_1}2^{j_1-l_{1,2}}\|\vp_{j_1}\|_{L^2}\|\vp_{j_2}\|_{B^{1,6}}\|\vp_{j_3}\|_{L^2},
\end{align*}

\begin{align*}
&\left\| (|\xi|^{r+1}+|\xi|^{w+1}) \iint_{\R^2} e^{it\Phi(\eta_1,\eta_2, \xi)}\frac{\Tb_1(\eta_1, \eta_2, \xi - \eta_1 - \eta_2)}{\partial_{\eta_1}\Phi(\eta_1,\eta_2,\xi)} \cutoff_{\lb}^1\psi_{j-3}(\eta_1+\eta_2-2\xi)\right.\\
&\qquad\left. \partial_{\eta_1}\hat h_{j_1}(\eta_1)\hat h_{j_2}(\eta_2) \hat h_{j_3}(\xi-\eta_1-\eta_2) \diff{\eta_1} \diff{\eta_2}\right\|_{L^\infty_\xi}\\
\lesssim~&(t+1)^{wp_1}2^{j_1}[(1+2^{2j_1})^{5/2}2^{-j_1}2^{-l_{1,2}}][2^{j_1+j_2+j_3}(1+2^{j_1})^{-1}(1+2^{3j_2})(1+2^{j_3})]\|\partial_\xi\hat h_{j_1}\|_{L^2}\|\vp_{j_2}\|_{L^\infty}\|\vp_{j_3}\|_{L^2}\\
\lesssim~&(t+1)^{wp_1} 2^{3j_1}(1+2^{j_1})^82^{-l_{1,2}} \|\partial_\xi\hat h_{j_1}\|_{L^2}\|\vp_{j_2}\|_{L^\infty}\|\vp_{j_3}\|_{L^2}\\
\lesssim~&(t+1)^{(w+4)p_1}2^{j_1-l_{1,2}} \|\partial_\xi\hat h_{j_1}\|_{L^2}\|\vp_{j_2}\|_{B^{1,6}}\|\vp_{j_3}\|_{L^2}.
\end{align*}
Similarly,
\begin{align*}
&\left\| (|\xi|^{r+1}+|\xi|^{w+1}) \iint_{\R^2} e^{it\Phi(\eta_1,\eta_2, \xi)}\frac{\Tb_1(\eta_1, \eta_2, \xi - \eta_1 - \eta_2)}{\partial_{\eta_1}\Phi(\eta_1,\eta_2,\xi)} \cutoff_{\lb}^1\psi_{j-3}(\eta_1+\eta_2-2\xi)\right.\\
&\qquad\left. \hat h_{j_1}(\eta_1)\hat h_{j_2}(\eta_2) \partial_{\eta_1}\hat h_{j_3}(\xi-\eta_1-\eta_2) \diff{\eta_1} \diff{\eta_2}\right\|_{L^\infty_\xi}\\
\lesssim~&(t+1)^{(w+4)p_1}2^{j_1-l_{1,2}} \|\vp_{j_1}\|_{L^2}\|\vp_{j_2}\|_{B^{1,6}}\|\partial_\xi\hat h_{j_3}\|_{L^2}.
\end{align*}

Then we take the summation for $l_{1,1}$ and $l_{1,2}$ from $\log_2(\varrho(t))$ to $j_1+2$, considering the support of the cut-off functions, 
\begin{align*}
&\Bigg\|(|\xi|^{r+1}+|\xi|^{w+1})\sum\limits_{\min\{l_{1,1}, l_{1,2}\}\geq \log_2[\varrho(t)]}\iint_{\R^2} \cutoff_{\lb}^{1} \Tb_1(\eta_1, \eta_2, \xi - \eta_1 - \eta_2) e^{it\Phi(\eta_1,\eta_2, \xi)}\psi_{j-3}(\eta_1+\eta_2-2\xi)\\
&\qquad\qquad\qquad\qquad \hat h_{j_1}^{\iota_1}(\eta_1)\hat h_{j_2}^{\iota_2}(\eta_2) \hat h_{j_3}(\xi-\eta_1-\eta_2) \upsilon_{\iota_3}(\xi) \diff{\eta_1}\diff{\eta_2}\Bigg\|_{L^\infty_\xi}\\
\lesssim~&(|j_1|+|\log_2[\varrho(t)]|)2^{j_1}\max\{[\varrho(t)]^{-1}, 2^{-j_1}\}(t+1)^{(w+6)p_1-1} \|\vp_{j_1}\|_{L^2}\|\vp_{j_2}\|_{B^{1,6}}(\|\partial_\xi\hat h_{j_3}\|_{L^2}+\|\vp_{j_3}\|_{L^2})\\
&~~+(|j_1|+|\log_2[\varrho(t)]|)2^{j_1}\max\{[\varrho(t)]^{-1}, 2^{-j_1}\}(t+1)^{(w+6)p_1-1}\|2^{j_1/2}\hat\vp_{j_1}\|_{L^\infty_\xi}\|2^{j_2/2}\hat \vp_{j_2}\|_{L^\infty_\xi}\|\vp_{j_3}\|_{B^{1,6}}.
\end{align*}
Notice that $|\log_2[\varrho(t)]|=O(\log_2(t+1))$ and $|2^{j_1}|\lesssim (t+1)^{p_1}$. So we have
\begin{align*}
&\Bigg\|(|\xi|^{r+1}+|\xi|^{w+1})\sum\limits_{\min\{l_{1,1}, l_{1,2}\}\geq \log_2[\varrho(t)]}\iint_{\R^2} \cutoff_{\lb}^{1} \Tb_1(\eta_1, \eta_2, \xi - \eta_1 - \eta_2) e^{it\Phi(\eta_1,\eta_2, \xi)}\psi_{j-3}(\eta_1+\eta_2-2\xi)\\
&\qquad\qquad\qquad\qquad \hat h_{j_1}^{\iota_1}(\eta_1)\hat h_{j_2}^{\iota_2}(\eta_2) \hat h_{j_3}(\xi-\eta_1-\eta_2) \upsilon_{\iota_3}(\xi) \diff{\eta_1}\diff{\eta_2}\Bigg\|_{L^\infty_\xi}\\
\lesssim~&(t+1)^{(w+8)p_1-0.51}\Big\{ \|\vp_{j_1}\|_{L^2}\|\vp_{j_2}\|_{B^{1,6}}(\|\partial_\xi\hat h_{j_3}\|_{L^2}+\|\vp_{j_3}\|_{L^2})+\|2^{j_1/2}\hat\vp_{j_1}\|_{L^\infty_\xi}\|2^{j_2/2}\hat \vp_{j_2}\|_{L^\infty_\xi}\|\vp_{j_3}\|_{B^{1,6}}\Big\}.
\end{align*}
Taking summation for corresponding $j_1, j_2, j_3$ and using the bootstrap assumption, we have
\begin{align*}
&\sum\limits_{j_1,j_2,j_3}\Bigg\|(|\xi|^{r+1}+|\xi|^{w+1})\sum\limits_{\min\{l_{1,1}, l_{1,2}\}\geq \log_2[\varrho(t)]}\iint_{\R^2} \cutoff_{\lb}^{1} \Tb_1(\eta_1, \eta_2, \xi - \eta_1 - \eta_2) e^{it\Phi(\eta_1,\eta_2, \xi)}\psi_{j-3}(\eta_1+\eta_2-2\xi)\\
&\qquad\qquad\qquad\qquad \hat h_{j_1}^{\iota_1}(\eta_1)\hat h_{j_2}^{\iota_2}(\eta_2) \hat h_{j_3}(\xi-\eta_1-\eta_2) \upsilon_{\iota_3}(\xi) \diff{\eta_1}\diff{\eta_2}\Bigg\|_{L^\infty_\xi}\\
\lesssim~&(t+1)^{(w+8)p_1-0.55-0.5}\ve_1^3,
\end{align*}
which is integrable in time $t\in(0,\infty)$. 

The other term to be estimated is \eqref{eq8.7}. We consider the case when $\max\{l_{2,1}, l_{2,2}\}\geq \log_2(\varrho(t))$.
 Since $\eta_1$ and $\eta_2$ are symmetric, we assume $l_{2,1}\geq l_{2,2}$. We integrate by part with respect to $\eta_1$ and obtain
\begin{equation}
\begin{aligned}
& \iint_{\R^2}\partial_{\eta_1}[\cutoff_{\lb}^{2}(1-\psi_{j-3}(\eta_1+\eta_2-2\xi))\frac{\Tb_1(\eta_1,\eta_2,\xi-\eta_1-\eta_2)}{it\partial_{\eta_1}\Phi} ]e^{it\Phi}\hat h_{j_1}^{\iota_1}(\eta_1)\hat h_{j_2}^{\iota_2}(\eta_2)\hat h_{j_3}(\xi-\eta_1-\eta_2)\upsilon_{\iota_3}(\xi)\diff \eta_1\diff \eta_2\\
&+  \iint_{\R^2}\cutoff_{\lb}^{2}(1-\psi_{j-3}(\eta_1+\eta_2-2\xi))\frac{\Tb_1(\eta_1,\eta_2,\xi-\eta_1-\eta_2)}{it\partial_{\eta_1}\Phi} e^{it\Phi} \partial_{\eta_1}\hat h_{j_1}^{\iota_1}(\eta_1)\hat h_{j_2}^{\iota_2}(\eta_2)\hat h_{j_3}(\xi-\eta_1-\eta_2)\upsilon_{\iota_3}(\xi)\diff \eta_1\diff \eta_2\\
&+  \iint_{\R^2}\cutoff_{\lb}^{2}(1-\psi_{j-3}(\eta_1+\eta_2-2\xi))\frac{\Tb_1(\eta_1,\eta_2,\xi-\eta_1-\eta_2)}{it\partial_{\eta_1}\Phi} e^{it\Phi} \hat h_{j_1}^{\iota_1}(\eta_1)\hat h_{j_2}^{\iota_2}(\eta_2)\partial_{\eta_1}\hat h_{j_3}(\xi-\eta_1-\eta_2)\upsilon_{\iota_3}(\xi)\diff \eta_1\diff \eta_2.
\end{aligned}\end{equation}
 
By the symbol estimates Proposition \ref{Prop_symb_Tmu} and
\begin{multline*}
\left\|\psi_{j_1}(\eta_1)\psi_{j_2}(\eta_2)\psi_{j_3}(\eta_3)[1-\psi_{j-3}(\eta_1+\eta_2-2\xi)]\upsilon_{\iota_1}(\eta_1)\upsilon_{\iota_2}(\eta_2)\upsilon_{\iota_3}(\eta_1+\eta_2+\eta_3)\cutoff_{\lb}^2\right.\\
\left.\cdot\left[\frac{1}{(1+\eta_1^2)^{3/2}}-\frac{1}{[1+\eta_3^2]^{3/2}} \right]^{-1}\right\|_{S^\infty}\lesssim (1+2^{2j_1})^{5/2}2^{-j_1}2^{-l_{2,1}},
\end{multline*}
\begin{multline*}
\left\|\psi_{j_1}(\eta_1)\psi_{j_2}(\eta_2)\psi_{j_3}(\eta_3)[1-\psi_{j-3}(\eta_1+\eta_2-2\xi)]\upsilon_{\iota_1}(\eta_1)\upsilon_{\iota_2}(\eta_2)\upsilon_{\iota_3}(\eta_1+\eta_2+\eta_3)\cutoff_{\lb}^2\right.\\
\left.\cdot\left[-\frac{3\eta_1}{(1+\eta_1^2)^{5/2}}-\frac{3\eta_3}{(1+\eta_3^2)^{5/2}}\right]\right\|_{S^\infty}
\lesssim 2^{j_{1}}(1+2^{2j_1})^{-5/2},
\end{multline*}
we obtain the estimates of each integrals as following 
\begin{align*}
&\left\|(|\xi|^{r+1}+|\xi|^{w+1})\iint_{\R^2}    e^{it\Phi(\eta_1,\eta_2, \xi)}\partial_{\eta_1} \left[\frac{\Tb_1(\eta_1, \eta_2, \xi - \eta_1 - \eta_2)}{\partial_{\eta_1}\Phi(\eta_1,\eta_2,\xi)}\right]\cutoff_{\lb}^2(1-\psi_{j-3}(\eta_1+\eta_2-2\xi))\right.\\
& \qquad\left. \hat h_{j_1}(\eta_1)\hat h_{j_2}(\eta_2) \hat h_{j_3}(\xi-\eta_1-\eta_2) \diff{\eta_1} \diff{\eta_2}\right\|_{L^\infty_\xi}\\
\lesssim~&(t+1)^{wp_1}2^{j_1}\Big\{ [(1+2^{2j_1})^{5/2}2^{-j_1}2^{-l_{2,1}}]^2[2^{j_{1}}(1+2^{2j_1})^{-5/2}][2^{j_1+j_2+j_3}(1+2^{j_1})^{-1}(1+2^{3j_2})(1+2^{j_3})]\\
&+[(1+2^{2j_1})^{5/2}2^{-j_1}2^{-l_{2,1}}][2^{j_2+j_3}(1+2^{5j_2})(1+2^{j_3})]\Big\}\|\psi_{l_{2,1}}(\eta_1-\frac\xi3)\hat\vp_{j_1}(\eta_1)\|_{L^2_{\eta_1}L^\infty_\xi}\\
&\qquad\cdot\|\psi_{l_{2,2}}(\eta_2-\frac\xi3)\hat \vp_{j_2}(\eta_2)\|_{L^2_{\eta_2}L^\infty_\xi}\|\vp_{j_3}\|_{L^\infty}\\
\lesssim~&(t+1)^{wp_1}[(1+2^{j_1})^{11}2^{-l_{2,1}}2^{2j_1}]\|\hat\vp_{j_1}\|_{L^\infty_\xi}\|\vp_{j_2}\|_{L^\infty_\xi}\|\vp_{j_3}\|_{L^\infty}\\
\lesssim~&(t+1)^{(w+7)p_1}2^{j_1-l_{2,1}}\|2^{j_1/2}\hat\vp_{j_1}\|_{L^\infty_\xi}\|2^{j_2/2}\hat\vp_{j_2}\|_{L^\infty_\xi}\|\vp_{j_3}\|_{B^{1,6}},
\end{align*}
\begin{align*}
&\left\|(|\xi|^{r+1}+|\xi|^{w+1})\iint_{\R^2}    e^{it\Phi(\eta_1,\eta_2, \xi)}\left[\frac{\Tb_1(\eta_1, \eta_2, \xi - \eta_1 - \eta_2)}{\partial_{\eta_1}\Phi(\eta_1,\eta_2,\xi)}\right]\partial_{\eta_1} \cutoff_{\lb}^2[1-\psi_{j-3}(\eta_1+\eta_2-2\xi)]\right.\\
&\qquad\left.  \hat h_{j_1}(\eta_1)\hat h_{j_2}(\eta_2) \hat h_{j_3}(\xi-\eta_1-\eta_2) \diff{\eta_1} \diff{\eta_2}\right\|_{L^\infty_\xi}\\
\lesssim~&(t+1)^{wp_1}2^{j_1}\Big\{ [(1+2^{2j_1})^{5/2}2^{-j_1}2^{-l_{2,1}}][2^{j_1+j_2+j_3}(1+2^{j_1})^{-1}(1+2^{3j_2})(1+2^{j_3})]2^{-l_{2,1}}\Big\}\\
&\qquad\cdot\|\psi_{l_{2,1}}(\eta_1-\xi)\hat\vp_{j_1}(\eta_1)\|_{L^2_{\eta_1}L^\infty_\xi}\|\psi_{l_{2,2}}(\eta_2-\xi)\hat \vp_{j_2}(\eta_2)\|_{L^2_{\eta_2}L^\infty_\xi}\|\vp_{j_3}\|_{L^\infty}\\
\lesssim~&(t+1)^{wp_1}2^{3j_1}(1+2^{j_1})^82^{-l_{2,1}}\|\hat\vp_{j_1}\|_{L^\infty_\xi}\|\hat \vp_{j_2}\|_{L^\infty_\xi}\|\vp_{j_3}\|_{L^\infty}\\
\lesssim~&(t+1)^{(w+5)p_1}2^{j_1-l_{2,1}}\|2^{j_1/2}\hat\vp_{j_1}\|_{L^\infty_\xi}\|2^{j_2/2}\hat \vp_{j_2}\|_{L^\infty_\xi}\|\vp_{j_3}\|_{B^{1,6}},
\end{align*}
and
\begin{align*}
&\left\|(|\xi|^{r+1}+|\xi|^{w+1})\iint_{\R^2}  e^{it\Phi(\eta_1,\eta_2, \xi)}\frac{\Tb_1(\eta_1, \eta_2, \xi - \eta_1 - \eta_2)}{\partial_{\eta_1}\Phi(\eta_1,\eta_2,\xi)}\cutoff_{\lb}^2[1-\psi_{j-3}(\eta_1+\eta_2-2\xi) ]\right.\\
&\qquad\left.  \partial_{\eta_1}\hat h_{j_1}(\eta_1)\hat h_{j_2}(\eta_2) \hat h_{j_3}(\xi-\eta_1-\eta_2)\diff{\eta_1} \diff{\eta_2}\right\|_{L^\infty_\xi}\\
\lesssim~&(t+1)^{wp_1}2^{j_1}[(1+2^{2j_1})^{5/2}2^{-j_1}2^{-l_{2,1}}][2^{j_1+j_2+j_3}(1+2^{j_1})^{-1}(1+2^{3j_2})(1+2^{j_3})]\|\partial_\xi\hat h_{j_1}\|_{L^2}\|\vp_{j_2}\|_{L^\infty}\|\vp_{j_3}\|_{L^2}\\
\lesssim~&(t+1)^{wp_1}2^{3j_1}(1+2^{j_1})^82^{-l_{2,1}} \|\partial_\xi\hat h_{j_1}\|_{L^2}\|\vp_{j_2}\|_{L^\infty}\|\vp_{j_3}\|_{L^2}\\
\lesssim~&(t+1)^{(w+5)p_1}2^{j_1-l_{2,1}} \|\partial_\xi\hat h_{j_1}\|_{L^2}\|\vp_{j_2}\|_{B^{1,6}}\|\vp_{j_3}\|_{L^2}.
\end{align*}
Similarly,
\begin{align*}
&\left\|(|\xi|^{r+1}+|\xi|^{w+1})\iint_{\R^2}  e^{it\Phi(\eta_1,\eta_2, \xi)}\frac{\Tb_1(\eta_1, \eta_2, \xi - \eta_1 - \eta_2)}{\partial_{\eta_1}\Phi(\eta_1,\eta_2,\xi)}\cutoff_{\lb}^2[1-\psi_{j-3}(\eta_1+\eta_2-2\xi) ]\right.\\
&\qquad\left.  \hat h_{j_1}(\eta_1)\hat h_{j_2}(\eta_2)\partial_{\eta_1} \hat h_{j_3}(\xi-\eta_1-\eta_2)\diff{\eta_1} \diff{\eta_2}\right\|_{L^\infty_\xi}\\
\lesssim~&(t+1)^{(w+5)p_1}2^{j_1-l_{2,1}} \|\vp_{j_1}\|_{L^2}\|\vp_{j_2}\|_{B^{1,6}}\|\partial_\xi\hat h_{j_3}\|_{L^2}.
\end{align*}

Then we take the summation for $l_{2,1}$ and $l_{2,2}$ from $\log_2\varrho(t)$ to $j_1+2$, considering the support of the cut-off functions, 
\begin{align*}
&\Bigg\|(|\xi|^{r+1}+|\xi|^{w+1})\sum\limits_{\max\{l_{2,1}, l_{2,2}\}\geq \log_2[\rho(t)]}\iint_{\R^2} \cutoff_{\lb}^{2} \Tb_1(\eta_1, \eta_2, \xi - \eta_1 - \eta_2) e^{it\Phi(\eta_1,\eta_2, \xi)}[1-\psi_{j-3}(\eta_1+\eta_2-2\xi)] \\
&\qquad\qquad\qquad\qquad  \hat h_{j_1}^{\iota_1}(\eta_1)\hat h_{j_2}^{\iota_2}(\eta_2) \hat h_{j_3}(\xi-\eta_1-\eta_2)\upsilon_{\iota_3}(\xi)\diff{\eta_1} \diff{\eta_2}\Bigg\|_{L^\infty_\xi}\\
\lesssim~&(t+1)^{(w+7)p_1-1}2^{j_1}(|j_1|+|\log_2[\varrho(t)]|)\max\{[\varrho(t)]^{-1}, 2^{-j_1}\}\Big(\|2^{j_1/2}\hat\vp_{j_1}\|_{L^\infty_\xi}\|2^{j_2/2}\hat\vp_{j_2}\|_{L^\infty_\xi}\|\vp_{j_3}\|_{B^{1,6}}\\
&\qquad+\|\partial_\xi\hat h_{j_1}\|_{L^2}\|\vp_{j_2}\|_{B^{1,6}}\|\vp_{j_3}\|_{L^2}+\|\vp_{j_1}\|_{L^2}\|\vp_{j_2}\|_{B^{1,6}}\|\partial_\xi\hat h_{j_3}\|_{L^2}\Big)\\
\lesssim~&(t+1)^{(w+9)p_1-0.55}\Big(\|2^{j_1/2}\hat\vp_{j_1}\|_{L^\infty_\xi}\|2^{j_2/2}\hat\vp_{j_2}\|_{L^\infty_\xi}\|\vp_{j_3}\|_{B^{1,6}}\\
&\qquad+\|\partial_\xi\hat h_{j_1}\|_{L^2}\|2^{j_2}\vp_{j_2}\|_{L^\infty}\|\vp_{j_3}\|_{L^2}+\|\vp_{j_1}\|_{L^2}\|\vp_{j_2}\|_{B^{1,6}}\|\partial_\xi\hat h_{j_3}\|_{L^2}\Big).
\end{align*}
Taking summation for corresponding $j_1, j_2, j_3$ and using the bootstrap assumption, we have
\begin{align*}
&\sum\limits_{j_1,j_2,j_3}\Bigg\|(|\xi|^{r}+|\xi|^{w})\xi\sum\limits_{\min\{l_{2,1}, l_{2,2}\}\geq \log_2[\varrho(t)]}\iint_{\R^2} \cutoff_{\lb}^{1} \Tb_1(\eta_1, \eta_2, \xi - \eta_1 - \eta_2) e^{it\Phi(\eta_1,\eta_2, \xi)}[1-\psi_{j-3}(\eta_1+\eta_2-2\xi)]\\
&\qquad\qquad\qquad\qquad \hat h_{j_1}^{\iota_1}(\eta_1)\hat h_{j_2}^{\iota_2}(\eta_2) \hat h_{j_3}(\xi-\eta_1-\eta_2) \upsilon_{\iota_3}(\xi) \diff{\eta_1}\diff{\eta_2}\Bigg\|_{L^\infty_\xi}\\
\lesssim~&(t+1)^{(w+9)p_1-0.55-0.5+0.01}\ve_1^3,
\end{align*}
which is integrable in time $t\in(0,\infty)$.

{\noindent\bf 3. $(\iota_1, \iota_2, \iota_3)=(+, -, +)$ or $(-, +, -)$. } 
We define the cut-off function 
\begin{align*}
\cutoff_{\lb}^{3}(\eta_1, \eta_2, \xi):=&\psi_{l_{3,1}}(\eta_1-\xi)\psi_{l_{3,2}}(\eta_1+\eta_2),
\end{align*}
where $\lb=(l_{3,1}, l_{3,2})\in \Z^2$. 

Considering the support of the integrand, we write the integral as
\begin{align*}
&\iint_{\R^2}  \Tb_1(\eta_1, \eta_2, \xi - \eta_1 - \eta_2) e^{it\Phi(\eta_1,\eta_2, \xi)} \hat h_{j_1}^{\iota_1}(\eta_1)\hat h_{j_2}^{\iota_2}(\eta_2) \hat h_{j_3}(\xi-\eta_1-\eta_2)\upsilon_{\iota_3}(\xi) \diff{\eta_1} \diff{\eta_2}\\
=~&\sum\limits_{\max\{l_{3,1}, l_{3,2}\}\leq j+2}\iint_{\R^2} \cutoff_{\lb}^{3} \Tb_1(\eta_1, \eta_2, \xi - \eta_1 - \eta_2) e^{it\Phi(\eta_1,\eta_2, \xi)} \hat h_{j_1}^{\iota_1}(\eta_1)\hat h_{j_2}^{\iota_2}(\eta_2) \hat h_{j_3}(\xi-\eta_1-\eta_2)\upsilon_{\iota_3}(\xi)  \diff{\eta_1} \diff{\eta_2}.
 \end{align*}
When $l_{3,1} \geq l_{3,2}$, using \eqref{eqn718} and doing integration by part, we obtain
\begin{align*}
& \iint_{\R^2} \cutoff_{\lb}^{3}\Tb_1(\eta_1,\eta_2,\xi-\eta_1-\eta_2) e^{it\Phi} \hat h_{j_1}^{\iota_1}(\eta_1)\hat h_{j_2}^{\iota_2}(\eta_2)\hat h_{j_3}(\xi-\eta_1-\eta_2)\upsilon_{\iota_3}(\xi)\diff \eta_1\diff \eta_2\\
=~&-\frac it  \iint_{\R^2} \cutoff_{\lb}^{3}\Tb_1(\eta_1,\eta_2,\xi-\eta_1-\eta_2) \frac{1}{p'(\eta_1)-p'(\eta_2)}(\partial_{\eta_1}-\partial_{\eta_2})e^{it\Phi} \hat h_{j_1}^{\iota_1}(\eta_1)\hat h_{j_2}^{\iota_2}(\eta_2)\hat h_{j_3}(\xi-\eta_1-\eta_2)\upsilon_{\iota_3}(\xi)\diff \eta_1\diff \eta_2\\\nonumber
=~&\frac it  \iint_{\R^2}e^{it\Phi} (\partial_{\eta_1}-\partial_{\eta_2})\left[ \cutoff_{\lb}^{3} \frac{\Tb_1(\eta_1,\eta_2,\xi-\eta_1-\eta_2)}{p'(\eta_1)-p'(\eta_2)}\hat h_{j_1}^{\iota_1}(\eta_1)\hat h_{j_2}^{\iota_2}(\eta_2)\hat h_{j_3}(\xi-\eta_1-\eta_2)\right]\upsilon_{\iota_3}(\xi)\diff \eta_1\diff \eta_2\\\nonumber
=~&\frac it  \iint_{\R^2}e^{it\Phi}\Big[ \cutoff_{\lb}^{3} (\partial_1-\partial_2)\Tb_1(\eta_1,\eta_2,\xi-\eta_1-\eta_2) +(\partial_1-\partial_2) \cutoff_{\lb}^{3}\Tb_1(\eta_1,\eta_2,\xi-\eta_1-\eta_2)  \\
&- \cutoff_{\lb}^{3}\Tb_1(\eta_1,\eta_2,\xi-\eta_1-\eta_2) \frac{p''(\eta_1)+p''(\eta_2)}{p'(\eta_1)-p'(\eta_2)}\Big]\frac{1}{p'(\eta_1)-p'(\eta_2)}\hat h_{j_1}^{\iota_1}(\eta_1)\hat h_{j_2}^{\iota_2}(\eta_2)\hat h_{j_3}(\xi-\eta_1-\eta_2)\upsilon_{\iota_3}(\xi)\diff \eta_1\diff \eta_2\\\nonumber
&+\frac it  \iint_{\R^2}e^{it\Phi} \cutoff_{\lb}^{3} \frac{\Tb_1(\eta_1,\eta_2,\xi-\eta_1-\eta_2)}{p'(\eta_1)-p'(\eta_2)}(\partial_1-\partial_2)\left[\hat h_{j_1}^{\iota_1}(\eta_1)\hat h_{j_2}^{\iota_2}(\eta_2)\right]\hat h_{j_3}(\xi-\eta_1-\eta_2)\upsilon_{\iota_3}(\xi)\diff \eta_1\diff \eta_2.
\end{align*}
By the symbol estimates Proposition \ref{Prop_symb_Tmu} and Proposition \ref{PropA3}, we obtain the estimates for each integral above.
\begin{align*}
&\left\|(|\xi|^{r+1}+|\xi|^{w+1}) \iint_{\R^2}e^{it\Phi} \cutoff_{\lb}^{3} (\partial_1-\partial_2)\Tb_1(\eta_1,\eta_2,\xi-\eta_1-\eta_2)\frac{1}{p'(\eta_1)-p'(\eta_2)}\right.\\
&\qquad\left.\hat h_{j_1}^{\iota_1}(\eta_1)\hat h_{j_2}^{\iota_2}(\eta_2)\hat h_{j_3}(\xi-\eta_1-\eta_2)\upsilon_{\iota_3}(\xi)\diff \eta_1\diff \eta_2\right\|_{L^\infty_\xi}\\
\lesssim~&(t+1)^{wp_1}2^{j_1}[(1+2^{2j_1})^{5/2}2^{-j_1}2^{- l_{3,1}}][2^{j_1+j_2+j_3}(1+2^{j_1})^{-1}(1+2^{5j_2})(1+2^{j_3})]\|\vp_{j_1}\|_{L^\infty}\|\vp_{j_2}\|_{L^2}\|\vp_{j_3}\|_{L^2}\\
\lesssim~&(t+1)^{wp_1}2^{3j_1}(1+2^{j_1})^{10}2^{- l_{3,1}}\|\vp_{j_1}\|_{L^\infty}\|\vp_{j_2}\|_{L^2}\|\vp_{j_3}\|_{L^2}\\
\lesssim~&(t+1)^{(w+7)p_1}2^{j_1- l_{3,1}}\|\vp_{j_1}\|_{B^{1,6}}\|\vp_{j_2}\|_{L^2}\|\vp_{j_3}\|_{L^2},
\end{align*}
\begin{align*}
&\left\| (|\xi|^{r+1}+|\xi|^{w+1})\iint_{\R^2}e^{it\Phi} (\partial_1-\partial_2)\cutoff_{\lb}^{3} \Tb_1(\eta_1,\eta_2,\xi-\eta_1-\eta_2)\frac{1}{p'(\eta_1)-p'(\eta_2)}\right.\\
&\qquad\left.\cdot\hat h_{j_1}^{\iota_1}(\eta_1)\hat h_{j_2}^{\iota_2}(\eta_2)\hat h_{j_3}(\xi-\eta_1-\eta_2)\upsilon_{\iota_3}(\xi)\diff \eta_1\diff \eta_2\right\|_{L^\infty_\xi}\\
\lesssim~&(t+1)^{wp_1}2^{j_1}2^{-l_{3,1}}[2^{j_1+j_2+j_3}(1+2^{j_1})^{-1}(1+2^{3j_2})(1+2^{j_3})]
[(1+2^{2j_1})^{5/2}2^{-j_1}2^{-l_{3,1}}]\\
&\quad\cdot\|\psi_{l_{3,1}}(\eta_1-\xi)\hat\vp_{j_1}(\eta_1)\|_{L^2_{\eta_1}L^\infty_\xi}\|\psi_{l_{3,2}}(\xi-\eta_3)\hat\vp_{j_3}(\eta_3)\|_{L^2_{\eta_3}L^\infty_\xi}\|\vp_{j_2}\|_{L^\infty}\\
\lesssim~& (t+1)^{wp_1}2^{3j_1-l_{3,1}}(1+2^{j_1})^{8}\|\hat\vp_{j_1}\|_{L^\infty_\xi}\|\hat\vp_{j_3}\|_{L^\infty_\xi}\|\vp_{j_2}\|_{B^{1,6}}\\
\lesssim~& (t+1)^{(w+5)p_1}2^{j_1-l_{3,1}}\|2^{j_1/2}\hat\vp_{j_1}\|_{L^\infty_\xi}\|2^{j_3/2}\hat\vp_{j_3}\|_{L^\infty_\xi}\|\vp_{j_2}\|_{B^{1,6}},
\end{align*}
\begin{align*}
&\left\|(|\xi|^{r+1}+|\xi|^{w+1}) \iint_{\R^2}e^{it\Phi} \cutoff_{\lb}^{3} \Tb_1(\eta_1,\eta_2,\xi-\eta_1-\eta_2)\frac{(p''(\eta_1)+p''(\eta_2))}{(p'(\eta_1)-p'(\eta_2))^2}\right.\\
&\qquad\left.\cdot\hat h_{j_1}^{\iota_1}(\eta_1)\hat h_{j_2}^{\iota_2}(\eta_2)\hat h_{j_3}(\xi-\eta_1-\eta_2)\upsilon_{\iota_3}(\xi)\diff \eta_1\diff \eta_2\right\|_{L^\infty_\xi}\\
\lesssim~&(t+1)^{wp_1}2^{j_1}[(1+2^{2j_1})^{5/2}2^{-j_1}2^{-l_{3,1}}][2^{j_1+j_2+j_3}(1+2^{j_1})^{-1}(1+2^{3j_2})(1+2^{j_3})]\|\vp_{j_1}\|_{L^\infty}\|\vp_{j_2}\|_{L^2}\|\vp_{j_3}\|_{L^2}\\
\lesssim~&(t+1)^{wp_1}2^{3j_1}[(1+2^{j_1})^{8}2^{-l_{3,1}}\|\vp_{j_1}\|_{L^\infty}\|\vp_{j_2}\|_{L^2}\|\vp_{j_3}\|_{L^2}\\
\lesssim~&(t+1)^{(w+5)p_1}2^{j_1- l_{3,1}}\|\vp_{j_1}\|_{B^{1,6}}\|\vp_{j_2}\|_{L^2}\|\vp_{j_3}\|_{L^2},
\end{align*}
\begin{align*}
&\Bigg\|(|\xi|^{r+1}+|\xi|^{w+1})\iint_{\R^2}e^{it\Phi} \cutoff_{\lb}^{3} \frac{\Tb_1(\eta_1,\eta_2,\xi-\eta_1-\eta_2)}{p'(\eta_1)-p'(\eta_2)} (\partial_1-\partial_2)\left[\hat h_{j_1}^{\iota_1}(\eta_1)\hat h_{j_2}^{\iota_2}(\eta_2)\right]\hat h_{j_3}(\xi-\eta_1-\eta_2)\upsilon_{\iota_3}(\xi)\diff \eta_1\diff \eta_2\Bigg\|_{L^\infty_\xi}\\
\lesssim~& (t+1)^{wp_1}2^{j_1}[(1+2^{2j_1})^{5/2}2^{-j_1}2^{- l_{3,1}}][2^{j_1+j_2+j_3}(1+2^{j_1})^{-1}(1+2^{3j_2})(1+2^{j_3})]\\
&\quad \cdot\Big(\|\partial_{\eta_1}\hat h_{j_1}(\eta_1)\|_{L^2_{\eta_1}}\|\vp_{j_2}\|_{L^2}+\|\vp_{j_1}\|_{L^2}\|\partial_{\eta_2}\hat h_{j_2}(\eta_2)\|_{L^2_{\eta_2}}\Big)\|\vp_{j_3}\|_{L^\infty}\\
\lesssim~& (t+1)^{wp_1}2^{3j_1}(1+2^{j_1})^{8}2^{- l_{3,1}}\Big(\|\partial_{\eta_1}\hat h_{j_1}(\eta_1)\|_{L^2_{\eta_1}}\|\vp_{j_2}\|_{L^2}+\|\vp_{j_1}\|_{L^2}\|\partial_{\eta_2}\hat h_{j_2}(\eta_2)\|_{L^2_{\eta_2}}\Big)\|\vp_{j_3}\|_{L^\infty}\\
\lesssim~&(t+1)^{(w+5)p_1}2^{j_1- l_{3,1}}\Big(\|\partial_{\xi}\hat h_{j_1}\|_{L^2_{\xi}}\|\vp_{j_2}\|_{L^2}+\|\vp_{j_1}\|_{L^2}\|\partial_{\xi}\hat h_{j_2}\|_{L^2_{\xi}}\Big)\|\vp_{j_3}\|_{B^{1,6}}.
\end{align*}
When $l_{3,1} < l_{3,2}$, we integrate by part with $\eta_2$, 
\begin{align*}
& \iint_{\R^2} \cutoff_{\lb}^{3}\Tb_1(\eta_1,\eta_2,\xi-\eta_1-\eta_2) e^{it\Phi} \hat h_{j_1}^{\iota_1}(\eta_1)\hat h_{j_2}^{\iota_2}(\eta_2)\hat h_{j_3}(\xi-\eta_1-\eta_2)\upsilon_{\iota_3}(\xi)\diff \eta_1\diff \eta_2\\
=~&-\frac it  \iint_{\R^2} \cutoff_{\lb}^{3}\Tb_1(\eta_1,\eta_2,\xi-\eta_1-\eta_2) \frac{1}{p'(\eta_2)-p'(\xi-\eta_1-\eta_2)}\partial_{\eta_2}e^{it\Phi} \hat h_{j_1}^{\iota_1}(\eta_1)\hat h_{j_2}^{\iota_2}(\eta_2)\hat h_{j_3}(\xi-\eta_1-\eta_2)\upsilon_{\iota_3}(\xi)\diff \eta_1\diff \eta_2\\
=~&\frac it  \iint_{\R^2} e^{it\Phi}\partial_{\eta_2}\left[ \cutoff_{\lb}^{3} \frac{\Tb_1(\eta_1,\eta_2,\xi-\eta_1-\eta_2)}{p'(\eta_2)-p'(\xi-\eta_1-\eta_2)}\hat h_{j_2}^{\iota_2}(\eta_2)\hat h_{j_3}(\xi-\eta_1-\eta_2)\right]\hat h_{j_1}^{\iota_1}(\eta_1)\upsilon_{\iota_3}(\xi)\diff \eta_1\diff \eta_2\\
=~&\frac it  \iint_{\R^2} e^{it\Phi}\partial_{\eta_2} \cutoff_{\lb}^{3} \frac{\Tb_1(\eta_1,\eta_2,\xi-\eta_1-\eta_2)}{p'(\eta_2)-p'(\xi-\eta_1-\eta_2)}\hat h_{j_2}^{\iota_2}(\eta_2)\hat h_{j_3}(\xi-\eta_1-\eta_2)\hat h_{j_1}^{\iota_1}(\eta_1)\upsilon_{\iota_3}(\xi)\diff \eta_1\diff \eta_2\\
&+\frac it  \iint_{\R^2} e^{it\Phi} \cutoff_{\lb}^{3} \frac{\partial_{\eta_2}\Tb_1(\eta_1,\eta_2,\xi-\eta_1-\eta_2)}{p'(\eta_2)-p'(\xi-\eta_1-\eta_2)}\hat h_{j_2}^{\iota_2}(\eta_2)\hat h_{j_3}(\xi-\eta_1-\eta_2)\hat h_{j_1}^{\iota_1}(\eta_1)\upsilon_{\iota_3}(\xi)\diff \eta_1\diff \eta_2\\
&-\frac it  \iint_{\R^2} e^{it\Phi} \cutoff_{\lb}^{3}\Tb_1(\eta_1,\eta_2,\xi-\eta_1-\eta_2) \frac{p''(\eta_2)+p''(\xi-\eta_1-\eta_2)}{[p'(\eta_2)-p'(\xi-\eta_1-\eta_2)]^2}\hat h_{j_2}^{\iota_2}(\eta_2)\hat h_{j_3}(\xi-\eta_1-\eta_2)\hat h_{j_1}^{\iota_1}(\eta_1)\upsilon_{\iota_3}(\xi)\diff \eta_1\diff \eta_2\\
&+\frac it  \iint_{\R^2} e^{it\Phi} \cutoff_{\lb}^{3} \frac{\Tb_1(\eta_1,\eta_2,\xi-\eta_1-\eta_2)}{p'(\eta_2)-p'(\xi-\eta_1-\eta_2)}\partial_{\eta_2}\hat h_{j_2}^{\iota_2}(\eta_2)\hat h_{j_3}(\xi-\eta_1-\eta_2)\hat h_{j_1}^{\iota_1}(\eta_1)\upsilon_{\iota_3}(\xi)\diff \eta_1\diff \eta_2\\
&+\frac it  \iint_{\R^2} e^{it\Phi} \cutoff_{\lb}^{3} \frac{\Tb_1(\eta_1,\eta_2,\xi-\eta_1-\eta_2)}{p'(\eta_2)-p'(\xi-\eta_1-\eta_2)}\hat h_{j_2}^{\iota_2}(\eta_2)\partial_{\eta_2}\hat h_{j_3}(\xi-\eta_1-\eta_2)\hat h_{j_1}^{\iota_1}(\eta_1)\upsilon_{\iota_3}(\xi)\diff \eta_1\diff \eta_2.
\end{align*}
Then we estimate each integrals in the following. 
\begin{align*}
&\left\|(|\xi|^{r+1}+|\xi|^{w+1}) \iint_{\R^2} e^{it\Phi}\partial_{\eta_2} \cutoff_{\lb}^{3} \frac{\Tb_1(\eta_1,\eta_2,\xi-\eta_1-\eta_2)}{p'(\eta_2)-p'(\xi-\eta_1-\eta_2)}\hat h_{j_2}^{\iota_2}(\eta_2)\hat h_{j_3}(\xi-\eta_1-\eta_2)\hat h_{j_1}^{\iota_1}(\eta_1)\upsilon_{\iota_3}(\xi)\diff \eta_1\diff \eta_2\right\|_{L^\infty_\xi}\\
\lesssim~& (t+1)^{wp_1}2^{j_1}[(1+2^{2j_1})^{5/2}2^{-j_1}2^{-l_{3,2}}]2^{-l_{3,2}}[2^{j_1+j_2+j_3}(1+2^{j_1})^{-1}(1+2^{3j_2})(1+2^{j_3})]\\
&\quad\cdot\|\psi_{l_{3,1}}(\eta_1-\xi)\hat\vp_{j_1}(\eta_1)\|_{L^2_{\eta_1}L^\infty_\xi}\|\psi_{l_{3,2}}(\xi-\eta_3)\hat\vp_{j_3}(\eta_3)\|_{L^2_{\eta_3}L^\infty_\xi}\|\vp_{j_2}\|_{L^\infty}\\
\lesssim~&(t+1)^{wp_1}2^{3j_1}(1+2^{j_1})^{8}2^{- l_{3,2}}\|\hat\vp_{j_1}\|_{L^\infty_\xi}\|\hat\vp_{j_3}\|_{L^\infty_\xi}\|\vp_{j_2}\|_{L^\infty}\\
\lesssim~&(t+1)^{(w+5)p_1}2^{j_1- l_{3,2}}\|2^{j_1/2}\hat\vp_{j_1}\|_{L^\infty_\xi}\|2^{j_3/2}\hat\vp_{j_3}\|_{L^\infty_\xi}\|\vp_{j_2}\|_{B^{1,6}},
\end{align*}
\begin{align*}
&\left\|(|\xi|^{r+1}+|\xi|^{w+1}) \iint_{\R^2} e^{it\Phi} \cutoff_{\lb}^{3} \frac{\partial_{\eta_2}\Tb_1(\eta_1,\eta_2,\xi-\eta_1-\eta_2)}{p'(\eta_2)-p'(\xi-\eta_1-\eta_2)}\hat h_{j_2}^{\iota_2}(\eta_2)\hat h_{j_3}(\xi-\eta_1-\eta_2)\hat h_{j_1}^{\iota_1}(\eta_1)\upsilon_{\iota_3}(\xi)\diff \eta_1\diff \eta_2 \right\|_{L^\infty_\xi}\\
\lesssim~& (t+1)^{wp_1}2^{j_1}[(1+2^{2j_1})^{5/2}2^{-j_1}2^{- l_{3,2}}][2^{j_2+j_3}(1+2^{5j_2})(1+2^{j_3})]\|\vp_{j_1}\|_{L^\infty}\|\vp_{j_2}\|_{L^2}\|\vp_{j_3}\|_{L^2}\\
\lesssim~&(t+1)^{(w+7)p_1}2^{j_1- l_{3,2}}\|\vp_{j_1}\|_{B^{1,6}}\|\vp_{j_2}\|_{L^2}\|\vp_{j_3}\|_{L^2}.
\end{align*}
By the symbol estimate \eqref{eqnB10} and Proposition \ref{Prop_symb_Tmu},
\begin{align*}
&\left\|(|\xi|^{r+1}+|\xi|^{w+1})  \iint_{\R^2} e^{it\Phi} \cutoff_{\lb}^{3}\Tb_1(\eta_1,\eta_2,\xi-\eta_1-\eta_2) \frac{p''(\eta_2)+p''(\xi-\eta_1-\eta_2)}{[p'(\eta_2)-p'(\xi-\eta_1-\eta_2)]^2}\right.\\
&\qquad\left.\hat h_{j_1}^{\iota_1}(\eta_1)\hat h_{j_2}^{\iota_2}(\eta_2)\hat h_{j_3}(\xi-\eta_1-\eta_2)\upsilon_{\iota_3}(\xi)\diff \eta_1\diff \eta_2\right\|_{L^\infty_\xi}\\
\lesssim~&(t+1)^{wp_1}2^{j_1}[(1+2^{2j_1})^{5/2}2^{-j_1}2^{-l_{3,2}}][2^{j_2+j_3}(1+2^{3j_2})(1+2^{j_3})]\|\vp_{j_1}\|_{L^\infty}\|\vp_{j_2}\|_{L^2}\|\vp_{j_3}\|_{L^2}\\
\lesssim~& (t+1)^{(w+5)p_1}2^{j_1- l_{3,2}}\|\vp_{j_1}\|_{B^{1,6}}\|\vp_{j_2}\|_{L^2}\|\vp_{j_3}\|_{L^2},
\end{align*}

\begin{align*}
&\left\|(|\xi|^{r+1}+|\xi|^{w+1})\iint_{\R^2} e^{it\Phi} \cutoff_{\lb}^{3} \frac{\Tb_1(\eta_1,\eta_2,\xi-\eta_1-\eta_2)}{p'(\eta_2)-p'(\xi-\eta_1-\eta_2)}\partial_{\eta_2}\hat h_{j_2}^{\iota_2}(\eta_2)\hat h_{j_3}(\xi-\eta_1-\eta_2)\hat h_{j_1}^{\iota_1}(\eta_1)\upsilon_{\iota_3}(\xi)\diff \eta_1\diff \eta_2 \right\|_{L^\infty_\xi}\\
\lesssim~& (t+1)^{wp_1}2^{j_1}[(1+2^{2j_1})^{5/2}2^{-j_1}2^{- l_{3,2}}][2^{j_2+j_3}(1+2^{3j_2})(1+2^{j_3})]\|\vp_{j_1}\|_{L^2}\|\partial_{\xi}\hat h_{j_2}\|_{L^2_{\xi}}\|\vp_{j_3}\|_{L^\infty}\\
\lesssim~& (t+1)^{(w+5)p_1}2^{j_1- l_{3,2}}\|\vp_{j_1}\|_{L^2}\|\partial_{\xi}\hat h_{j_2}\|_{L^2_{\xi}}\|\vp_{j_3}\|_{B^{1,6}},
\end{align*}

\begin{align*}
&\left\|(|\xi|^{r+1}+|\xi|^{w+1})\iint_{\R^2} e^{it\Phi} \cutoff_{\lb}^{3} \frac{\Tb_1(\eta_1,\eta_2,\xi-\eta_1-\eta_2)}{p'(\eta_2)-p'(\xi-\eta_1-\eta_2)}\hat h_{j_2}^{\iota_2}(\eta_2)\partial_{\eta_2}\hat h_{j_3}(\xi-\eta_1-\eta_2)\hat h_{j_1}^{\iota_1}(\eta_1)\upsilon_{\iota_3}(\xi)\diff \eta_1\diff \eta_2 \right\|_{L^\infty_\xi}\\
\lesssim~& (t+1)^{wp_1}2^{j_1}[(1+2^{2j_1})^{5/2}2^{-j_1}2^{- l_{3,2}}][2^{j_2+j_3}(1+2^{3j_2})(1+2^{j_3})]\|\partial_{\xi}\hat h_{j_3}\|_{L^2_{\xi}}\|\vp_{j_2}\|_{L^2}\|\vp_{j_1}\|_{L^\infty}\\
\lesssim~&(t+1)^{(w+5)p_1}2^{j_1- l_{3,2}}\|\partial_{\xi}\hat h_{j_3}\|_{L^2_{\xi}}\|\vp_{j_2}\|_{L^2}\|\vp_{j_1}\|_{B^{1,6}}.
\end{align*}

Then we take the summation for $\log_2\varrho(t)<\max\{l_{3,1}, l_{3,2}\}$,
\begin{align*}
&\sum\limits_{l_{3,1}, l_{3,2}\geq\log_2[\varrho(t)] }\Bigg\|(|\xi|^{r+1}+|\xi|^{w+1})\\
&\hspace{3cm}\iint_{\R^2} \cutoff_{\lb}^{3} \Tb_1(\eta_1, \eta_2, \xi - \eta_1 - \eta_2) e^{it\Phi(\eta_1,\eta_2, \xi)} \hat h_{j_1}^{\iota_1}(\eta_1)\hat h_{j_2}^{\iota_2}(\eta_2) \hat h_{j_3}(\xi-\eta_1-\eta_2)\upsilon_{\iota_3}(\xi)  \diff{\eta_1} \diff{\eta_2}\Bigg\|_{L^\infty_\xi}\\
\lesssim~&(t+1)^{(w+7)p_1-1}2^{j_1}(|j_1|+|\log_2[\varrho(t)]|)\max\{[\varrho(t)]^{-1}, 2^{-j_1}\}\Big(\|2^{j_1/2}\hat\vp_{j_1}\|_{L^\infty_\xi}\|2^{j_2/2}\hat\vp_{j_2}\|_{L^\infty_\xi}\|\vp_{j_3}\|_{B^{1,6}}\\
&\qquad+\|\vp_{j_1}\|_{B^{1,6}}\|\vp_{j_2}\|_{L^2}\|\vp_{j_3}\|_{L^2}+\|\partial_\xi\hat h_{j_1}\|_{L^2}\|\vp_{j_2}\|_{B^{1,6}}\|\vp_{j_3}\|_{L^2}+\|\vp_{j_1}\|_{L^2}\|\vp_{j_2}\|_{B^{1,6}}\|\partial_\xi\hat h_{j_3}\|_{L^2}\Big)\\
\lesssim~&(t+1)^{(w+9)p_1-0.55}\Big(\|2^{j_1/2}\hat\vp_{j_1}\|_{L^\infty_\xi}\|2^{j_2/2}\hat\vp_{j_2}\|_{L^\infty_\xi}\|\vp_{j_3}\|_{B^{1,6}}+\|\vp_{j_1}\|_{B^{1,6}}\|\vp_{j_2}\|_{L^2}\|\vp_{j_3}\|_{L^2}\\
&\qquad+\|\partial_\xi\hat h_{j_1}\|_{L^2}\|2^{j_2}\vp_{j_2}\|_{L^\infty}\|\vp_{j_3}\|_{L^2}+\|\vp_{j_1}\|_{L^2}\|\vp_{j_2}\|_{B^{1,6}}\|\partial_\xi\hat h_{j_3}\|_{L^2}\Big).
 \end{align*}
 Taking the summation for corresponding $j_1, j_2, j_3$ and using the bootstrap assumption, we have
 \begin{equation}\label{eqn8.9}
 \begin{aligned}
 &\sum\limits_{j_1,j_2,j_3}\sum\limits_{ l_{3,1}, l_{3,2}\geq\log_2[\varrho(t)] }\Bigg\|(|\xi|^{r+1}+|\xi|^{w+1})\iint_{\R^2} \cutoff_{\lb}^{3} \Tb_1(\eta_1, \eta_2, \xi - \eta_1 - \eta_2) e^{it\Phi(\eta_1,\eta_2, \xi)} \\
 &\qquad \hat h_{j_1}^{\iota_1}(\eta_1)\hat h_{j_2}^{\iota_2}(\eta_2) \hat h_{j_3}(\xi-\eta_1-\eta_2)\upsilon_{\iota_3}(\xi)  \diff{\eta_1} \diff{\eta_2}\Bigg\|_{L^\infty_\xi}\\
 \lesssim~&(t+1)^{(w+9)p_1-0.55-0.5+0.01}\ve_1^3,
 \end{aligned}\end{equation}
 which is integrable in time $t\in(0,\infty)$.

{\noindent\bf 4. $(\iota_1, \iota_2, \iota_3)=(-, +, +)$ or $(+, -, -)$.} We define the cut-off function 
\begin{align*}
\cutoff_{\lb}^{4}(\eta_1, \eta_2, \xi):=&\psi_{l_{4,1}}(\eta_2-\xi)\psi_{l_{4,2}}(\eta_1+\eta_2),
\end{align*}
where $\lb=(l_{4,1}, l_{4,2})\in \Z^2$. 

Considering the support of the integrand, we write the integral as
\begin{align*}
&\iint_{\R^2}  \Tb_1(\eta_1, \eta_2, \xi - \eta_1 - \eta_2) e^{it\Phi(\eta_1,\eta_2, \xi)} \hat h_{j_1}^{\iota_1}(\eta_1)\hat h_{j_2}^{\iota_2}(\eta_2) \hat h_{j_3}(\xi-\eta_1-\eta_2)\upsilon_{\iota_3}(\xi) \diff{\eta_1} \diff{\eta_2}\\
=~&\sum\limits_{\max\{l_{4,1}, l_{4,2}\}\leq j+2}\iint_{\R^2} \cutoff_{\lb}^{4} \Tb_1(\eta_1, \eta_2, \xi - \eta_1 - \eta_2) e^{it\Phi(\eta_1,\eta_2, \xi)} \hat h_{j_1}^{\iota_1}(\eta_1)\hat h_{j_2}^{\iota_2}(\eta_2) \hat h_{j_3}(\xi-\eta_1-\eta_2)\upsilon_{\iota_3}(\xi)  \diff{\eta_1} \diff{\eta_2}.
 \end{align*}
Since $\eta_1$ and $\eta_2$ are symmetric in the expression of the integral, by interchanging $\eta_1$ and $\eta_2$, we can transform this case into the previous case. So we can obtain the similar estimate as \eqref{eqn8.9}.

\subsection{Space resonance estimate}
Then we estimate
\begin{align}\label{eqn89}
\iint_{\R^2}   \Tb_1(\eta_1, \eta_2, \xi - \eta_1 - \eta_2) e^{it\Phi(\eta_1,\eta_2, \xi)} \hat h_{j_1}^{\iota_1}(\eta_1)\hat h_{j_2}^{\iota_2}(\eta_2) \hat h_{j_3}(\xi-\eta_1-\eta_2)\upsilon_{\iota_3}(\xi)[1-\psi_{j-3}(\eta_1+\eta_2-2\xi)] \cutoff_{\lb}^{2} \diff{\eta_1} \diff{\eta_2},
 \end{align}
 when $(\iota_1, \iota_2, \iota_3)=(+, +, +)$ or $(-, -, -)$, and $\max\{l_{2,1}, l_{2,2}\}\leq \log_2\rho(t)=-0.3\log_2(t+1)$.
 In this case, we can expand $\Tb_1 / \Phi$ around $(\eta_1,\eta_2)=(\xi/3, \xi/3)$ as
\begin{equation}\label{T1Phi}
\frac{\Tb_1(\eta_1, \eta_2, \xi - \eta_1 - \eta_2)}{\Phi(\eta_1,\eta_2, \xi) }= 
\frac{9\Tb_1(\frac\xi3,\frac\xi3,\frac\xi3)\sqrt{1+\xi^2}\sqrt{1+\xi^2/9}(\sqrt{1+\xi^2}+\sqrt{1+\xi^2/9})}{8\xi^3}.
\end{equation}
We will prove 
\begin{equation}\label{eqn814}
 \left\|(|\xi|^{r+1}+|\xi|^{w+1})\int_0^\infty\sum\limits_{\max\{l_{2,1}, l_{2,2}\}\leq \log_2\rho(t)}\eqref{eqn89}\diff t\right\|_{L^\infty_\xi}<\infty.
\end{equation}
By interchanging the order of the summation and the time integral, it is equivalent to 
\[
 \sum\limits_{\max\{l_{2,1}, l_{2,2}\}\leq 0}\left\|(|\xi|^{r+1}+|\xi|^{w+1})\int_0^{t_\lb}\eqref{eqn89}\diff t\right\|_{L^\infty_\xi}<\infty.
\]
where $t_\lb:=2^{-\frac{10}3\max\{l_{2,1}, l_{2,2}\}}$.
After writing
\begin{align*}
e^{i\tau \Phi(\xi, \eta_1,\eta_2)}=\frac{1}{i\Phi(\xi, \eta_1,\eta_2)} \left[\partial_\tau e^{i\tau\Phi(\xi, \eta_1,\eta_2)}\right],
\end{align*}
and integrating by parts with respect to $\tau$ in each time interval between the time discontinuities, we get that
\[
\begin{aligned}
&\int_0^{t_\lb} \iint_{\R^2}  \Tb_1(\eta_1, \eta_2, \xi - \eta_1 - \eta_2)e^{i\tau\Phi(\eta_1,\eta_2, \xi)} \hat h_{j_1}^{\iota_1}(\eta_1)\hat h_{j_2}^{\iota_2}(\eta_2) \hat h_{j_3}(\xi-\eta_1-\eta_2)\upsilon_{\iota_3}(\xi)[1-\psi_{j-3}(\eta_1+\eta_2-2\xi)] \cutoff_{\lb}^{2} \diff{\eta_1} \diff{\eta_2} \diff{\tau}
\\[1ex]
&\qquad\qquad=  J_1-\int_0^{t_\lb} J_2(\tau)\diff{\tau},
\end{aligned}
\]
where
\begin{align*}
J_1:=\iint_{\R^2}   \frac{\Tb_1(\eta_1, \eta_2, \xi - \eta_1 - \eta_2)}{i\Phi(\xi, \eta_1,\eta_2)}  e^{i\tau\Phi(\xi, \eta_1,\eta_2)}&\hat h_{j_1}^{\iota_1}(\eta_1)\hat h_{j_2}^{\iota_2}(\eta_2) \hat h_{j_3}(\xi-\eta_1-\eta_2)\\
&\cdot\upsilon_{\iota_3}(\xi)[1-\psi_{j-3}(\eta_1+\eta_2-2\xi)] \cutoff_{\lb}^{2}\diff{\eta_1} \diff{\eta_2}\bigg|_{t=0}^{t_\lb},
\end{align*}
\begin{align*}
&J_2(\tau) :=  \iint_{\R^2} \frac{\Tb_1(\eta_1, \eta_2, \xi - \eta_1 - \eta_2)}{\Phi(\xi, \eta_1,\eta_2)}  e^{i\tau\Phi(\xi, \eta_1,\eta_2)} \partial_\tau \left[\hat h_{j_1}^{\iota_1}(\eta_1)\hat h_{j_2}^{\iota_2}(\eta_2) \hat h_{j_3}(\xi-\eta_1-\eta_2)\right][1-\psi_{j-3}(\eta_1+\eta_2-2\xi)] \cutoff_{\lb}^{2} \diff{\eta_1} \diff{\eta_2}.
\end{align*}
For $J_1$, we have from \eqref{T1Phi} that
\[
\begin{aligned}
&\bigg|(|\xi|^{r+1}+|\xi|^{w + 1})\iint_{\R^2}  \frac{\Tb_1(\eta_1, \eta_2, \xi - \eta_1 - \eta_2)}{\Phi(\xi, \eta_1,\eta_2)}  e^{i\tau\Phi(\xi, \eta_1,\eta_2)}\hat h_{j_1}^{\iota_1}(\eta_1)\hat h_{j_2}^{\iota_2}(\eta_2) \hat h_{j_3}(\xi-\eta_1-\eta_2)\upsilon_{\iota_3}(\xi)\\
&\qquad\qquad[1-\psi_{j-3}(\eta_1+\eta_2-2\xi)] \cutoff_{\lb}^{2}\diff{\eta_1} \diff{\eta_2}\bigg|\\
\lesssim~~&(2^{rj_1}+2^{wj_1}) 2^{j_1}[2^{3j_1}(1+2^{j_1})^3][2^{-3j_1}(1+2^{j_1})^{3}] 2^{l_{2,1}+l_{2,2}} \|\hat h_{j_1}^{\iota_1}(\eta_1)\hat h_{j_2}^{\iota_2}(\eta_2) \hat h_{j_3}(\xi-\eta_1-\eta_2)\|_{L^\infty_{\eta_1\eta_2\xi}}\\
\lesssim~&2^{1.4j_1}(1+2^{j_1})^{17}2^{l_{2,1}+l_{2,2}} \|\hat h_{j_1}^{\iota_1}(\eta_1)\hat h_{j_2}^{\iota_2}(\eta_2) \hat h_{j_3}(\xi-\eta_1-\eta_2)\|_{L^\infty_{\eta_1\eta_2\xi}}\\
\end{aligned}
\]
which are summable for $\max\{l_{2,1}, l_{2,2}\}\leq 0$.

After taking the time derivative, the term $J_2$ can be written as a sum of three terms.
\begin{align*}
& \iint_{\R^2} \frac{\Tb_1(\eta_1, \eta_2, \xi - \eta_1 - \eta_2)}{\Phi(\xi, \eta_1,\eta_2)}  e^{i\tau\Phi(\xi, \eta_1,\eta_2)}  \left[\partial_\tau\hat h_{j_1}(\eta_1, \tau)\hat h_{j_2}(\eta_2, \tau) \hat h_{j_3}(\xi-\eta_1-\eta_2, \tau)\right][1-\psi_{j-3}(\eta_1+\eta_2-2\xi)] \cutoff_{\lb}^{2} \diff{\eta_1} \diff{\eta_2},\\[1ex]
& \iint_{\R^2} \frac{\Tb_1(\eta_1, \eta_2, \xi - \eta_1 - \eta_2)}{\Phi(\xi, \eta_1,\eta_2)}  e^{i\tau\Phi(\xi, \eta_1,\eta_2)}  \left[\hat h_{j_1}(\eta_1, \tau) \partial_\tau \hat h_{j_2}(\eta_2, \tau) \hat h_{j_3}(\tau, \xi-\eta_1-\eta_2)\right][1-\psi_{j-3}(\eta_1+\eta_2-2\xi)] \cutoff_{\lb}^{2} \diff{\eta_1} \diff{\eta_2},
\end{align*}
and
\begin{align*}
& \iint_{\R^2} \frac{\Tb_1(\eta_1, \eta_2, \xi - \eta_1 - \eta_2)}{\Phi(\xi, \eta_1,\eta_2)}  e^{i\tau\Phi(\xi, \eta_1,\eta_2)}  \left[\hat h_{j_1}(\eta_1, \tau)\hat h_{j_2}(\eta_2, \tau) \partial_\tau\hat h_{j_3}(\xi-\eta_1-\eta_2, \tau)\right][1-\psi_{j-3}(\eta_1+\eta_2-2\xi)] \cutoff_{\lb}^{2}  \diff{\eta_1} \diff{\eta_2}.
\end{align*}
Notice that by \eqref{eqhhat},and the bootstrap assumptions and Lemma \ref{NonDis}, we have
\begin{align*}
\|\partial_t h_j\|_{L^\infty}\lesssim& \bigg\|\psi_j(\xi)\xi\iint_{\R^2}\Tb_1(\eta_1,\eta_2,\xi-\eta_1-\eta_2) e^{it\Phi(\eta_1,\eta_2, \xi)}\hat h(\xi-\eta_1-\eta_2)\hat h(\eta_1)\hat h(\eta_2)\diff\eta_1\diff\eta_2\bigg\|_{L^1_\xi}+\left\|\psi_j(\xi)\xi\widehat{\mathcal N_{\geq 5}(\vp)}\right\|_{L^1_\xi}\\
\lesssim~~ &\|\vp_j\|_{B^{1,1}} \sum\limits_{j=1}^{\infty} \left(\|\vp\|_{B^{1,6}}^{2 j } \right)\\
\lesssim~~ &\|\vp_j\|_{B^{1,1}}(t + 1)^{-1}\ve_1^2.
\end{align*}
Therefore, we obtain
\[
\begin{aligned}
&\int_0^{t_{\lb}}\left| (|\xi|^{r+1}+|\xi|^{w+ 1}) J_2(\tau)\right|\diff\tau\\
\lesssim &\sum t_{\lb}^{(w-r)p_1}2^{(r+1)j_1}\int_0^{t_{\lb}}(1+2^{6j_1}) \| \psi_{l_{2,1}}(\eta_1-\frac\xi3)\hat h_{\ell_1}(\eta_1)\|_{L^2_{\eta_1}L^\infty_\xi}\|\psi_{l_{2,2}}(\eta_2-\frac\xi3)\hat h_{\ell_2}(\eta_2)\|_{L^2_{\eta_2}L^\infty_\xi}\|\partial_\tau  h_{\ell_3}\|_{L^\infty}\diff \tau \\
\lesssim &\sum t_{\lb}^{wp_1}\int_0^{t_{\lb}}(1+2^{6j_1}) 2^{(l_{2,1}+l_{2,2})/2}\|2^{j_1/2} \hat h_{\ell_1}\|_{L^\infty_\xi}\|2^{j_1/2}\hat h_{\ell_2}\|_{L^\infty_\xi}\|\partial_\tau  h_{\ell_3}\|_{L^\infty}\diff \tau \\
\lesssim &t_{\lb}^{(w+1)p_1-0.5} 2^{(l_{2,1}+l_{2,2})/2}\|2^{j_1/2} \hat h_{\ell_1}\|_{L_t^\infty L^\infty_\xi }\|2^{j_1/2}\hat h_{\ell_2}\|_{L_t^\infty L^\infty_\xi }\|\vp_{\ell_3}\|_{L_t^\infty B^{1,6}}\cdot\ve_1^2
\end{aligned}
\]
where the summation is for all permutations $(\ell_1, \ell_2, \ell_3)$ of $(j_1, j_2,j_3)$. Notice that the coefficients for $\lb$ is positive, thus it is summable for $\max\{l_{2,1}, l_{2,2}\}\leq 0$.

Combining the estimates for $J_1$ and $J_2$, we obtain \eqref{eqn814}.

\subsection{Space-time resonances}
The cubic nonlinear terms related to the space-time resonances can be formulated as
\begin{itemize}
\item[CASE 1.] When $(\iota_1, \iota_2, \iota_3)=(+, +, +)$ or $(-, -, -)$, $\max\{l_{1,1}, l_{1,2}\}\leq \log_2\varrho(t)$,
\begin{align*}
&\iint_{\R^2} \cutoff_{\lb}^{1} \Tb_1(\eta_1, \eta_2, \xi - \eta_1 - \eta_2) e^{it\Phi(\eta_1,\eta_2, \xi)} \hat h_{j_1}^{\iota_1}(\eta_1)\hat h_{j_2}^{\iota_2}(\eta_2) \hat h_{j_3}(\xi-\eta_1-\eta_2)\upsilon_{\iota_3}(\xi)\psi_{j-3}(\eta_1+\eta_2-2\xi) \diff{\eta_1} \diff{\eta_2},
 \end{align*}
with $\cutoff_{\lb}^{1}=\psi_{l_{1,1}}(\eta_1-\xi)\psi_{l_{1,2}}(\eta_2-\xi)$. Here the cut-off function $\psi_{j-3}(\eta_1+\eta_2-2\xi)$ can be removed when $t$ is large.
\item[CASE 2.] When $(\iota_1, \iota_2, \iota_3)=(+, -, +)$ or $(-, +, -)$, $\max\{l_{3,1}, l_{3,2}\}\leq \log_2\varrho(t)$,
\begin{align*}
\iint_{\R^2} \cutoff_{\lb}^{3} \Tb_1(\eta_1, \eta_2, \xi - \eta_1 - \eta_2) e^{it\Phi(\eta_1,\eta_2, \xi)} \hat h_{j_1}^{\iota_1}(\eta_1)\hat h_{j_2}^{\iota_2}(\eta_2) \hat h_{j_3}(\xi-\eta_1-\eta_2)\upsilon_{\iota_3}(\xi)  \diff{\eta_1} \diff{\eta_2},
 \end{align*}
 with $\cutoff_{\lb}^{3}=\psi_{l_{3,1}}(\eta_1-\xi)\psi_{l_{3,2}}(\eta_1+\eta_2)$.

\item[CASE 3.] When $(\iota_1, \iota_2, \iota_3)=(-, +, +)$ or $(+, -, -)$, $\max\{l_{4,1}, l_{4,2}\}\leq \log_2\varrho(t)$,
\begin{align*}
\iint_{\R^2} \cutoff_{\lb}^{4} \Tb_1(\eta_1, \eta_2, \xi - \eta_1 - \eta_2) e^{it\Phi(\eta_1,\eta_2, \xi)} \hat h_{j_1}^{\iota_1}(\eta_1)\hat h_{j_2}^{\iota_2}(\eta_2) \hat h_{j_3}(\xi-\eta_1-\eta_2)\upsilon_{\iota_3}(\xi)  \diff{\eta_1} \diff{\eta_2},
 \end{align*}
 with $\cutoff_{\lb}^{4}=\psi_{l_{4,1}}(\eta_2-\xi)\psi_{l_{4,2}}(\eta_1+\eta_2)$.
\end{itemize}
We notice that after interchanging $\eta_1$ and $\eta_2$, CASE 3 will be transformed to CASE 2. After interchanging $\eta_2$ and $\xi-\eta_1-\eta_2$, CASE 2 will be transformed to CASE 1. So we only estimate CASE 1 in the following.

By taking the summation for $\max\{l_{1,1}, l_{1,2}\}\leq \log_2\varrho(t)$, we write the integral as 
\begin{align*}
\iint_{\R^2}  \Tb_1(\eta_1, \eta_2, \xi - \eta_1 - \eta_2) e^{it\Phi(\eta_1,\eta_2, \xi)}  \psi\left(\frac{\eta_1-\xi}{2^{\lfloor\log_2\varrho(t)\rfloor}}\right)\psi\left(\frac{\eta_2-\xi}{2^{\lfloor\log_2\varrho(t)\rfloor}}\right)\hat h_{j_1}(\eta_1)\hat h_{j_2}(\eta_2) \hat h_{j_3}(\xi-\eta_1-\eta_2)\diff{\eta_1} \diff{\eta_2}.
\end{align*}
Denote 
\[
\varrho_1(t)=2^{\lfloor\log_2\varrho(t)\rfloor}.
\]
Then it satisfies $\varrho(t)/2\leq\varrho_1(t)<\varrho(t)$.

Considering the modified scattering and \eqref{defU}, in the following, we will estimate
\begin{align*}
& \frac13 i\xi \iint_{\R^2} \Tb_1(\eta_1, \eta_2, \xi - \eta_1 - \eta_2) e^{it\Phi(\eta_1,\eta_2, \xi)}\psi\left(\frac{\eta_1-\xi}{\varrho_1(t)}\right)\psi\left(\frac{\eta_2-\xi}{\varrho_1(t)}\right) \hat h_{j_1}(\eta_1)\hat h_{j_2}(\eta_2) \hat h_{j_3}(\xi-\eta_1-\eta_2)  \diff{\eta_1} \diff{\eta_2}\\
&- \frac{\pi i \xi }{3A}\Tb_1(\xi, \xi, -\xi) \frac{|\hat{h}(t, \xi)|^2\hat{h}(t, \xi)}{t + 1}.
\end{align*}
We can split it into two parts:
\begin{align}
\label{modscat1}
\begin{split}
& \frac{i \xi}{6} \iint_{\R^2} e^{i t \Phi(\xi, \eta_1, \eta_2)} \psi\bigg(\frac{\eta_1 - \xi}{ \varrho_1(t)}\bigg) \cdot \psi\bigg(\frac{\eta_2 - \xi}{ \varrho_1(t)}\bigg)\\*
&\qquad \cdot \bigg[\Tb_1(\eta_1, \eta_2, \xi - \eta_1 - \eta_2) \hat{h}_{j_1}(\eta_1)\hat{h}_{j_2}(\eta_2) \hat{h}_{j_3}(\xi-\eta_1-\eta_2) - \Tb_1(\xi, \xi, -\xi)|\hat{h}(\xi)|^2\hat{h}(\xi)\bigg] \diff{\eta_1} \diff{\eta_2}\\
\end{split}
\end{align}
and
\begin{align}
\label{modscat2}
\frac{i \xi}{6} \Tb_1(\xi, \xi, - \xi) |\hat{h}(t,\xi)|^2\hat{h}(t,\xi) \bigg[\iint_{A_2} e^{i t \Phi(\xi, \eta_1, \eta_2)} \psi\bigg(\frac{\eta_1 - \xi}{ \varrho_1(t)}\bigg) \cdot \psi\bigg(\frac{\eta_2 - \xi}{ \varrho_1(t)}\bigg) \diff{\eta_1} \diff{\eta_2} - \frac{2 \pi }{A(t + 1)}\bigg].
\end{align}
The estimates for \eqref{modscat1} are achieved by a Taylor expansion, \eqref{Tmu_est2} and \eqref{Tmu_est3}.
\begin{align*}
&\bigg|(|\xi|^{r}+|\xi|^{w})\frac16 i\xi \iint_{\R^2} e^{i t \Phi(\xi, \eta_1, \eta_2)} \psi\bigg(\frac{\eta_1 - \xi}{ \varrho_1(t)}\bigg) \cdot \psi\bigg(\frac{\eta_2 - \xi}{ \varrho_1(t)}\bigg)\\*
&\qquad \cdot \Big[\Tb_1(\eta_1, \eta_2, \xi - \eta_1 - \eta_2) \hat{h}_{j_1}(\eta_1)\hat{h}_{j_2}(\eta_2) \hat{h}_{j_3}(\xi-\eta_1-\eta_2) - \Tb_1(\xi, \xi, -\xi)|\hat{h}(\xi)|^2\hat{h}(\xi) \Big] \diff{\eta_1} \diff{\eta_2}\bigg|\\
\lesssim~& (|\xi|^{r}+|\xi|^{w}) |\xi| \iint_{\R^2} \left|\partial_{\eta_1}\left[\Tb_1(\eta_1, \eta_2, \xi - \eta_1 - \eta_2) \hat{h}_{j_1}(\eta_1)\hat{h}_{j_2}(\eta_2) \hat{h}_{j_3}(\xi-\eta_1-\eta_2)\right] \bigg|_{\eta_1 = \eta_1'} (\xi-\eta_1)\right|\\
& \hspace {.75in}+\left|\partial_{\eta_2}\left[\Tb_1(\eta_1, \eta_2, \xi - \eta_1 - \eta_2) \hat{h}_{j_1}(\eta_1)\hat{h}_{j_2}(\eta_2) \hat{h}_{j_3}(\xi-\eta_1-\eta_2)\right] \bigg|_{\eta_2 = \eta_2'} (\xi-\eta_2)\right| \diff{\eta_1} \diff{\eta_2}\\
\lesssim~&(t+1)^{(w+8)p_1} \Big( \||\xi|^{\frac{1+r}3}\hat\vp_{j_1}\|_{L^\infty_\xi}\||\xi|^{\frac{1+r}3}\hat\vp_{j_2}\|_{L^\infty_\xi}\||\xi|^{\frac{1+r}3}\hat\vp_{j_3}\|_{L^\infty_\xi} [ \varrho_1(t)]^{3}\\
&\qquad+ \sum\||\xi|^{\frac{1+r}3}\hat\vp_{\ell_1}\|_{L^\infty_\xi}\||\xi|^{\frac{1+r}3}\hat\vp_{\ell_2}\|_{L^\infty_\xi}\|\partial_\xi\hat h_{\ell_3}\|_{L^2_\xi}[ \varrho_1(t)]^{5/2}\Big).
\end{align*}

As for \eqref{modscat2}, it suffices to estimate
\begin{align*}
& \bigg|(|\xi|^{r}+|\xi|^{w}) \frac{i \xi}{6} \Tb_1(\xi, \xi, - \xi) |\hat{h}(t,\xi)|^2\hat{h}(t,\xi) \bigg[\iint_{\R^2} e^{i t \Phi(\xi, \eta_1, \eta_2)} \psi\bigg(\frac{\eta_1 - \xi}{ \varrho_1(t)}\bigg) \cdot \psi\bigg(\frac{\eta_2 - \xi}{ \varrho_1(t)}\bigg) \diff{\eta_1} \diff{\eta_2} - \frac{2 \pi }{\mathfrak{A}t}\bigg]\bigg|\\
\lesssim ~& \|(|\xi|^{r}+|\xi|^{w}) \hat{\vp}(\xi)\|_{L^\infty_\xi} \||\xi|^{1/2} \hat{\vp}(\xi)\|_{L^\infty_\xi} \|(|\xi|^{1/2}+|\xi|^w)\hat{\vp}(\xi)\|_{L^\infty_\xi}\\
& \qquad \cdot \bigg\|\iint_{\R^2} e^{i t \Phi(\xi, \eta_1, \eta_2)} \psi\bigg(\frac{\eta_1 - \xi}{ \varrho_1(t)}\bigg) \cdot \psi\bigg(\frac{\eta_2 - \xi}{ \varrho_1(t)}\bigg) \diff{\eta_1} \diff{\eta_2} - \frac{2 \pi }{\mathfrak{A}t}\bigg\|_{L^\infty_\xi}\\
\lesssim ~& \|\vp\|_{Z}^3 \bigg\|\iint_{\R^2} e^{i t \Phi(\xi, \eta_1, \eta_2)} \psi\bigg(\frac{\eta_1 - \xi}{ \varrho_1(t)}\bigg) \cdot \psi\bigg(\frac{\eta_2 - \xi}{ \varrho_1(t)}\bigg) \diff{\eta_1} \diff{\eta_2} - \frac{2 \pi }{\mathfrak{A}t}\bigg\|_{L^\infty_\xi}.
\end{align*}
Writing $(\eta_1, \eta_2) = (\xi + \zeta_1, \xi + \zeta_2)$, we find from \eqref{defPhi} that
\begin{align*}
\Phi(\xi, \eta_1, \eta_2)
& = \frac{3\xi}{(1+\xi^2)^{5/2}}\zeta_1 \zeta_2 + O\bigg(\zeta_1^3 + \zeta_2^3\bigg)
=\frac{3\xi}{(1+\xi^2)^{5/2}}\zeta_1 \zeta_2  + O\left([ \varrho_1(t)]^3 \right).
\end{align*}
Since $p_0 = 10^{-4}$ and $ \varrho_1(t)$ satisfies \eqref{varrho}, the error term is integrable in time, so
we now only need to estimate
\begin{align}
\label{modscat}
\begin{split}
J_3 &:= \bigg\|\iint_{\R^2} e^{- i t \mathfrak{A} \zeta_1 \zeta_2} \psi\bigg(\frac{\zeta_1}{ \varrho_1(t)}\bigg) \cdot \psi\bigg(\frac{\zeta_2}{ \varrho_1(t)}\bigg) \diff{\zeta_1} \diff{\zeta_2} - \frac{2 \pi }{\mathfrak{A}t}\bigg\|_{L^\infty_\xi}.
\end{split}
\end{align}
Making the change of variables
\[
\zeta_1 = \sqrt{\frac{1}{\mathfrak{A}t}} x_1,\qquad \zeta_2 = \sqrt{\frac{1}{\mathfrak{A}t}} x_2,
\]
in \eqref{modscat} and using the fact that $|\xi| \le (t + 1)^{p_1}$, we find that
\begin{align}
\label{modscatest}
\begin{split}
J_3 &\le \frac{(t + 1)^{p_1}}{t} \bigg\|\iint_{\R^2} e^{- i x_1 x_2} \psi\bigg(\frac{1}{\sqrt{\mathfrak{A}t}  \varrho_1(t)} x_1\bigg) \cdot \psi\bigg(\frac{1}{\sqrt{\mathfrak{A}t}  \varrho_1(t)} x_2\bigg) \diff{x_1} \diff{x_2} - 2 \pi\bigg\|_{L^\infty_\xi}.
\end{split}
\end{align}
The integral identity
\[
\int_\R e^{- a x^2 - b x} \diff{x} =\sqrt{ \frac{\pi}{a}} e^{\frac{b^2}{4 a}}, \qquad \text{for all $a, b \in \mathbb{C}$ with $\Re{a} > 0$}
\]
gives that
\[
\iint_{\R^2} e^{- i x_1 x_2} e^{- \frac{x_1^2}{\mathfrak{B}^2}} e^{- \frac{x_2^2}{\mathfrak{B}^2}} \diff{x_1} \diff{x_2} = \sqrt{\pi} \mathfrak{B} \int_\R e^{- \frac{x_2^2}{\mathfrak{B}^2}} e^{- \frac{\mathfrak{B}^2 x_2^2}{4}} \diff{x_2} = 2 \pi + O(\mathfrak{B}^{-1}),
\qquad \text{as $\mathfrak{B}\to \infty$},
\]
and therefore
\begin{align}
\label{2piint}
\left|\iint_{\R^2} e^{- i x_1 x_2} \psi\bigg(\frac{x_1}{\mathfrak{B}}\bigg) \psi\bigg(\frac{x_2}{\mathfrak{B}}\bigg) \diff{x_1} \diff{x_2} -2 \pi\right|\lesssim \mathfrak{B}^{- 1 / 2},
\qquad \text{as $\mathfrak{B}\to \infty$}.
\end{align}
Thus

Using \eqref{2piint} with $\mathfrak{B} = \sqrt{\mathfrak{A}t}  \varrho_1(t)  = O(t^{0.05})$ in \eqref{modscatest} then yields
\begin{align*}
J_3 \lesssim (t + 1)^{ p_1-1.025}.
\end{align*}
Since $p_1 = 10^{-5}$, the right-hand side decays faster in time than $1 / t$, which implies that \eqref{modscat2} is integrable in time and bounded by a constant multiple of $\ve_0^3$. 

Putting all the above estimates together, we conclude that
\[
\int_0^{\infty} \|(|\xi|^{r}+|\xi|^{w}) U(t, \xi)\|_{L^\infty_\xi} \diff t\lesssim \ve_0.
\]

\subsection{Higher-degree terms}\label{higherorder}
In this subsection, we prove that
\[
\left\|\cutoffxi(t, \xi)(|\xi|^{r+1}+|\xi|^{w+1})\widehat{\mathcal{N}_{\geq5}(\vp)}\right\|_{L^\infty_\xi}
\]
is integrable in time. 

By Lemma \ref{multilinear} and the symbol estimate Proposition \ref{Prop_symb_Tmu}, we have
\begin{align*}
\left\|\cutoffxi(t, \xi)(|\xi|^{r+1}+|\xi|^{w+1})\widehat{\mathcal{N}_{\geq5}(\vp)}\right\|_{L^\infty_\xi}\lesssim& ~(t+1)^{(w+1)p_1}\|\mathcal{N}_{\geq 5}(\vp)\|_{L^1} \lesssim (t+1)^{(w+1)p_1} \|\vp\|_{L^2}^2\sum\limits_{n=2}^\infty \left(\|\vp\|_{B^{1,6}}^{2n-1}\right).
\end{align*}
Using the dispersive estimate Lemma~\ref{NonDis}, we see that the right-hand-side is integrable in $t$, which leads to
\[
\int_0^\infty \left\| \cutoffxi(t, \xi)(|\xi|^{r+1}+|\xi|^{w+1})\widehat{\mathcal N_{\geq5}(\vp)} \right\|_{L_\xi^\infty}\diff t\lesssim \ve_0.
\]
This completes the proof of Theorem~\ref{bootstrap}.

\appendix

\section{Estimate of the symbols}
In this appendix, we estimate the symbols of the multilinear Fourier integral operators used in this paper. 
\begin{proposition}\label{Prop_symb_Tmu}
Assume $\mu\geq 1$,  $j_1, j_2,\cdots, j_{2\mu+1}\in\Z$, and $j_1\geq\cdots\geq j_{2\mu+1}$.
\begin{multline}\label{Tmu_est1}
 \|\Tb_{\mu}(\eta_1,\eta_2,\cdots, \eta_{2\mu+1})\psi_{j_1}(\eta_1)\psi_{j_2}(\eta_2)\cdots\psi_{j_{2\mu+1}}(\eta_{2\mu+1})\|_{S^{\infty}}\\
\lesssim 2^{(j_1+j_2+\cdots+j_{2\mu+1})}(1+2^{j_1})^{-1}(1+2^{2 j_2})\prod\limits_{k=2}^{2\mu+1} (1+2^{j_k}).
\end{multline}
\begin{multline}\label{Tmu_est2}
\|\partial_{\eta_1}\Tb_{\mu}(\eta_1,\eta_2,\cdots, \eta_{2\mu+1})\psi_{j_1}(\eta_1)\psi_{j_2}(\eta_2)\cdots\psi_{j_{2\mu+1}}(\eta_{2\mu+1})\|_{S^{\infty}}\\
\lesssim  (1+2^{2j_1})^{-1}(1+2^{j_2})^{4}2^{(j_2+\cdots+j_{2\mu+1})}\prod\limits_{k=2}^{2\mu+1} (1+2^{j_k}).
\end{multline}
For $l=2,\cdots, 2\mu+1$,
\begin{multline}\label{Tmu_est3}
\|\partial_{\eta_{l}}\Tb_{\mu}(\eta_1,\eta_2,\cdots, \eta_{2\mu+1})\psi_{j_1}(\eta_1)\psi_{j_2}(\eta_2)\cdots\psi_{j_{2\mu+1}}(\eta_{2\mu+1})\|_{S^{\infty}}\\
\lesssim  \prod\limits_{k\neq 1,l}[2^{j_k}(1+2^{j_k})](1+2^{j_l})^2(1+2^{j_2})^22^{j_1}(1+2^{j_1})^{-1}.
\end{multline}
Specially, when $j_1<0$, \eqref{Tmu_est1} implies
\begin{align}\label{Tmu_est4}
\|\Tb_{\mu}(\eta_1,\eta_2,\cdots, \eta_{2\mu+1})\psi_{j_1}(\eta_1)\psi_{j_2}(\eta_2)\cdots\psi_{j_{2\mu+1}}(\eta_{2\mu+1})\|_{S^{\infty}}
\lesssim 2^{(j_1+\cdots+j_{2\mu+1})}.
\end{align}
When $j_1-1>j_2\geq \cdots \geq j_{2\mu+1}$, and $j_1>0$,
\begin{multline}\label{Tmu_est5}
\|[\Tb_{\mu}(\eta_1,\eta_2,\cdots, \eta_{2\mu+1})-\Tb_{\mu}(\eta_1+\cdots+\eta_{2\mu+1},\eta_2,\cdots, \eta_{2\mu+1})]\psi_{j_1}(\eta_1)\psi_{j_2}(\eta_2)\cdots\psi_{j_{2\mu+1}}(\eta_{2\mu+1})\|_{S^{\infty}}\\
\lesssim 2^{-j_1}(1+2^{j_2})^3\prod\limits_{k=2}^{2\mu+1}2^{j_k}(1+2^{j_k}).
\end{multline}
\end{proposition}

\begin{proof}
First, we notice that
\begin{align*}
&\|\Tb_{\mu}(\eta_1,\eta_2,\cdot, \eta_{2\mu+1})\psi_{j_1}(\eta_1)\cdots\psi_{j_{2\mu+1}}(\eta_{2\mu+1})\|_{S^{\infty}}\\
=&\iint_{\R^{2\mu+1}}\left|\iint_{\R^{2\mu+1}}e^{-i(\sum_{i=1}^{2\mu+1}x_i\eta_i)}\Tb(\eta_1,\cdots,\eta_{2\mu+1})\psi_{j_1}(\eta_1)\cdots\psi_{j_{2\mu+1}}(\eta_{2\mu+1})\diff \eta_1\cdots\diff \eta_{2\mu+1}\right|\diff x_1\cdots\diff x_{2\mu+1}\\
=&\iint_{\R^{2\mu+1}}\left|\iint_{\R^{2\mu+1}}e^{-i(\sum_{i=1}^{2\mu+1}y_i\xi_i)}\Tb(2^{j_1}\xi_1,\cdots,2^{j_{2\mu+1}}\xi_{2\mu+1})\psi_{0}(\xi_1)\cdots \psi_{0}(\xi_{2\mu+1})\diff \xi_1\cdots\diff \xi_{2\mu+1}\right|\diff y_1\cdots\diff y_{2\mu+1}\\
\lesssim&\left|\iint_{\R^3}e^{-i(\sum_{i=1}^{2\mu+1}y_i\xi_i)}(1-\partial_{\xi_1}^2)\cdots(1-\partial_{\xi_{2\mu+1}}^2)[\Tb(2^{j_1}\xi_1,\cdots,2^{j_{2\mu+1}}\xi_{2\mu+1})\psi_{0}(\xi_1)\cdots\psi_{0}(\xi_{2\mu+1})]\diff \xi_1\cdots\diff \xi_{2\mu+1}\right|\\
\lesssim & \|(1-\partial_{\xi_1}^2)\cdots(1-\partial_{\xi_{2\mu+1}}^2)[\Tb(2^{j_1}\xi_1,\cdots,2^{j_{2\mu+1}}\xi_{2\mu+1})\psi_{0}(\xi_1)\cdots\psi_{0}(\xi_{2\mu+1})]\|_{L^\infty}
\end{align*}
Here $\Tb$ is $\Tb_{\mu}$ or $\partial_{\eta_1}\Tb_\mu$.
So, we mainly need to estimate $|(1-\partial_{\xi_1}^2)\cdots(1-\partial_{\xi_{2\mu+1}}^2)\Tb(2^{j_1}\xi_1,\cdots,2^{j_{2\mu+1}}\xi_{2\mu+1})|$ under the assumption that $|\xi_1|, \cdots, |\xi_{2\mu+1}|$ are in $(\frac58, \frac85)$. In the following proof, we will use this bound for $|\xi_1|,\cdots, |\xi_{2\mu+1}|$ repeatedly. 

{\bf Proof of \eqref{Tmu_est1}. }Since the above integral is symmetric with respect to the sub-index, we assume $j_1$ is the maximum of $\{j_1,\cdots, j_{2\mu+1}\}$.  
By \eqref{Tmu}, we obtain
\begin{align*}
&(1-\partial_{\xi_1}^2)\cdots(1-\partial_{\xi_{2\mu+1}}^2)\Tb_\mu(2^{j_1}\xi_1,\cdots,2^{j_{2\mu+1}}\xi_{2\mu+1})\\
=
&2\int_{\R}\chi_{\{|\zeta|>1\}}[(1-e^{i2^{j_1}\xi_1\zeta})+2^{2j_1}\zeta^2e^{i2^{j_1}\xi_1\zeta}]\cdots[(1-e^{i2^{j_{2\mu+1}}\xi_{2\mu+1}\zeta})+2^{2j_{2\mu+1}}\zeta^2e^{i2^{j_{2\mu+1}}\xi_{2\mu+1}\zeta}]\\
&\qquad\qquad\qquad\cdot A_{\mu}(\zeta)\diff \zeta\\
&+2\int_{\R}\chi_{\{|\zeta|<1\}} [(1-e^{i2^{j_1}\xi_1\zeta})+2^{2j_1}\zeta^2e^{i2^{j_1}\xi_1\zeta}]\cdots[(1-e^{i2^{j_{2\mu+1}}\xi_{2\mu+1}\zeta})+2^{2j_{2\mu+1}}\zeta^2e^{i2^{j_{2\mu+1}}\xi_{2\mu+1}\zeta}]\\
&\qquad\qquad\qquad\cdot  B_{\mu}(\zeta)\diff \zeta.
\end{align*}
Since $|A_\mu(\zeta)|\lesssim e^{-|\zeta|}$ on $|\zeta|>1$, $|\zeta|^{2\mu}B_\mu(\zeta)$ is bounded on $|\zeta|<1$, and $\left|\frac{1-e^{i2^{j}\xi\zeta}}{\zeta}\right|\leq 2^{j}|\xi|$, we obtain
\begin{align}\label{j1small}
\left|(1-\partial_{\xi_1}^2)\cdots(1-\partial_{\xi_{2\mu+1}}^2)\Tb_\mu(2^{j_1}\xi_1,\cdots,2^{j_{2\mu+1}}\xi_{2\mu+1})\right|\lesssim 2^{(j_1+\cdots+j_{2\mu+1})}\prod\limits_{k=1}^{2\mu+1} (1+2^{j_k})
\end{align}
and
\begin{align*}
\left|(1-\partial_{\xi_2}^2)\cdots(1-\partial_{\xi_{2\mu+1}}^2)\Tb_\mu(2^{j_1}\xi_1,\cdots,2^{j_{2\mu+1}}\xi_{2\mu+1})\right|\lesssim 2^{(j_2+\cdots+j_{2\mu+1})}\prod\limits_{k=2}^{2\mu+1} (1+2^{j_k}).
\end{align*}
When $j_1<0$, \eqref{Tmu_est1} and \eqref{Tmu_est4} follow directly from \eqref{j1small}. In the following, we assume $j_1>0$.

For the term with $\partial_{\xi_1}^2$,
\begin{align*}
&\partial_{\xi_1}^2(1-\partial_{\xi_2}^2)\cdots(1-\partial_{\xi_{2\mu+1}}^2)\Tb_\mu(2^{j_1}\xi_1,\cdots,2^{j_{2\mu+1}}\xi_{2\mu+1})\\
=
&2\int_{\R}\chi_{\{|\zeta|>1\}}[(1-e^{i2^{j_2}\xi_2\zeta})+2^{2j_2}\zeta^2e^{i2^{j_2}\xi_2\zeta}]\cdots[(1-e^{i2^{j_{2\mu+1}}\xi_{2\mu+1}\zeta})+2^{2j_{2\mu+1}}\zeta^2e^{i2^{j_{2\mu+1}}\xi_{2\mu+1}\zeta}]\\
&\qquad\qquad\qquad\cdot (\zeta^22^{2j_1}e^{i2^{j_1}\xi_1\zeta}) A_{\mu}(\zeta)\diff \zeta\\
&+2\int_{\R}\chi_{{\{|\zeta|<1\}}}[(1-e^{i2^{j_2}\xi_2\zeta})+2^{2j_2}\zeta^2e^{i2^{j_2}\xi_2\zeta}]\cdots[(1-e^{i2^{j_{2\mu+1}}\xi_{2\mu+1}\zeta})+2^{2j_{2\mu+1}}\zeta^2e^{i2^{j_{2\mu+1}}\xi_{2\mu+1}\zeta}]\\
&\qquad\qquad\qquad\cdot (\zeta^22^{2j_1}e^{i2^{j_1}\xi_1\zeta}) B_{\mu}(\zeta)\diff \zeta.
\end{align*}
Then, by integral by part, the first integral is bounded by
\begin{align*}
&\Bigg|2\int_{\R}e^{i2^{j_1}\xi_1\zeta}\partial_{\zeta}^2\left\{\chi_{\{|\zeta|>1\}}[(1-e^{i2^{j_2}\xi_2\zeta})+2^{2j_2}\zeta^2e^{i2^{j_2}\xi_2\zeta}]\cdots\right.\\
&\left.\qquad\qquad\qquad\cdot[(1-e^{i2^{j_{2\mu+1}}\xi_{2\mu+1}\zeta})+2^{2j_{2\mu+1}}\zeta^2e^{i2^{j_{2\mu+1}}\xi_{2\mu+1}\zeta}] \zeta^2A_{\mu}(\zeta)\right\}\diff \zeta\Bigg|\\
\lesssim~ &2^{(j_2+\cdots+j_{2\mu+1})}\prod\limits_{k=2}^{2\mu+1} (1+2^{j_k})2^{2\max\{0,j_2,\cdots, j_{2\mu+1}\}}.
\end{align*}
The second integral is bounded by
\begin{align*}
&\Bigg|2\int_{\R}e^{i2^{j_1}\xi_1\zeta}\partial_{\zeta}^2\left\{\chi_{\{|\zeta|<1\}}[(1-e^{i2^{j_2}\xi_2\zeta})+2^{2j_2}\zeta^2e^{i2^{j_2}\xi_2\zeta}]\cdots\right.\\
&\qquad\qquad\qquad\left.\cdot [(1-e^{i2^{j_{2\mu+1}}\xi_{2\mu+1}\zeta})+2^{2j_{2\mu+1}}\zeta^2e^{i2^{j_{2\mu+1}}\xi_{2\mu+1}\zeta}] \zeta^2 B_{\mu}(\zeta)\right\}\diff \zeta\Bigg|\\
\lesssim~ &2^{(j_2+\cdots+j_{2\mu+1})}\prod\limits_{k=2}^{2\mu+1} (1+2^{j_k})2^{2\max\{0, j_2,\cdots, j_{2\mu+1}\}}.
\end{align*}
Thus, we proved \eqref{Tmu_est1}. 

{\bf Proof of \eqref{Tmu_est2}. }To prove \eqref{Tmu_est2}, we estimate $|(1-\partial_{\xi_2}^2)\cdots(1-\partial_{\xi_{2\mu+1}}^2)[\partial_1\Tb_\mu](2^{j_1}\xi_1,\cdots,2^{j_{2\mu+1}}\xi_{2\mu+1})|$ and $|\partial_{\xi_1^2}(1-\partial_{\xi_2}^2)\cdots(1-\partial_{\xi_{2\mu+1}}^2)[\partial_1\Tb_\mu](2^{j_1}\xi_1,\cdots,2^{j_{2\mu+1}}\xi_{2\mu+1})|$ separately. Here the notation $\partial_1$ is the partial derivative to the first variable of the function $\Tb_\mu(\eta_1, \cdots, \eta_{2\mu+1})$.
\begin{align*}
&(1-\partial_{\xi_1}^2)(1-\partial_{\xi_2}^2)\cdots(1-\partial_{\xi_{2\mu+1}}^2)[\partial_1\Tb_\mu](2^{j_1}\xi_1,\cdots,2^{j_{2\mu+1}}\xi_{2\mu+1})\\
=
&2\int_{\R}\chi_{\{|\zeta|>1\}}[(1-e^{i2^{j_2}\xi_2\zeta})+2^{2j_2}\zeta^2e^{i2^{j_2}\xi_2\zeta}]\cdots[(1-e^{i2^{j_{2\mu+1}}\xi_{2\mu+1}\zeta})+2^{2j_{2\mu+1}}\zeta^2e^{i2^{j_{2\mu+1}}\xi_{2\mu+1}\zeta}]\\
&\qquad\qquad\qquad\cdot (1+2^{2j_1}\zeta^2)(-i\zeta e^{i2^{j_1}\xi_1\zeta}) A_{\mu}(\zeta)\diff \zeta\\
&+2\int_{\R}\chi_{\{|\zeta|<1\}} [(1-e^{i2^{j_2}\xi_2\zeta})+2^{2j_2}\zeta^2e^{i2^{j_2}\xi_2\zeta}]\cdots[(1-e^{i2^{j_{2\mu+1}}\xi_{2\mu+1}\zeta})+2^{2j_{2\mu+1}}\zeta^2e^{i2^{j_{2\mu+1}}\xi_{2\mu+1}\zeta}]\\
&\qquad\qquad\qquad\cdot (1+2^{2j_1}\zeta^2)(-i\zeta e^{i2^{j_1}\xi_1\zeta}) B_{\mu}(\zeta)\diff \zeta.
\end{align*}
For $j_1<0$, 
\begin{align*}
\left|(1-\partial_{\xi_1}^2)(1-\partial_{\xi_2}^2)\cdots(1-\partial_{\xi_{2\mu+1}}^2)[\partial_1\Tb_\mu](2^{j_1}\xi_1,\cdots,2^{j_{2\mu+1}}\xi_{2\mu+1})\right|
\lesssim 2^{(j_2+\cdots+j_{2\mu+1})}\prod\limits_{k=2}^{2\mu+1} (1+2^{j_k}),
\end{align*}
and \eqref{Tmu_est2} is achieved. In the following, we assume $j_1\geq0$.

Notice that
\begin{align*}
&(1-\partial_{\xi_2}^2)\cdots(1-\partial_{\xi_{2\mu+1}}^2)[\partial_1\Tb_\mu](2^{j_1}\xi_1,\cdots,2^{j_{2\mu+1}}\xi_{2\mu+1})\\
=
&2\int_{\R}\chi_{\{|\zeta|>1\}}[(1-e^{i2^{j_2}\xi_2\zeta})+2^{2j_2}\zeta^2e^{i2^{j_2}\xi_2\zeta}]\cdots[(1-e^{i2^{j_{2\mu+1}}\xi_{2\mu+1}\zeta})+2^{2j_{2\mu+1}}\zeta^2e^{i2^{j_{2\mu+1}}\xi_{2\mu+1}\zeta}]
\\
&\qquad\qquad\qquad\cdot (-i\zeta e^{i2^{j_1}\xi_1\zeta}) A_{\mu}(\zeta)\diff \zeta
\\
&+2\int_{\R}\chi_{\{|\zeta|<1\}} [(1-e^{i2^{j_2}\xi_2\zeta})+2^{2j_2}\zeta^2e^{i2^{j_2}\xi_2\zeta}]\cdots[(1-e^{i2^{j_{2\mu+1}}\xi_{2\mu+1}\zeta})+2^{2j_{2\mu+1}}\zeta^2e^{i2^{j_{2\mu+1}}\xi_{2\mu+1}\zeta}]\\
&\qquad\qquad\qquad\cdot (-i\zeta e^{i2^{j_1}\xi_1\zeta}) B_{\mu}(\zeta)\diff \zeta.
\end{align*}
The first integral can be written as 
\begin{align*}
&2\int_{\R}\chi_{\{|\zeta|>1\}}[(1-e^{i2^{j_2}\xi_2\zeta})+2^{2j_2}\zeta^2e^{i2^{j_2}\xi_2\zeta}]\cdots[(1-e^{i2^{j_{2\mu+1}}\xi_{2\mu+1}\zeta})+2^{2j_{2\mu+1}}\zeta^2e^{i2^{j_{2\mu+1}}\xi_{2\mu+1}\zeta}]
\\
&\qquad\qquad\qquad\cdot (-i\zeta e^{i2^{j_1}\xi_1\zeta}) A_{\mu}(\zeta)\diff \zeta
\\
=&-2i\int_{\R}\chi_{\{|\zeta|>1\}}[(1-e^{i2^{j_2}\xi_2\zeta})+2^{2j_2}\zeta^2e^{i2^{j_2}\xi_2\zeta}]\cdots[(1-e^{i2^{j_{2\mu+1}}\xi_{2\mu+1}\zeta})+2^{2j_{2\mu+1}}\zeta^2e^{i2^{j_{2\mu+1}}\xi_{2\mu+1}\zeta}]\\
&\qquad\qquad\qquad\cdot (\frac1{i2^{j_1}\xi_1}\partial_{\zeta})^2 e^{i2^{j_1}\xi_1\zeta} \zeta  A_{\mu}(\zeta)\diff \zeta.
\end{align*}
After integrating by part twice, we obtain
\begin{align*}
&\Bigg|2\int_{\R}\chi_{\{|\zeta|>1\}}[(1-e^{i2^{j_2}\xi_2\zeta})+2^{2j_2}\zeta^2e^{i2^{j_2}\xi_2\zeta}]\cdots[(1-e^{i2^{j_{2\mu+1}}\xi_{2\mu+1}\zeta})+2^{2j_{2\mu+1}}\zeta^2e^{i2^{j_{2\mu+1}}\xi_{2\mu+1}\zeta}]\\
&\qquad\qquad\qquad\cdot (-i\zeta e^{i2^{j_1}\xi_1\zeta})A_{\mu}(\zeta)\diff \zeta\Bigg|\\
\lesssim& 2^{(j_2+\cdots+j_{2\mu+1})}\prod\limits_{k=2}^{2\mu+1} (1+2^{j_k})2^{2\max\{0, j_2, \cdots, j_{2\mu+1}\}-2j_1}.
\end{align*}
For the second integral
\begin{align*}
&2\int_{\R}\chi_{\{|\zeta|<1\}} [(1-e^{i2^{j_2}\xi_2\zeta})+2^{2j_2}\zeta^2e^{i2^{j_2}\xi_2\zeta}]\cdots[(1-e^{i2^{j_{2\mu+1}}\xi_{2\mu+1}\zeta})+2^{2j_{2\mu+1}}\zeta^2e^{i2^{j_{2\mu+1}}\xi_{2\mu+1}\zeta}]
\\
&\qquad\qquad\qquad\cdot (-i\zeta e^{i2^{j_1}\xi_1\zeta}) B_{\mu}(\zeta)\diff \zeta
\\
=&-2i\int_{\R}\chi_{\{|\zeta|<1\}} [(1-e^{i2^{j_2}\xi_2\zeta})+2^{2j_2}\zeta^2e^{i2^{j_2}\xi_2\zeta}]\cdots[(1-e^{i2^{j_{2\mu+1}}\xi_{2\mu+1}\zeta})+2^{2j_{2\mu+1}}\zeta^2e^{i2^{j_{2\mu+1}}\xi_{2\mu+1}\zeta}]\\
&\qquad\qquad\qquad\cdot \left((\frac1{i2^{j_1}\xi_1}\partial_{\zeta})^2 e^{i2^{j_1}\xi_1\zeta}\right) \zeta B_{\mu}(\zeta)\diff \zeta.
\end{align*}
After integrating by part twice, we need to estimate
\begin{align*}
&\frac1{2^{2j_1}\xi_1^2}(-2i)\int_{\R} e^{i2^{j_1}\xi_1\zeta} \partial_{\zeta}^2\left\{ \chi_{\{|\zeta|<1\}} [(1-e^{i2^{j_2}\xi_2\zeta})+2^{2j_2}\zeta^2e^{i2^{j_2}\xi_2\zeta}]\cdots\right.\\
&\qquad\qquad\qquad\left.\cdot[(1-e^{i2^{j_{2\mu+1}}\xi_{2\mu+1}\zeta})+2^{2j_{2\mu+1}}\zeta^2e^{i2^{j_{2\mu+1}}\xi_{2\mu+1}\zeta}]  \zeta B_{\mu}(\zeta)\right\}\diff \zeta.
\end{align*} 
Recall that after ignoring a constant multiple, the most singular term in $B_\mu(\zeta)$ is $I_0(\zeta)|\zeta|^{-2\mu}$. This term leads to
\begin{align*}
\int_{\R}\chi_{\{|\zeta|<1\}} [(1-e^{i2^{j_2}\xi_2\zeta})+2^{2j_2}\zeta^2e^{i2^{j_2}\xi_2\zeta}]\cdots[(1-e^{i2^{j_{2\mu+1}}\xi_{2\mu+1}\zeta})+2^{2j_{2\mu+1}}\zeta^2e^{i2^{j_{2\mu+1}}\xi_{2\mu+1}\zeta}]\\
\cdot e^{i2^{j_1}\xi_1\zeta} I_0(\zeta)|\zeta|^{-2\mu-1}\diff \zeta.
\end{align*}
By Taylor expansion of $I_0(\zeta)$ around $0$, it suffices to estimate  
\begin{align*}
I_0(0)\int_{\R}\chi_{\{|\zeta|<1\}} [(1-e^{i2^{j_2}\xi_2\zeta})+2^{2j_2}\zeta^2e^{i2^{j_2}\xi_2\zeta}]\cdots[(1-e^{i2^{j_{2\mu+1}}\xi_{2\mu+1}\zeta})+2^{2j_{2\mu+1}}\zeta^2e^{i2^{j_{2\mu+1}}\xi_{2\mu+1}\zeta}]\\
\cdot e^{i2^{j_1}\xi_1\zeta} |\zeta|^{-2\mu-1}\diff \zeta.
\end{align*}
We claim that the above integral is finite. In fact, we can rewrite the integral as
\begin{align*}
&\int_{\R}\chi_{\{|\zeta|<1\}} \frac{(1-e^{i2^{j_2}\xi_2\zeta})+2^{2j_2}\zeta^2e^{i2^{j_2}\xi_2\zeta}}{\zeta}\cdots\frac{(1-e^{i2^{j_{2\mu+1}}\xi_{2\mu+1}\zeta})+2^{2j_{2\mu+1}}\zeta^2e^{i2^{j_{2\mu+1}}\xi_{2\mu+1}\zeta}}{\zeta}\frac{e^{i2^{j_1}\xi_1\zeta}}{\zeta} \diff \zeta\\
=&\int_{\R}\chi_{\{|\zeta|<1\}}\prod\limits_{k=2}^{2\mu+1}\left(\frac{(1-e^{i2^{j_k}\xi_k\zeta})+2^{2j_k}\zeta^2e^{i2^{j_k}\xi_k\zeta}+i2^{j_k}\xi_k\zeta}{\zeta^2}\zeta-i2^{j_k}\xi_k\right)\frac{e^{i2^{j_1}\xi_1\zeta}}{\zeta} \diff \zeta.
\end{align*}
By Taylor expansion, 
\[
\lim\limits_{\zeta\to0}\frac{(1-e^{i2^{j_k}\xi_k\zeta})+2^{2j_k}\zeta^2e^{i2^{j_k}\xi_k\zeta}+i2^{j_k}\xi_k\zeta}{\zeta^2}=2^{2j_k}(\xi_k^2+1), \quad k=2,\cdots, 2\mu+1.
\]
Therefore, after expanding the product, the terms with $\frac{(1-e^{i2^{j_k}\xi_k\zeta})+2^{2j_k}\zeta^2e^{i2^{j_k}\xi_k\zeta}+i2^{j_k}\xi_k\zeta}{\zeta^2}\zeta$ are integrable and bounded by a constant multiple of $2^{2(j_2+\cdots+j_{2\mu+1})}$. The last term is
\[
(-1)^\mu \prod\limits_{k=2}^{2\mu+1}\left(2^{j_k}\xi_k\right) I_0(0)\int_{\R} \chi_{\{|\zeta|<1\}}\frac{e^{i2^{j_1}\xi_1\zeta}}{\zeta} \diff \zeta,
\]
in which the integral is
\begin{align*}
& \int_{\R}\chi_{\{|\zeta|<1\}} \frac{e^{i2^{j_1}\xi_1\zeta}}{\zeta} \diff \zeta\\
=& \int_{\R}\chi_{\{|\zeta|<1\}} \frac{\sin(2^{j_1}\xi_1\zeta)}{\zeta} \diff \zeta\\
=& \int_{\R}\chi_{\{|\zeta|<2^{j_1}|\xi_1|\}} \frac{\sin(\zeta)}{\zeta} \diff \zeta.
\end{align*}
If $2^{j_1}|\xi_1|<1$, 
\[
\left|\int_{\R}\chi_{\{|\zeta|<2^{j_1}|\xi_1|\}} \frac{\sin(\zeta)}{\zeta} \diff \zeta \right|\leq \int_{|\zeta|<1}\frac{|\sin(\zeta)|}{|\zeta|} \diff \zeta.
\]
The function $\frac{|\sin(\zeta)|}{|\zeta|}$ is integrable, so the right-hand-side is a constant.
The integral $\int_{\R}\chi_{\{|\zeta|<2^{j_1}|\xi_1|\}} \frac{\sin(\zeta)}{\zeta} \diff \zeta$ is bounded by a constant independent of $j_1$. 

If $2^{j_1}|\xi_1|>1$, we split the integral into two parts: $|\zeta|<1$ and $1<|\zeta|<2^{j_1}|\xi_1|$. The estimate of the integral of $|\zeta|<1$ is similar as above. For $1<|\zeta|<2^{j_1}|\xi_1|$, we can estimate the integral as following. For a fixed $j_1$ and $\xi_1$, we have an integer $m$, such that $2m\pi< 2^{j_1}|\xi_1|\leq 2(m+1)\pi$. Then the integral can be written as
\begin{align*}
\int_{1<|\zeta|<2^{j_1}|\xi_1|} \frac{\sin \zeta}{\zeta}\diff \zeta=&2\int_{1<\zeta<2^{j_1}|\xi_1|} \frac{\sin \zeta}{\zeta}\diff \zeta\\
=&2\int_{1}^{2\pi} \frac{\sin \zeta}{\zeta}\diff \zeta+2\sum\limits_{n=1}^{m-1}\int_{2n\pi}^{2(n+1)\pi} \frac{\sin \zeta}{\zeta}\diff \zeta+2\int_{2m\pi}^{2^{j_1}|\xi_1|} \frac{\sin \zeta}{\zeta}\diff \zeta.
\end{align*}
The first and the third integrals are bounded by a constant independent with $j_1$. 
\begin{align*}
\left|2\int_{1}^{2\pi} \frac{\sin \zeta}{\zeta}\diff \zeta+2\int_{2m\pi}^{2^{j_1|\xi_1|}} \frac{\sin \zeta}{\zeta}\diff \zeta\right|\leq 2\int_1^{2\pi}1\diff \zeta+2\int_{2m\pi}^{2^{j_1}|\xi_1|}1\diff \zeta<8\pi.
\end{align*}
The integral in the middle is
\begin{align*}
0<\int_{2n\pi}^{2(n+1)\pi} \frac{\sin \zeta}{\zeta}\diff \zeta=&\int_{2n\pi}^{2(n+1)\pi} \frac{-\diff \cos \zeta}{\zeta}\\
=&\frac{\cos(2n\pi)}{2n\pi}-\frac{\cos(2(n+1)\pi)}{2(n+1)\pi}-\int_{2n\pi}^{2(n+1)\pi} \frac{\cos\zeta}{\zeta^2}\diff\zeta\\
\leq&\frac{1}{2n(n+1)\pi}+\int_{2n\pi}^{2(n+1)\pi} \frac{1}{\zeta^2}\diff\zeta\\
=&\frac{1}{n(n+1)\pi}.
\end{align*}
Then we take the summation from $1$ to $m-1$, we obtain
\begin{align*}
\left|\int_{1<|\zeta|<2^{j_1}|\xi_1|} \frac{\sin \zeta}{\zeta}\diff \zeta\right|\leq 8\pi+\frac1\pi. 
\end{align*}
Thus,
\[
\left|(-1)^\mu \prod\limits_{k=2}^{2\mu+1}\left(2^{j_k}\xi_k\right) I_0(0)\int_{\R} \chi_{\{|\zeta|<1\}}\frac{e^{i2^{j_1}\xi_1\zeta}}{\zeta} \diff \zeta\right|\lesssim 2^{j_2+\cdots+j_{2\mu+1}}.
\]
So 
\begin{equation}\label{Est3_17}
\begin{aligned}
\left|\int_{\R}\chi_{\{|\zeta|<1\}} [(1-e^{i2^{j_2}\xi_2\zeta})+2^{2j_2}\zeta^2e^{i2^{j_2}\xi_2\zeta}]\cdots[(1-e^{i2^{j_{2\mu+1}}\xi_{2\mu+1}\zeta})+2^{2j_{2\mu+1}}\zeta^2e^{i2^{j_{2\mu+1}}\xi_{2\mu+1}\zeta}]e^{i2^{j_1}\xi_1\zeta} |\zeta|^{-2\mu-1}\diff \zeta\right|\\
\lesssim 2^{2(j_2+\cdots+j_{2\mu+1})},
\end{aligned}
\end{equation}
and
\begin{align*}
&\Bigg|\frac1{2^{2j_1}\xi_1^2}\int_{\R} e^{i2^{j_1}\xi_1\zeta} \partial_{\zeta}^2\left\{ \chi_{\{|\zeta|<1\}} [(1-e^{i2^{j_2}\xi_2\zeta})+2^{2j_2}\zeta^2e^{i2^{j_2}\xi_2\zeta}]\cdots\right.\\
&\qquad\qquad\qquad\left.\cdot[(1-e^{i2^{j_{2\mu+1}}\xi_{2\mu+1}\zeta})+2^{2j_{2\mu+1}}\zeta^2e^{i2^{j_{2\mu+1}}\xi_{2\mu+1}\zeta}]  \zeta B_{\mu}(\zeta)\right\}\diff \zeta\Bigg|\\
&\lesssim  2^{2(j_2+\cdots+j_{2\mu+1})+2\max\{0, j_2, \cdots, j_{2\mu+1}\}-2j_1}.
\end{align*}
Then we estimate 
$|\partial_{\xi_1}^2(1-\partial_{\xi_2}^2)\cdots(1-\partial_{\xi_{2\mu+1}}^2)[\partial_1\Tb_1](2^{j_1}\xi_1,\cdots,2^{j_{2\mu+1}}\xi_{2\mu+1})|$. 
\begin{align*}
&\partial_{\xi_1}^2(1-\partial_{\xi_2}^2)\cdots(1-\partial_{\xi_{2\mu+1}}^2)[\partial_\ell\Tb_1](2^{j_1}\xi_1,\cdots,2^{j_{2\mu+1}}\xi_{2\mu+1})\\
=
&2\int_{\R}\chi_{\{|\zeta|>1\}}[(1-e^{i2^{j_2}\xi_2\zeta})+2^{2j_2}\zeta^2e^{i2^{j_2}\xi_2\zeta}]\cdots[(1-e^{i2^{j_{2\mu+1}}\xi_{2\mu+1}\zeta})+2^{2j_{2\mu+1}}\zeta^2e^{i2^{j_{2\mu+1}}\xi_{2\mu+1}\zeta}]
\\
&\qquad\qquad\qquad\cdot (i2^{2j_1}\zeta^3 e^{i2^{j_1}\xi_1\zeta}) A_{\mu}(\zeta)\diff \zeta
\\
&+2\int_{\R}\chi_{\{|\zeta|<1\}} [(1-e^{i2^{j_2}\xi_2\zeta})+2^{2j_2}\zeta^2e^{i2^{j_2}\xi_2\zeta}]\cdots[(1-e^{i2^{j_{2\mu+1}}\xi_{2\mu+1}\zeta})+2^{2j_{2\mu+1}}\zeta^2e^{i2^{j_{2\mu+1}}\xi_{2\mu+1}\zeta}]\\
&\qquad\qquad\qquad\cdot (i2^{2j_1}\zeta^3  e^{i2^{j_1}\xi_1\zeta}) B_{\mu}(\zeta)\diff \zeta
\\
=
&2\int_{\R}\chi_{\{|\zeta|>1\}}[(1-e^{i2^{j_2}\xi_2\zeta})+2^{2j_2}\zeta^2e^{i2^{j_2}\xi_2\zeta}]\cdots[(1-e^{i2^{j_{2\mu+1}}\xi_{2\mu+1}\zeta})+2^{2j_{2\mu+1}}\zeta^2e^{i2^{j_{2\mu+1}}\xi_{2\mu+1}\zeta}]
\\
&\qquad\qquad\qquad\cdot (i2^{2j_1} (\frac1{i2^{j_1}\xi_1}\partial_{\zeta})^4 e^{i2^{j_1}\xi_1\zeta}) \zeta^3  A_{\mu}(\zeta)\diff \zeta
\\
&+2\int_{\R}\chi_{\{|\zeta|<1\}} [(1-e^{i2^{j_2}\xi_2\zeta})+2^{2j_2}\zeta^2e^{i2^{j_2}\xi_2\zeta}]\cdots[(1-e^{i2^{j_{2\mu+1}}\xi_{2\mu+1}\zeta})+2^{2j_{2\mu+1}}\zeta^2e^{i2^{j_{2\mu+1}}\xi_{2\mu+1}\zeta}]
\\
&\qquad\qquad\qquad\cdot (i2^{2j_1} (\frac1{i2^{j_1}\xi_1}\partial_{\zeta})^4 e^{i2^{j_1}\xi_1\zeta}) \zeta^3 B_{\mu}(\zeta)\diff \zeta.
\end{align*}
For the first integral, after integrating by part four times, we have
\begin{align*}
&\Bigg|\int_{\R}\chi_{\{|\zeta|>1\}}[(1-e^{i2^{j_2}\xi_2\zeta})+2^{2j_2}\zeta^2e^{i2^{j_2}\xi_2\zeta}]\cdots[(1-e^{i2^{j_{2\mu+1}}\xi_{2\mu+1}\zeta})+2^{2j_{2\mu+1}}\zeta^2e^{i2^{j_{2\mu+1}}\xi_{2\mu+1}\zeta}]\\
&\qquad\qquad\qquad\cdot (i2^{2j_1} (\frac1{i2^{j_1}\xi_1}\partial_{\zeta})^4 e^{i2^{j_1}\xi_1\zeta}) \zeta^3A_{\mu}(\zeta)\diff \zeta\Bigg|\\
\lesssim & 2^{2(j_2+\cdots+j_{2\mu+1})+4\max\{0, j_2, \cdots, j_{2\mu+1}\}-2j_1}.
\end{align*}
For the second integral, after integrating by part four times, by \eqref{Est3_17}, the most singular term is bounded by
\begin{align*}
&\bigg|\int_{\R}\chi_{\{|\zeta|<1\}} [(1-e^{i2^{j_2}\xi_2\zeta})+2^{2j_2}\zeta^2e^{i2^{j_2}\xi_2\zeta}]\cdots[(1-e^{i2^{j_{2\mu+1}}\xi_{2\mu+1}\zeta})+2^{2j_{2\mu+1}}\zeta^2e^{i2^{j_{2\mu+1}}\xi_{2\mu+1}\zeta}]\\
&\qquad(i2^{-2j_1}\xi_1^{-4} e^{i2^{j_1}\xi_1\zeta}) \frac{1}{\zeta^{2\mu+1}}\diff \zeta\bigg|\\
\lesssim~& 2^{2(j_2+\cdots+j_{2\mu+1})-2j_1}+2^{(j_2+\cdots+j_{2\mu+1})-2j_1}.
\end{align*}
In total, the second integral is bounded by 
\begin{align*}
&\Bigg| \int_{\R}\chi_{\{|\zeta|<1\}} [(1-e^{i2^{j_2}\xi_2\zeta})+2^{2j_2}\zeta^2e^{i2^{j_2}\xi_2\zeta}]\cdots[(1-e^{i2^{j_{2\mu+1}}\xi_{2\mu+1}\zeta})+2^{2j_{2\mu+1}}\zeta^2e^{i2^{j_{2\mu+1}}\xi_{2\mu+1}\zeta}]\\
&\qquad\qquad\qquad\cdot (i2^{2j_1} (\frac1{i2^{j_1}\xi_1}\partial_{\zeta})^4 e^{i2^{j_1}\xi_1\zeta}) \zeta^3 B_{\mu}(\zeta)\diff \zeta\Bigg|\\
\lesssim & (2^{(j_2+\cdots+j_{2\mu+1})}+2^{2(j_2+\cdots+j_{2\mu+1})})2^{4\max\{0, j_2, \cdots, j_{2\mu+1}\}-2j_1}.
\end{align*}
So we proved \eqref{Tmu_est2}.

{\bf Proof of \eqref{Tmu_est3}.}
For $l=2,3,\cdots, 2\mu+1$,
\begin{align}\nonumber
&(1-\partial_{\xi_1}^2)(1-\partial_{\xi_2}^2)\cdots(1-\partial_{\xi_{2\mu+1}}^2)[\partial_l\Tb_\mu](2^{j_1}\xi_1,\cdots,2^{j_{2\mu+1}}\xi_{2\mu+1})\\\nonumber
=
&2\int_{\R}\chi_{\{|\zeta|>1\}}\prod\limits_{k\neq l}[(1-e^{i2^{j_k}\xi_k\zeta})+2^{2j_k}\zeta^2e^{i2^{j_k}\xi_k\zeta}](1+2^{2j_l}\zeta^2)(-i\zeta e^{i2^{j_l}\xi_l\zeta}) A_{\mu}(\zeta)\diff \zeta\\\label{eqnA8}
&+2\int_{\R}\chi_{\{|\zeta|<1\}}\prod\limits_{k\neq l}[(1-e^{i2^{j_k}\xi_k\zeta})+2^{2j_k}\zeta^2e^{i2^{j_k}\xi_k\zeta}](1+2^{2j_l}\zeta^2)(-i\zeta e^{i2^{j_l}\xi_l\zeta}) B_{\mu}(\zeta)\diff \zeta.
\end{align}
We can estimate the first integral by considering the following two parts. One part is
\begin{align*}
&\left|\int_{\R}\chi_{\{|\zeta|>1\}}(1-e^{i2^{j_1}\xi_1\zeta})\prod\limits_{k\neq 1, l}[(1-e^{i2^{j_k}\xi_k\zeta})+2^{2j_k}\zeta^2e^{i2^{j_k}\xi_k\zeta}](1+2^{2j_l}\zeta^2)(-i\zeta e^{i2^{j_l}\xi_l\zeta})A_{\mu}(\zeta)\diff \zeta\right|\\
\lesssim~&\prod\limits_{k\neq 1,l}[(1+2^{j_k})2^{j_k}](1+2^{j_l})^22^{j_1}(1+2^{j_1})^{-1}.
\end{align*}
The other part, after integral-by-part,
\begin{align*}
&\left|\int_{\R}\chi_{\{|\zeta|>1\}}(2^{2j_1}\zeta^2e^{i2^{j_1}\xi_1\zeta})\prod\limits_{k\neq 1, l}[(1-e^{i2^{j_k}\xi_k\zeta})+2^{2j_k}\zeta^2e^{i2^{j_k}\xi_k\zeta}](1+2^{2j_l}\zeta^2)(-i\zeta e^{i2^{j_l}\xi_l\zeta}) A_{\mu}(\zeta)\diff \zeta\right|\\
=~&\left|\int_{\R}\chi_{\{|\zeta|>1\}}2^{2j_1}\zeta^2(\frac1{i2^{j_1}\xi_1}\partial_{\zeta})^2e^{i2^{j_1}\xi_1\zeta}\prod\limits_{k\neq 1, l}[(1-e^{i2^{j_k}\xi_k\zeta})+2^{2j_k}\zeta^2e^{i2^{j_k}\xi_k\zeta}](1+2^{2j_l}\zeta^2)(-i\zeta e^{i2^{j_l}\xi_l\zeta})A_{\mu}(\zeta)\diff \zeta\right|\\
=~&\left|\int_{\R}\chi_{\{|\zeta|>1\}}\xi_1^{-2}e^{i2^{j_1}\xi_1\zeta}\partial_{\zeta}^2\left\{\prod\limits_{k\neq 1, l}[(1-e^{i2^{j_k}\xi_k\zeta})+2^{2j_k}\zeta^2e^{i2^{j_k}\xi_k\zeta}](1+2^{2j_l}\zeta^2)(-i\zeta e^{i2^{j_l}\xi_l\zeta})A_{\mu}(\zeta)\zeta^2\right\}\diff \zeta\right|\\
\lesssim~&\prod\limits_{k\neq 1,l}[2^{j_k}(1+2^{j_k})](1+2^{j_l})^2(1+2^{j_2})^22^{j_1}(1+2^{j_1})^{-1}.
\end{align*}
The second integral of \eqref{eqnA8} satisfies
\begin{align*}
&\left|\int_{\R}\chi_{\{|\zeta|<1\}}\prod\limits_{k\neq l}[(1-e^{i2^{j_k}\xi_k\zeta})+2^{2j_k}\zeta^2e^{i2^{j_k}\xi_k\zeta}](1+2^{2j_l}\zeta^2)(-i\zeta e^{i2^{j_l}\xi_l\zeta}) B_{\mu}(\zeta)\diff \zeta\right|\\
=~&\left|\int_{\R}\chi_{\{|\zeta|<1\}}\prod\limits_{k\neq l}[(1-e^{i2^{j_k}\xi_k\zeta})+2^{2j_k}\zeta^2e^{i2^{j_k}\xi_k\zeta}](1+2^{2j_l}\zeta^2)(-i\zeta e^{i2^{j_l}\xi_l\zeta})|\zeta|^{-2\mu} [|\zeta|^{2\mu} B_{\mu}(\zeta)]\diff \zeta\right|\\
\lesssim~& \prod\limits_{k\neq1,l}[2^{j_k}(1+2^{j_k})] (1+2^{j_l})^22^{j_1}(1+2^{j_1})^{-1}.
\end{align*}
Combining the above two parts, we obtain \eqref{Tmu_est3}.

{\bf Proof of \eqref{Tmu_est5}.}
Denote the difference as $\Tf_\mu(\eta_1,\eta_2,\cdots, \eta_{2\mu+1})$ in short,
\begin{align*}
\Tf_\mu:=&\Tb_{\mu}(\eta_1,\eta_2,\cdots, \eta_{2\mu+1})-\Tb_{\mu}(\eta_1+\cdots+\eta_{2\mu+1},\eta_2,\cdots, \eta_{2\mu+1})\\
=~&2\int_{\R}\chi_{\{|\zeta|>1\}}e^{i\eta_1\zeta}(1-e^{i(\eta_2+\cdots+\eta_{2\mu+1})\zeta})\prod_{k=2}^{2\mu+1}(1-e^{i\eta_k\zeta}) A_\mu(\zeta)\diff \zeta\\
&+2\int_{\R}\chi_{\{|\zeta|<1\}}e^{i\eta_1\zeta}(1-e^{i(\eta_2+\cdots+\eta_{2\mu+1})\zeta}) \prod_{k=2}^{2\mu+1}(1-e^{i\eta_k\zeta})B_\mu(\zeta)
\diff\zeta.
\end{align*}
We will estimate $|(1-\partial_{\xi_1}^2)\cdots(1-\partial_{\xi_{2\mu+1}}^2)\Tf_\mu(2^{j_1}\xi_1,\cdots, 2^{j_{2\mu+1}}\xi_{2\mu+1})|$.
Notice that
\begin{align*}
&(1-\partial_{\xi_1}^2)\cdots(1-\partial_{\xi_{2\mu+1}}^2)[e^{i2^{j_1}\xi_1\zeta}(1-e^{i(2^{j_2}\xi_2+\cdots+2^{j_{2\mu+1}}\xi_{2\mu+1})\zeta})\prod_{k=2}^{2\mu+1}(1-e^{i2^{j_k}\xi_k\zeta})]\\
=~&(1+2^{2j_1}\zeta^2)e^{i2^{j_1}\xi_1\zeta}\prod_{k=2}^{2\mu+1}[(1-e^{i2^{j_k}\xi_k\zeta})-2^{2j_k}\zeta^2e^{i2^{j_k}\xi_k\zeta}]\\
&-(1+2^{2j_1}\zeta^2)e^{i2^{j_1}\xi_1\zeta}\prod_{k=2}^{2\mu+1}e^{i2^{j_k}\xi_k\zeta}[1-e^{i2^{j_k}\xi_k\zeta}+2^{2j_k}\zeta^2-2^{2j_k+2}\zeta^2e^{i2^{j_k}\xi_k\zeta}]\\
=:~&(1+2^{2j_1}\zeta^2)e^{i2^{j_1}\xi_1\zeta}M(\xib, \zeta).
\end{align*}
With compact support assumption of $\xib$, the function $M(\xib, \zeta)$ satisfies
\begin{align*}
\left|\frac{\partial_\zeta M(\xib, \zeta)}{|\zeta|^{2\mu}}\right|\lesssim \prod\limits_{k=2}^{2\mu+1}2^{j_k}(1+2^{j_k}|\zeta|)(1+2^{j_2}|\zeta|),\qquad
\left|\frac{\partial_\zeta^3[\zeta^2 M(\xib, \zeta)]}{|\zeta|^{2\mu}}\right|\lesssim \prod\limits_{k=2}^{2\mu+1}2^{j_k}(1+2^{j_k}|\zeta|)(1+2^{j_2}|\zeta|)^3.
\end{align*}
So we have
\begin{align*}
&\left|(1-\partial_{\xi_1}^2)\cdots(1-\partial_{\xi_{2\mu+1}}^2)\Tf_\mu(2^{j_1}\xi_1,\cdots, 2^{j_{2\mu+1}}\xi_{2\mu+1})\right|\\
\leq ~&2\left|\int_{\R}\chi_{\{|\zeta|>1\}}(1+2^{2j_1}\zeta^2)e^{i2^{j_1}\xi_1\zeta}M(\xib, \zeta) A_\mu(\zeta)\diff \zeta\right|\\
&+2\left|\int_{\R}\chi_{\{|\zeta|<1\}}(1+2^{2j_1}\zeta^2)e^{i2^{j_1}\xi_1\zeta}M(\xib, \zeta)B_\mu(\zeta)
\diff\zeta\right|
\\
\leq ~&2\left|\int_{\R}\chi_{\{|\zeta|>1\}}(2^{j_1}\xi_1)^{-1}\partial_{\zeta}e^{i2^{j_1}\xi_1\zeta}M(\xib, \zeta)  A_\mu(\zeta)\diff \zeta\right|
\\
&+2\left|\int_{\R}\chi_{\{|\zeta|<1\}}(2^{j_1}\xi_1)^{-1}\partial_{\zeta}e^{i2^{j_1}\xi_1\zeta}M(\xib, \zeta)B_\mu(\zeta)
\diff\zeta\right|
\\
&2\left|\int_{\R}\chi_{\{|\zeta|>1\}}2^{2j_1}\zeta^2[(2^{j_1}\xi_1)^{-1}\partial_{\zeta}]^3e^{i2^{j_1}\xi_1\zeta}M(\xib, \zeta) A_\mu(\zeta)\diff \zeta\right|\\
&+2\left|\int_{\R}\chi_{\{|\zeta|<1\}}2^{2j_1}\zeta^2[(2^{j_1}\xi_1)^{-1}\partial_{\zeta}]^3e^{i2^{j_1}\xi_1\zeta}M(\xib, \zeta)B_\mu(\zeta)
\diff\zeta\right|\\
\lesssim~& 2^{-j_1}(1+2^{j_2})^3\prod\limits_{k=2}^{2\mu+1}2^{j_k}(1+2^{j_k}).
\end{align*}
Thus we proved \eqref{Tmu_est5}.
\end{proof}

\begin{proposition}\label{propA2}
If $j_1\geq j_2\geq j_3$ and $j_1-j_3>1$.
\begin{equation}\label{eqnA7}
\left\|\frac{1}{p'(\eta_1)-p'(\eta_3)}\psi_{j_1}(\eta_1)\psi_{j_2}(\eta_2)\psi_{j_3}(\eta_3)\right\|_{S^\infty}\lesssim \frac{(1+2^{6j_1})(1+2^{3j_3})}{2^{2j_1}+2^{6j_1}}\lesssim\frac{(1+2^{2j_1})(1+2^{3j_3})}{2^{2j_1}} .
\end{equation}
\begin{equation}\label{eqnA8}
\left\|[p''(\eta_1)+p''(\eta_3)]\psi_{j_1}(\eta_1)\psi_{j_2}(\eta_2)\psi_{j_3}(\eta_3)\right\|_{S^\infty}\lesssim 2^{j_1}(1+2^{j_1})^{-1} .
\end{equation}
\end{proposition}
\begin{proof}
Since
\begin{align*}
\frac{1}{p'(\eta_1)-p'(\eta_3)}=~&\frac{1}{(1+\eta_1^2)^{-3/2}-(1+\eta_3^2)^{-3/2}}\\
=~&\frac{(1+\eta_1^2)^{3/2}(1+\eta_3^2)^{3/2}}{(1+\eta_3^2)^{3/2}-(1+\eta_1^2)^{3/2}}\\
=~&\frac{(1+\eta_1^2)^{3/2}(1+\eta_3^2)^{3/2}[(1+\eta_3^2)^{3/2}+(1+\eta_1^2)^{3/2}]}{(1+\eta_3^2)^{3}-(1+\eta_1^2)^{3}}\\
=~&\frac{(1+\eta_1^2)^{3/2}(1+\eta_3^2)^{3/2}[(1+\eta_3^2)^{3/2}+(1+\eta_1^2)^{3/2}]}{(3\eta_3^2+3\eta_3^4+\eta_3^6)-(3\eta_1^2+3\eta_1^4+\eta_1^6)},
\end{align*}
\begin{align*}
p''(\eta_1)+p''(\eta_3)=~&-3\eta_1(1+\eta_1^2)^{-5/2}+3\eta_3(1+\eta_3^2)^{-5/2}\\
=~&\frac{-3\eta_1(1+\eta_3^2)^{5/2}+3\eta_3(1+\eta_1^2)^{5/2}}{(1+\eta_1^2)^{5/2}(1+\eta_3^2)^{5/2}},
\end{align*}
by direct calculation, we can prove \eqref{eqnA7} and \eqref{eqnA8}.

\end{proof}

\begin{proposition}\label{PropA3}
$j_1, j_2, j_3\in\{j-1, j, j+1\}$. When $(\iota_1,\iota_2)=(+, -)$ or $(-,+)$,
\begin{equation}\label{eqnB7}
\left\|\frac{p'(\eta_3)-p'(\eta_1+\eta_2+\eta_3)}{p'(\eta_1)-p'(\eta_2)}\psi_{j_1}(\eta_1)\psi_{j_2}(\eta_2)\psi_{j_3}(\eta_3)\upsilon_{\iota_1}(\eta_1)\upsilon_{\iota_2}(\eta_2)\upsilon_{\iota_3}(\eta_3)\right\|_{S^\infty}\lesssim 1,
\end{equation}
\begin{equation}\label{eqnB8}
\left\|\frac{p''(\eta_1)+p''(\eta_2)}{p'(\eta_1)-p'(\eta_2)}\psi_{j_1}(\eta_1)\psi_{j_2}(\eta_2)\psi_{j_3}(\eta_3)\upsilon_{\iota_1}(\eta_1)\upsilon_{\iota_2}(\eta_2)\upsilon_{\iota_3}(\eta_3)\right\|_{S^\infty}\lesssim 1.
\end{equation}
By symmetry, when $(\iota_2,\iota_3)=(+, -)$ or $(-,+)$,
\begin{equation}\label{eqnB9}
\left\|\frac{p'(\eta_1)-p'(\eta_1+\eta_2+\eta_3)}{p'(\eta_2)-p'(\eta_3)}\psi_{j_1}(\eta_1)\psi_{j_2}(\eta_2)\psi_{j_3}(\eta_3)\upsilon_{\iota_1}(\eta_1)\upsilon_{\iota_2}(\eta_2)\upsilon_{\iota_3}(\eta_3)\right\|_{S^\infty}\lesssim 1,
\end{equation}
\begin{equation}\label{eqnB10}
\left\|\frac{p''(\eta_2)+p''(\eta_3)}{p'(\eta_2)-p'(\eta_3)}\psi_{j_1}(\eta_1)\psi_{j_2}(\eta_2)\psi_{j_3}(\eta_3)\upsilon_{\iota_1}(\eta_1)\upsilon_{\iota_2}(\eta_2)\upsilon_{\iota_3}(\eta_3)\right\|_{S^\infty}\lesssim 1.
\end{equation}
When $(\iota_1,\iota_3)=(+, -)$ or $(-,+)$,
\begin{equation}\label{eqnB11}
\left\|\frac{p''(\eta_1)+p''(\eta_3)}{p'(\eta_1)-p'(\eta_3)}\psi_{j_1}(\eta_1)\psi_{j_2}(\eta_2)\psi_{j_3}(\eta_3)\upsilon_{\iota_1}(\eta_1)\upsilon_{\iota_2}(\eta_2)\upsilon_{\iota_3}(\eta_3)\right\|_{S^\infty}\lesssim 1.
\end{equation}

\end{proposition}
\begin{proof} By direct calculation,
\begin{align*}
&\frac{p'(\eta_3)-p'(\eta_1+\eta_2+\eta_3)}{p'(\eta_1)-p'(\eta_2)}\\
=&\frac{\eta_1+\eta_2+2\eta_3}{\eta_2-\eta_1}\cdot\frac{[(1+(\eta_1+\eta_2+\eta_3)^2)^2+(1+(\eta_1+\eta_2+\eta_3)^2)(1+\eta_3^2)+(1+\eta_3^2)^2]}{[(1+\eta_2^2)^{2}+(1+\eta_2^2)(1+\eta_1^2)+(1+\eta_1^2)^{2}]}\\
&\cdot\frac{(1+\eta_1^2)^{3/2}(1+\eta_2^2)^{3/2}[(1+\eta_2^2)^{3/2}+(1+\eta_1^2)^{3/2}]}{(1+(\eta_1+\eta_2+\eta_3)^2)^{3/2}(1+\eta_3^2)^{3/2}[(1+(\eta_1+\eta_2+\eta_3)^2)^{3/2}+(1+\eta_3^2)^{3/2}]}.
\end{align*}
Since 
\begin{align*}
\left\|\frac{\eta_1+\eta_2+2\eta_3}{\eta_2-\eta_1}\psi_{j_1}(\eta_1)\psi_{j_2}(\eta_2)\psi_{j_3}(\eta_3)\psi_{\leq k_1}(\eta_1+\eta_2)\psi_{k_2}(\eta_1-\eta_2)\right\|_{S^\infty}
\lesssim~ 1,
\end{align*}
and
\begin{align*}
&\left\|\frac{[(1+(\eta_1+\eta_2+\eta_3)^2)^2+(1+(\eta_1+\eta_2+\eta_3)^2)(1+\eta_3^2)+(1+\eta_3^2)^2]}{[(1+\eta_2^2)^{2}+(1+\eta_2^2)(1+\eta_1^2)+(1+\eta_1^2)^{2}]}\right.\\
&\quad\cdot\frac{(1+\eta_1^2)^{3/2}(1+\eta_2^2)^{3/2}[(1+\eta_2^2)^{3/2}+(1+\eta_1^2)^{3/2}]}{(1+(\eta_1+\eta_2+\eta_3)^2)^{3/2}(1+\eta_3^2)^{3/2}[(1+(\eta_1+\eta_2+\eta_3)^2)^{3/2}+(1+\eta_3^2)^{3/2}]}\left.\cdot\tilde\psi_{j_1}(\eta_1)\tilde\psi_{j_2}(\eta_2)\tilde\psi_{j_3}(\eta_3)\right\|_{S^\infty}
\lesssim~1,
\end{align*}
by algebraic property of $S^\infty$ norm, we achieve \eqref{eqnB7}.

After simplification,
\begin{align*}
&\frac{p''(\eta_1)+p''(\eta_2)}{p'(\eta_1)-p'(\eta_2)}\\
=~&3\left(\frac{(1+\eta_1^2)^{5/2}[(1+\eta_1^2)^{5/2}+(1+\eta_2^2)^{5/2}]}{(\eta_2-\eta_1)[(1+\eta_2^2)^{2}+(1+\eta_2^2)(1+\eta_1^2)+(1+\eta_1^2)^{2}]}\right.\\
&\quad \left.+\frac{\eta_1[(1+\eta_2^2)^{4}+(1+\eta_2^2)^3(1+\eta_1^2)+(1+\eta_2^2)^2(1+\eta_1^2)^2+(1+\eta_2^2)(1+\eta_1^2)^3+(1+\eta_1^2)^{4}]}{[(1+\eta_2^2)^{2}+(1+\eta_2^2)(1+\eta_1^2)+(1+\eta_1^2)^{2}]}\right)\\
&\quad\cdot\frac{(1+\eta_1^2)^{3/2}+(1+\eta_2^2)^{3/2}}{(1+\eta_1^2)(1+\eta_2^2)[(1+\eta_1^2)^{5/2}+(1+\eta_2^2)^{5/2}]}.
\end{align*}
Since 
\[
\left\|\frac{(1+\eta_1^2)^{3/2}+(1+\eta_2^2)^{3/2}}{[(1+\eta_1^2)^{5/2}+(1+\eta_2^2)^{5/2}]}\psi_{j_1}(\eta_1)\psi_{j_2}(\eta_2)\psi_{j_3}(\eta_3)\psi_{\leq k_1}(\eta_1+\eta_2)\psi_{k_2}(\eta_1-\eta_2)\right\|_{S^\infty}\lesssim 1.
\]
\begin{align*}
\bigg\|\frac{\eta_1[(1+\eta_2^2)^{4}+(1+\eta_2^2)^3(1+\eta_1^2)+(1+\eta_2^2)^2(1+\eta_1^2)^2+(1+\eta_2^2)(1+\eta_1^2)^3+(1+\eta_1^2)^{4}]}{(1+\eta_1^2)(1+\eta_2^2)[(1+\eta_2^2)^{2}+(1+\eta_2^2)(1+\eta_1^2)+(1+\eta_1^2)^{2}]}\\
\psi_{j_1}(\eta_1)\psi_{j_2}(\eta_2)\psi_{j_3}(\eta_3)\psi_{\leq k_1}(\eta_1+\eta_2)\psi_{k_2}(\eta_1-\eta_2)\bigg\|_{S^\infty}\lesssim 1.
\end{align*}
\begin{align*}
\bigg\|\frac{(1+\eta_1^2)^{5/2}[(1+\eta_1^2)^{5/2}+(1+\eta_2^2)^{5/2}]}{(\eta_2-\eta_1)[(1+\eta_2^2)^{2}+(1+\eta_2^2)(1+\eta_1^2)+(1+\eta_1^2)^{2}]}\cdot\frac{(1+\eta_1^2)^{3/2}+(1+\eta_2^2)^{3/2}}{(1+\eta_1^2)(1+\eta_2^2)[(1+\eta_1^2)^{5/2}+(1+\eta_2^2)^{5/2}]}\\
\cdot\psi_{j_1}(\eta_1)\psi_{j_2}(\eta_2)\psi_{j_3}(\eta_3)\psi_{\leq k_1}(\eta_1+\eta_2)\psi_{k_2}(\eta_1-\eta_2)\bigg\|_{S^\infty}\lesssim 1.
\end{align*}
Then using the algebraic property of $S^\infty$ norm, we obtain \eqref{eqnB8}.
The estimates \eqref{eqnB9} to \eqref{eqnB11} can be obtained in a similar way.

\end{proof}

\begin{proposition}\label{PropB3}
$j_1, j_2, j_3\in\{j-1, j, j+1\}$. When $(\iota_1,\iota_2,\iota_3)=(+, +, +)$ or $(-, -, -)$,
\begin{equation}\label{eqnB11}
\left\|\frac{1}{p(\eta_1)+p(\eta_2)+p(\eta_3)-p(\eta_1+\eta_2+\eta_3)}\psi_{j_1}(\eta_1)\psi_{j_2}(\eta_2)\psi_{j_3}(\eta_3)\upsilon_{\iota_1}(\eta_1)\upsilon_{\iota_2}(\eta_2)\upsilon_{\iota_3}(\eta_3)\right\|_{S^\infty}\lesssim 1+2^{-3j}.
\end{equation}
\end{proposition}
\begin{proof}The proof is based on direct calculation. Notice that
\begin{align*}
&\frac{\eta_1}{\sqrt{1+\eta_1^2}}+\frac{\eta_2}{\sqrt{1+\eta_2^2}}+\frac{\eta_3}{\sqrt{1+\eta_3^2}}-\frac{\eta_1+\eta_2+\eta_3}{\sqrt{1+(\eta_1+\eta_2+\eta_3)^2}}\\
=~&\Big[\eta_1\sqrt{1+\eta_2^2}\sqrt{1+\eta_3^2}\big(\sqrt{1+(\eta_1+\eta_2+\eta_3)^2}-\sqrt{1+\eta_1^2}\big)\\
&+\eta_2\sqrt{1+\eta_1^2}\sqrt{1+\eta_3^2}\big(\sqrt{1+(\eta_1+\eta_2+\eta_3)^2}-\sqrt{1+\eta_2^2}\big)\\
&+\eta_3\sqrt{1+\eta_1^2}\sqrt{1+\eta_2^2}\big(\sqrt{1+(\eta_1+\eta_2+\eta_3)^2}-\sqrt{1+\eta_3^2}\big)\Big]\\
&\cdot\bigg[\sqrt{1+\eta_1^2}\sqrt{1+\eta_2^2}\sqrt{1+\eta_3^2}\sqrt{1+(\eta_1+\eta_2+\eta_3)^2}\bigg]^{-1}.
\end{align*}
By algebraic property of $S^\infty$-norm, we have
\begin{align*}
&\left\|\frac{1}{p(\eta_1)+p(\eta_2)+p(\eta_3)-p(\eta_1+\eta_2+\eta_3)}\psi_{j_1}(\eta_1)\psi_{j_2}(\eta_2)\psi_{j_3}(\eta_3)\upsilon_{\iota_1}(\eta_1)\upsilon_{\iota_2}(\eta_2)\upsilon_{\iota_3}(\eta_3)\right\|_{S^\infty}\\
\lesssim~& \frac{1+2^{4j}}{(1+2^{2j})2^{3j}}\lesssim 1+2^{-3j}.
\end{align*}

\end{proof}

\end{document}